\documentclass{l4dc2021} 

\title[Safe Learning]{Safely Learning Dynamical Systems}
\usepackage{times}
\usepackage{mathtools}
\usepackage{cleveref}

\DeclareGraphicsExtensions{.pdf,.eps}

\ifdefined\ISSUBMISSION
\newcommand{\myproof}[1]{}
\newcommand{\myproofof}[2]{}

\else
\newcommand{\myproof}[1]{\begin{proof}#1\end{proof}}
\newcommand{\myproofof}[2]{\begin{proof}[#1]#2\end{proof}}

\fi

\author{\Name{Amir Ali Ahmadi} \Email{aaa@princeton.edu}\\
 \Name{Abraar Chaudhry} \Email{azc@princeton.edu}\\
 \addr Princeton University
 \AND
 \Name{Vikas Sindhwani} \Email{sindhwani@google.com}\\
 \Name{Stephen Tu} \Email{stephentu@google.com}\\
 \addr Robotics at Google, New York}

\makeatletter
\ifjmlrutilssubfloats
\@ifl@t@r\fmtversion{2019/08/22}%
{
\renewcommand*\@subfigurelabel[3]{#1\subfigurelabel{#2}}
\renewcommand*\@subfigref[1]{%
{%
\def\@subfigurelabel##1##2##3{\subfigurelabel{##2}}%
\ref{#1}%
}%
}
}%
{
\renewcommand*\@subfigurelabel[2]{#1\subfigurelabel{#2}}
\renewcommand*\@subfigref[1]{%
{%
\def\@subfigurelabel##1##2{\subfigurelabel{##2}}%
\ref{#1}%
}%
}
}
\makeatother

\newcommand{\R}{\mathbb{R}}
\newcommand{\Rn}{\R^n}
\newcommand{\nn}{{n \times n}}
\newcommand{\Rnn}{\R^\nn}
\newcommand{\Snn}{\mathbb{S}^\nn}
\newcommand{\Tr}{\mathrm{Tr}}
\newcommand{\spanof}[1]{\textnormal{span}(#1)}

\newcommand{\norm}[1]{\| #1 \|}

\newcommand{\trueA}{A_{\star}}
\newcommand{\truef}{f_{\star}}
\newcommand{\trueg}{g_{\star}}
\newcommand{\defn}{\mathrel{\mathop :}=}

\newcommand{\aaa}[1]{{\textcolor{purple}{#1}}}

\newcommand{\newstuff}[1]{{\color{magenta}{#1}}}
\newcommand{\newnewstuff}[1]{{\color{olive}{#1}}}

\newcommand{\versionii}[1]{{\color{magenta}{#1}}}
\newcommand{\transpose}{\mathsf{T}}

\renewcommand{\textcolor}[1]{{{}}}
\renewcommand{\newstuff}[1]{{{#1}}}
\renewcommand{\newnewstuff}[1]{{{#1}}}
\renewcommand{\versionii}[1]{{{#1}}}

\renewcommand{\Pr}{\mathbb{P}}

\begin{document}

\maketitle

\vspace{-10mm}

\begin{abstract}%
A fundamental challenge in learning an unknown dynamical system
is to reduce model uncertainty by making measurements while maintaining safety. 
In this work, we formulate a mathematical definition of what
it means to safely learn a dynamical system 
by sequentially deciding where to initialize the next trajectory.
In our framework, the state of the system is required to stay within a safety
region for a horizon of $T$ time steps 
under the action of all dynamical systems that (i) belong to 
a given initial uncertainty set, and (ii)
are consistent with
the information gathered so far.

\newstuff{For our first set of results, we consider the setting of safely learning a linear dynamical system involving $n$ states.
For the case $T=1$, we present a linear programming-based
algorithm that either safely recovers the true dynamics from at most $n$ trajectories, or certifies that safe learning is impossible.
For $T=2$,
we give a semidefinite 
representation of the
set of safe initial conditions
and show that $\lceil n/2 \rceil$ trajectories generically suffice
for safe learning.
For $T = \infty$,
we provide semidefinite representable inner approximations
of the set of safe initial conditions and
show that one trajectory generically suffices for safe learning. \versionii{Finally, we extend a number of our results to the cases where the initial uncertainty set contains sparse, low-rank, or permutation matrices, or when the dynamical system involves a control input.}

Our second set of results concerns the problem of safely
learning a general class of nonlinear dynamical systems.
For the case $T=1$, we give a second-order cone programming
based representation of the set of safe initial conditions.
For $T=\infty$, 
we provide semidefinite representable
inner approximations to the set of safe initial conditions.
We show how one can safely collect trajectories and fit a polynomial model of the nonlinear dynamics that is consistent
with the initial uncertainty set and best agrees with the observations. \versionii{We also present extensions of some of our results to the cases where the measurements are noisy or the dynamical system involves disturbances.}
}

\end{abstract}

\begin{keywords} learning dynamical systems, safe learning, uncertainty quantification, robust optimization, conic optimization
\end{keywords}

\section{Problem Formulation \newnewstuff{and Outline of Contributions}}

\newnewstuff{In many applications such as robotics, autonomous systems, and safety-critical control, one needs to learn a model of a dynamical system by observing a small set of its trajectories in a safe manner. This model can serve as a tool for making predictions about unobserved trajectories of the system. It can also be used for accomplishing downstream control objectives. Often, an important challenge during the initial stages of learning is that deploying even a conservative learning strategy on a real world system, such as a robot, is fraught with risk. How should the robot be ``set loose" (i.e., initialized) in the real world so that our uncertainty about its dynamics is reduced, but with guarantees that the robot will remain safe (e.g., it does not exit a pre-specified region in state space)? How much more aggressive can our learning strategy get ``on the fly'' as uncertainty is reduced? This interplay between \emph{safety and uncertainty while learning dynamical systems} is the central theme of this paper. We propose a mathematical formulation that captures the essence of this interplay and study the optimization problems that arise from the formulation in several settings.}



\versionii{Before we present the mathematical framework of this paper, let us provide some conceptual intuition with \figureref{fig:picture}.
\begin{figure}
\figureconts
{fig:picture}
{\caption{A conceptual illustration of the safe learning problem.}}
{%
\subfigure[The safety region in blue; one safe and two potentially unsafe initialization points given uncertainty over dynamics][c]{%
\label{fig:picture1}
\includegraphics[width=.5\textwidth -.5em]{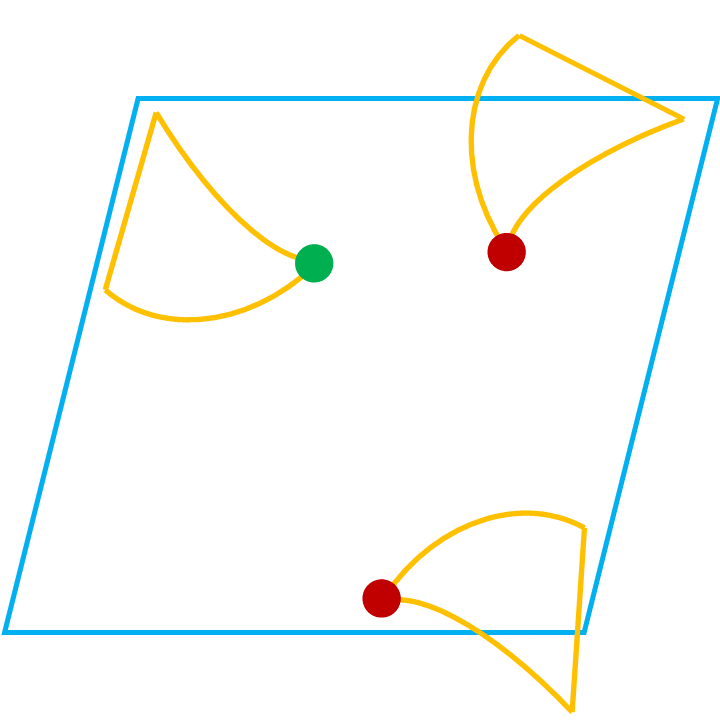}
} 
\subfigure[After safely observing one trajectory, uncertainty reduces and a previously unsafe initialization point becomes safe to query][c]{%
\label{fig:picture2}
\includegraphics[width=.5\textwidth -.5em]{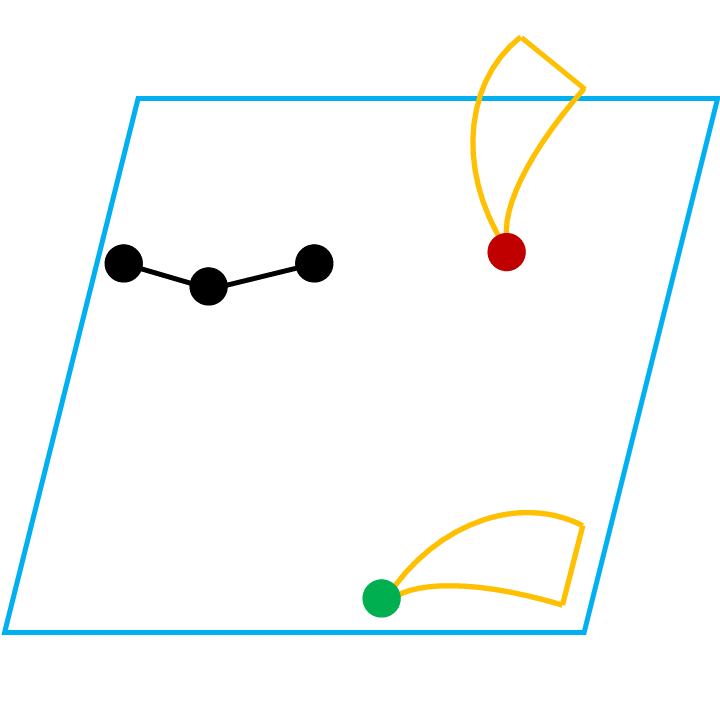}
}
}
\end{figure}
In this figure, the blue parallelogram represents the boundary of a safety region in which the trajectory of the dynamical system we wish to learn must stay throughout the learning process. In \figureref{fig:picture1}, we draw three points as examples of possible initializations of this dynamical system. If we choose one of these points and observe the resulting trajectory of the system, we can use our observations to learn more about the system parameters.
The safety constraint in this context means we must ensure the trajectory remains in the safety region up to a given horizon. If we truly knew nothing about the dynamics, then this task would be impossible since for any initialization, we could imagine some dynamics which would quickly take us out of the safety region.
If we suppose, however, that we have some initial information on the system, we can find sets representing the possible trajectories of all systems consistent with our information.
For our three candidate points, these sets are drawn in orange.
We should not initialize the system at either of the two red points since the system may take their associated trajectories outside the safety region.
The green point is ``safe'' to initialize from since its trajectory must stay in the safety region despite our uncertainty over the dynamics.

In \figureref{fig:picture2}, we imagine having safely observed a trajectory up to our given horizon initialized at the green point from \figureref{fig:picture1}. We can use the information from this trajectory (drawn in black) to learn more about the system and reduce our uncertainty over its potential trajectories.
This is represented by the orange sets being smaller than previously.
With this narrowed uncertainty, a previously unsafe initialization point now becomes safe to query as denoted by the green color in \figureref{fig:picture2}. Given this certification, we could then initialize the system at the green point, observe a new trajectory, and continue learning more about the system.

We now describe the safe learning problem more formally. }\newnewstuff{The central object of our mathematical framework is a discrete-time dynamical system
\begin{equation}\label{eq:dynamics}
x_{t+1} = \truef(x_t),
\end{equation}
\newnewstuff{where $\truef: \R^n \rightarrow \R^n$ is an \emph{unknown} map.
This could be either a naturally arising autonomous system, or a 
closed-loop control system with a fixed feedback policy. Our interest is in the} problem of safe data acquisition for estimating the unknown map $\truef$ from a collection of length-$T$ trajectories} $\{\phi_{\truef, T}(x_j)\}_{j=1}^{m}$,
where 
$\phi_{f,T}(x) \defn ( x, f(x), \dots, f^{(T)}(x) )$.

In our setting, we are given as input a set $S \subset \R^n$, called the \emph{safety region}, in which the state should remain throughout the learning process. 
\newnewstuff{We say that a\versionii{n initial} state $x$ is \emph{$T$-step safe} under a map $f: \R^n \rightarrow \R^n$ if $f^{(i)}(x) \in S$ for all $i=0, \ldots, T$. We define $S^T(f) \subseteq S$ to be the set of states that are $T$-step safe under $f$.}
In order to safely learn $\truef$, we require that measurements are made only at points in $S^T(\truef)$. 
Obviously, if we make no assumptions about $\truef$, this task is impossible. We assume, therefore, that the map $\truef$ belongs to a set of dynamics $U_0$, which we call the \emph{initial uncertainty set}. As experience is gathered, the uncertainty over $\truef$ decreases. 
Let us denote the uncertainty set after we have observed $k$ trajectories $\{ \phi_{\truef, T}(x_j)\}_{j=1}^{k}$ by,
\[U_k \defn \{f\in U_0 \mid \phi_{f,T}(x_j) = \phi_{\truef,T}(x_j) \:, j=1,\dots,k \}.\]
\newnewstuff{Observe that $U_{k+1} \subseteq U_k$ for all $k$.}
For a nonnegative integer $k$, define
\[S^T_k \defn \bigcap_{f \in U_k} S^T(f) \:, \]
the set of points that are $T$-step safe under all dynamics consistent \newnewstuff{with the initial uncertainty set} and the data after observing $k$ trajectories.
\newnewstuff{We refer to the set $S^T_k$ as the \emph{$T$-step safe set} (the dependence on $k$ is implicit).
Note that $S^T_k \subseteq S^T_{k+1}$ for all $k$.}
\newnewstuff{A primary goal of this paper is to characterize the sets $S^T_k$ as feasible regions of tractable optimization problems.
In certain settings where an exact tractable characterization is not possible, our goal would be to find tractable inner approximations of these sets. For robustness reasons, we would like these inner approximations to be full-dimensional so that safe queries to the system can be made while tolerating perturbations which may arise during implementation. 
}

\newnewstuff{A secondary goal of this paper is to provide algorithms for what we define as the \emph{$T$-step safe learning problem}.}
Fix a scalar $\bar{\varepsilon}>0$ \newnewstuff{and a norm $||.||$ on $\mathbb{R}^n$}.
Given a safety region $S\subset~\R^n$ and an initial uncertainty set $U_0$, the $T$-step safe learning problem (up to accuracy $\bar{\varepsilon}$ and with respect to norm $||.||$) is
to sequentially choose vectors $x_1,\dots,x_m$,
for some nonnegative integer $m$, such that:
\begin{enumerate}
    \item \textbf{(Safety)} for each $k=1, \dots, m$, $x_{k} \in S^T_{k-1}$,
    \item \textbf{(Learning)} $\sup_{f \in U_m, x \in S^T(\truef)} \norm{f(x) - \truef(x)} \leq \bar{\varepsilon}$.
\end{enumerate}
\newnewstuff{
%
If for a given $T$, no such sequence of vectors
$x_1, \ldots, x_m$ exists (for any $m$), we say that
$T$-step safe learning is impossible.
Note that if $T$-step safe learning is possible, then
$T'$-step safe learning is also possible for any 
$T' < T$.
%
%
Moreover, since the highest rate of safe information assimilation is achieved when $T=1$, to prove that safe learning is impossible for any $T$, it is necessary and sufficient to prove its impossibility for $T=1$.
}

\newnewstuff{
In many situations, the choice of the sequence of $\{x_1, \dots, x_m \}$ that achieves $T$-step safe learning may not be unique.
We further suppose that for a function $c: \Rn \mapsto \R$ that takes nonnegative values over $S$,} initializing the unknown system at a state $x\in S$ comes at a cost of $c(x)$. 
In such a setting, we are interested in safely learning \newnewstuff{the dynamical system at minimum total \versionii{initialization} cost.
Ideally, we wish to minimize $\sum_{k=1}^m c(x_k)$ 
over sequences $\{x_1, \dots, x_m \}$ that satisfy the safe learning conditions 1 and 2 above.
However, such an optimization problem cannot be solved
without knowing the action of the true dynamics
$\truef$ on the initialization points $\{x_k\}$ ahead of time. 
Hence, a natural online algorithm 
is to sequentially solve the following greedy optimization problem
}
\begin{equation}
\min_{x \in S^T_{k-1}} c(x) \:, \label{eq:greedy_safe}
\end{equation}
whose optimal solution gives the next cheapest $T$-step safe \newnewstuff{initialization} point $x_{k}$,
given information gathered before time $k$. \newnewstuff{A byproduct of our primary goal of characterizing the sets $S_k^T$ tractably is efficient algorithms
for solving the optimization problem \eqref{eq:greedy_safe}.}


\newnewstuff{
\paragraph{Contributions.}
In this paper, we derive tractable conic programs that exactly characterize or inner approximate $T$-step safe sets (for any $k$) for both linear systems and a general class of nonlinear systems in the extreme cases when $T=1$ and $T=\infty$.
For linear systems, we also address the case when $T=2$,
and provide algorithms for
solving the exact (i.e., $\bar{\varepsilon}=0$) 
$T$-step safe learning problem when $T=1,2,\infty$. Throughout the paper, we assume that the safety region $S$ is a polyhedron.

More specifically, for linear systems, we give an exact linear programming-based characterization of the one-step safe set when $U_0$ is a polytope, and an exact semidefinite programming-based characterization of the two-step safe set when $U_0$ is an ellipsoid.
Based on the former characterization,
%
we present a linear programming-based algorithm that either
learns the unknown dynamics by making at most $n$ one-step safe queries,
or certifies the impossibility of safe learning (for any $T$).
\versionii{This results demonstrates that safety requirements do not hinder the possibility to identify an $n$-dimensional system in $n$ steps.}
In the case of $T=2$, we show that under mild assumptions, $\lceil \frac{n}{2} \rceil$ trajectories (whose initializations are computed by semidefinite programming) suffice for safe learning.
Roughly speaking, these algorithms sequentially solve \eqref{eq:greedy_safe} and add appropriate safe perturbations to ensure that the remaining uncertainty $U_k$ is shrinking. When $T=\infty$,
under the assumption that $U_0$ is a compact subset of Schur-stable
matrices, we present a sum of squares hierarchy of semidefinite programs
that provide full-dimensional inner approximations of the infinite-step safe set.
%
%
Under mild assumptions, we show that a single trajectory randomly initialized
from our inner approximation suffices for safe learning.
\versionii{Finally, we extend some of our results for one-step safe learning to the case where more specialized side information is available, as well as to systems involving an affine control law. More specifically, we give an exact linear/semidefinite programming-based characterization of the set of one-step safe points of a linear system in the case where the governing matrix is known to be sparse, low-rank, or a permutation matrix. We also define a notion of controlled safe learning and show that this property can again be checked by linear programming.}








Turning to nonlinear systems, we consider the case when
the dynamics in \eqref{eq:dynamics} consists of a
linear term
plus a nonlinear function
with bounded growth. 
When $T=1$, we give an exact second-order cone programming-based representation of the safe set
when the uncertainty around the linear dynamics is represented
by a polyhedron. 
When $T=\infty$, we provide a hierarchy of semidefinite representable inner approximations to the infinite-step safe set.
Under the assumption that
the nonlinear function growth is relatively small compared
to the uncertainty around the linear part of the dynamics, 
we prove that our hierarchy provides a full-dimensional inner approximation.
By using our safe set representations, we show how one can safely collect trajectories to 
refine uncertainty regarding the linear term of the dynamics and fit a polynomial model of the nonlinear dynamics that is consistent
with the initial uncertainty set and best agrees with the observations.

\versionii{Finally, we show how some of our tractable conic programming formulations of the $T$-step safe sets can be extended to the cases when the dynamical system has disturbances or the measurements are noisy.
For $T=1$, both for linear and nonlinear systems, we can tolerate both bounded measurement noise and bounded disturbances and still given a exact characterization of the one-step safe set.
For $T=\infty$, both for linear and nonlinear systems, we can tolerate bounded measurement noise and still give tractable inner approximations of the infinite-step safe set.}



\paragraph{Outline.}
In \Cref{sec:relation_work}, we cover the relevant literature
around safe learning and control.
\Cref{sec:one-step},
\Cref{sec:linear:two-step}, and
\Cref{sec:linear} present our results and algorithms
for safely learning linear systems when $T=1,2,\infty$, respectively.
\Cref{sec:nonlinear} and \Cref{sec:inf_nonlinear} contain our results
for nonlinear systems when $T=1$ and $T=\infty$, respectively.
\versionii{In \sectionref{sec:side info}, we present our results on initial uncertainty sets containing sparse, low-rank, and permutation matrices. In Section~\ref{sec: controlled}, we define a notion of controlled safe learning and present an efficient algorithm for checking this property.}
Future directions for research are presented in \Cref{sec:conclusion}, \versionii{including extensions to continuous-time systems}.
\versionii{Omitted proofs of some technical statements can be found in \Cref{sec:appendix}.}
}




\section{Related Work}
\label{sec:relation_work}

The idea of using conic and robust optimization techniques for verifying various properties of a known dynamical system has been the focus of much
research in the control and optimization communities \citep{pablothesis,
   lasserre_book_2010, boyd_book_1994, PabloGregRekha_BOOK}. Our work borrows some
of these techniques to instead learn a dynamical system from data subject to certain safety constraints.
\versionii{Learning dynamical systems from data is an important problem in the field of system
identification; see, e.g.,~\cite{aastrom1971system,Antoulas,keesman2011system,brunton2019data}, and references therein.
This continues to be an active area of research, even for linear systems; see, e.g., the recent paper \cite{moitra23}.}

The problem of additionally accounting for safety constraints
during system identification has recently gained attention; see, e.g.,~\citet{brunke21survey} for an excellent
survey of this growing research field.
Here, we highlight a few of the key technical tools used:
Gaussian process models \citep{akametalu14safeGP,berkenkamp2015safe,berkenkamp17safeRL},
control barrier functions~\citep{cheng2019safe,taylor2020learning,luo21safenonlinear},
set invariance and uncertainty propagation~\citep{artstein2008invariance,gurriet2019invariance,koller19learning},
``safety critics'' in reinforcement learning~\citep{zhang2020rl,bharadhwaj21safetycritics},
and backup controllers~\citep{mannucci2018safeexploration,wabersich2021safety}.

We highlight two related works for safely learning linear
dynamical systems that, similar to our work, rely on optimization formulations to
ensure safety.
The first is the work of \citet{dean19safelqr},
which uses convex programming methods to approximate the solution to a
finite-time horizon linear-quadratic optimal control problem
with both model uncertainty and state/action constraints. 
\versionii{They provide convex optimization based sufficient conditions for existence of a control law that keeps a given initial condition safe up to a certain time horizon despite uncertainty and disturbance in the dynamics.
However, their approach requires the initial condition to be fixed and does not characterize the set of initial conditions for which there exists a control law ensuring safety (which would be the right generalization of our work to the controlled case; see Section~\ref{sec: controlled}).
Indeed, one can check that if both the control law and the initial condition are decision variables, then the formulation of \citet{dean19safelqr} is no longer convex. By contrast, our work focuses on autonomous systems and characterizes the set of initial conditions that remain safe despite uncertainty in the dynamics. Our characterizations are \emph{exact} for time horizons $T=1$ and $T=2$ and are in state space, while the sets provided by \citet{dean19safelqr} are in control space and in general conservative.
For $T=1$, our exact characterization of safe sets extends to cases including bounded disturbances in the dynamics and noise in the measurements (even for nonlinear systems). Our results for $T=1$ also extend to the problem of \emph{controlled safe learning} of linear control affine systems; see Section~\ref{sec: controlled}.}
In addition, the work in \citet{dean19safelqr} does not provide an implementable algorithm for handling the case when $T=\infty$ since the size of their
convex programs grow with the horizon length $T$.
\versionii{Also, that work does not use information on the fly for learning dynamics, which could be necessary for safe learning (see \sectionref{sec:dim increase}).}
We note that follow-up work \citep{chen2021rmpc,chen2022rmpc}
provides better heuristics for constructing inner approximations to the \versionii{set of safe controls}, but again these approximations are not exact.

The second related paper is that of \citet{lu17safeexploration}.
In this work, the authors address the problem of ``one-step safety''
and ``trajectory safety'' in a probabilistic framework. While similar sounding to our problem setup,
in their work
the initial condition is \versionii{again} \emph{fixed} and the question of either characterizing
or inner approximating the $T$-step safety sets $S_0^T$ is not addressed.
Furthermore, the proposed algorithm in the $T$-step setting requires a separate computation 
to check safety for every time step $t \in \{1, \dots, T\}$, and hence,
similar to \citet{dean19safelqr}, cannot
be implemented in the $T=\infty$ setting.
The algorithm also requires
checking safety for each coordinate of the state separately and relies on nonconvex optimization without optimality guarantees.

\newnewstuff{We would like to also highlight some papers on the topic of learning to stabilize dynamical systems, which has recently gained attention.
The work of \cite{Dean2020} studies linear systems with noise and shows how to learn a stabilizing controller efficiently, both with respect to the number of required samples and cost.
The work of \cite{Werner2023} shows that one can stabilize linear noiseless systems with less data than would be necessary to learn the system parameters exactly.
The work of \cite{guo} shows how to search for stabilizing controllers and associated Lyapunov functions for all continuous-time polynomial systems that are consistent with collected data.
In a follow-up paper, \cite{BISOFFI2022110537} use Petersen's lemma to further develop the method and prove a necessary and sufficient condition for data-driven stabilization of linear systems.
All these methods lead to stabilized systems wherein the state will remain bounded, however they do not consider explicit safety constraints.
In another follow-up paper, \cite{luppi} incorporate safety constraints into the framework of \cite{guo} and \cite{BISOFFI2022110537} and show how to find approximations of the continuous time analogue of $S^\infty_k$.
This work is concerned with stabilization of polynomial vector fields and is focused on the $T=\infty$ case. The approach is specific to data observation models which lead to ``matrix ellipsoidal'' uncertainty sets (see \cite{BISOFFI2022110537} for a definition) and does not consider the problems of trajectory initialization and its cost. The work is instead focused on the design of controllers which produce invariant subsets of the safety region that are defined as sublevel sets of polynomials. By contrast, the invariant sets that our work produces (for a different class of nonlinear dynamics) are semidefinite representable and hence can be optimized over efficiently. 
}

We end by noting that our work has some conceptual connections to the literature on experiment design \citep[see, e.g.,][]{pukelsheim06doe, decastro19experimentdesign}.
However, this literature typically does not consider dynamical systems or notions of safety.

A much shorter version of this work containing preliminary results on one-step and two-step safe learning has appeared in \citet{ahmadi2021shortpaper}.

\section{One-Step Safe Learning of Linear Systems}\label{sec:one-step}

In this section, we focus on characterizing one-step safe learning
for linear systems.
Here, the state evolves according to 
\begin{equation}\label{eq:linear dynamics}
x_{t+1} = \trueA x_t,
\end{equation}
where $\trueA$ is an unknown $n \times n$ matrix.
We assume we know that $\trueA$ belongs to a set $U_0 \subset \R^{n \times n}$ that represents our prior knowledge of $\trueA$.
In this section, we take $U_0$ to be a polyhedron; i.e., \begin{equation}\label{eq:U_0 polyhedron}
U_0 = \left\{ A \in \R^{n \times n} \mid  \Tr(V_j^\transpose  A) \leq v_j \quad j = 1, \dots ,s \right\}
\end{equation}
for some matrices $V_1,\dots,V_s \in \R^{n \times n}$ and scalars $v_1,\dots,v_s \in \R$. We also work with a polyhedral representation of the safety region $S$; i.e.,
\begin{equation}\label{eq:S polyhedron}
S = \left \{ x \in \R^n \mid  h_i^\transpose  x \leq b_i \quad i = 1, \dots ,r \right \}
\end{equation}
for some vectors $h_1,\dots,h_r \in \R^n$ and some scalars $b_1,\dots,b_r \in \R$.
We assume that \newnewstuff{initializing the system} at a point $x\in \R^n$ comes at a cost $c^\transpose  x$, for some given vector $c \in \Rn$.
In practice, initialization costs are nonnegative.
\newnewstuff{Since the set $S$ is often compact in applications, one can add a constant term to $c^\transpose  x$ to ensure this requirement without changing any of our optimization problems. We ignore this constant term in our formulations and examples. Our algorithms tractably extend to any semidefinite-representable cost function (see \cite{bental_nemirovski} for a definition) $c:\Rn\mapsto \R$ by replacing the objective function with a new variable $\beta$ and adding the constraint $c(x) \leq \beta$.}

%

We start by finding 
the minimum cost point that is one-step safe under all valid dynamics, i.e.,\ a point $x\in S$ such that $Ax\in S$ for all $A\in U_0$. Once this is done, we gain further information by observing the action $y = \trueA x$ of system \eqref{eq:linear dynamics} on our point $x$, which further constrains the uncertainty set $U_0$. We then repeat this procedure with the updated uncertainty set to find the next minimum cost one-step safe point. 
More generally, after collecting $k$ measurements, 
our uncertainty in the dynamics reduces to the set
\begin{align}
    U_k = \{ A \in U_0 \mid  Ax_j = y_j \quad j = 1,\dots,k\}. \label{eq:uk_linear}
\end{align}
Hence, the problem of finding the next cheapest one-step
safe \newnewstuff{initialization} point \versionii{(i.e., the version of \eqref{eq:greedy_safe} for this specific case)} becomes:
\begin{equation}\label{eq:one-step}
\begin{aligned}
\min_{x\in\R^n} \quad & c^\transpose  x\\
\textrm{s.t.} \quad & x \in S \\
& A x \in S \quad \forall A \in U_k. 
\end{aligned}
\end{equation}


In \sectionref{sec:one-step:duality}, we show that
problem \eqref{eq:one-step} can be efficiently solved.
We then use \eqref{eq:one-step}
as a subroutine in a one-step safe learning algorithm which 
we present in \sectionref{sec:one-step:algorithm}.

\subsection{Reformulation via Duality}
\label{sec:one-step:duality}

In this subsection, we reformulate problem \eqref{eq:one-step} as a linear program.
To do this, we introduce auxiliary variables $\mu_j^{(i)} \in \R$ and \versionii{$\eta_\ell^{(i)} \in \R^n$ for $i=1,\dots,r$, $j=1,\dots,s$, and $\ell=1,\dots,k$.
\begin{proposition}\label{prop:one-step LP}
The feasible set of problem \eqref{eq:one-step} is the projection to $x$-space of the feasible set of the following linear program:
\begin{equation}\label{eq:one-step LP}
\begin{aligned}
\min_{x,\mu,\eta} \quad & c^\transpose  x\\
\textrm{s.t.} \quad & h_i^\transpose  x \leq b_i \quad i = 1, \dots ,r \\
& \sum_{\ell=1}^{k} y_\ell^\transpose  \eta_\ell^{(i)} + \sum_{j=1}^{s} \mu_j^{(i)} v_j \leq b_i \quad i = 1, \dots ,r \\
& x h_i^\transpose  = \sum_{\ell=1}^{k} x_\ell \eta_\ell^{(i)\transpose } + \sum_{j=1}^{s} \mu_j^{(i)} V_j^\transpose  \quad i = 1, \dots ,r \\
& \mu^{(i)} \geq 0 \quad i = 1, \dots ,r.
\end{aligned}
\end{equation}
In particular, the optimal values of \eqref{eq:one-step} and \eqref{eq:one-step LP} are the same and the optimal solutions of \eqref{eq:one-step} are the optimal solutions of \eqref{eq:one-step LP} projected to $x$-space.
\end{proposition}}
\myproof{
Using the definitions of $S$ and $U_0$, let us first rewrite \eqref{eq:one-step} as a bilevel program:
\versionii{
\begin{equation}\label{eq:one-step bilevel}
\begin{aligned}
\min_{x} \quad & c^\transpose  x\\
\textrm{s.t.} \quad & h_i^\transpose  x \leq b_i \quad i = 1, \dots ,r \\
& \begin{bmatrix} \max_A \quad & h_i^\transpose  A x  \\ 
\textrm{s.t.} \quad & \Tr(V_j^\transpose  A) \leq v_j \quad j = 1, \dots ,s \\
& Ax_\ell = y_\ell \quad \ell=1, \dots, k \end{bmatrix} \leq b_i \quad i = 1, \dots ,r.
\end{aligned}
\end{equation}}
We proceed by taking the dual of the $r$ inner programs, treating the $x$ variable as fixed.
\versionii{By introducing dual variables $\mu_j^{(i)}$ and $\eta_\ell^{(i)}$ for $i=1,\dots,r$, $j=1,\dots,s$, and $\ell=1,\dots,k$, and by invoking strong duality of linear programming, we have

\begin{equation}\label{eq:one-step duality}
\begin{bmatrix} \max_A \: & h_i^\transpose  A x  \\ 
\textrm{s.t.} \: & \Tr(V_j^\transpose  A) \leq v_j \: j = 1, \dots ,s \\
& Ax_\ell = y_\ell \: \ell=1, \dots, k \end{bmatrix} = 
\begin{bmatrix} \min_{\mu^{(i)},\eta^{(i)}} \: & \sum_{\ell=1}^{k} y_\ell^\transpose  \eta_\ell^{(i)} + \sum_{j=1}^{s} \mu_j^{(i)} v_j \\ 
\textrm{s.t.} \: & x h_i^\transpose  = \sum_{\ell=1}^{k} x_\ell \eta_\ell^{(i)\transpose } + \sum_{j=1}^{s} \mu_j^{(i)} V_j^\transpose  \\
& \mu^{(i)} \geq 0 \end{bmatrix} 
\end{equation}}
for $i=1,\dots,r$.
Thus by replacing the inner problem of \eqref{eq:one-step bilevel} with the right-hand side of \eqref{eq:one-step duality}, the min-max problem \eqref{eq:one-step bilevel} becomes a min-min problem.
This min-min problem can be combined into a single minimization problem and be written as problem \eqref{eq:one-step LP}.
Indeed, if $x$ is feasible to \eqref{eq:one-step bilevel}, for that fixed $x$ and for each $i$, there exist values of $\mu^{(i)}$ and $\eta^{(i)}$ that attain the optimal value for \eqref{eq:one-step duality} and therefore the triple $(x,\mu,\eta)$ will be feasible to \eqref{eq:one-step LP}.
Conversely, if some $(x,\mu,\eta)$ is feasible to \eqref{eq:one-step LP}, it follows that $x$ is feasible to \eqref{eq:one-step bilevel}. This is because for any fixed $x$ and for each $i$, the optimal value of the left-hand side of \eqref{eq:one-step duality} is bounded from above by the objective value of the right-hand side evaluated at any feasible $\mu^{(i)}$ and $\eta^{(i)}$.
}
\versionii{
\begin{remark}
We note that problem \eqref{eq:one-step LP}
can be modified so that
one-step safety is achieved in the presence of 
bounded disturbances.
That is, suppose that the dynamics were governed by
\[
x_{t+1} = \trueA x_t + w_t,
\]
where $w_t$ represents some potentially adversarial disturbance and $\trueA \in U_0$.
We can still give an exact linear programming-based characterization of the one-step safety set in the case when we have $\|w_t\| \leq W_t$, where $\norm{\cdot}$ is any norm whose unit ball is a polytope and $W_t$ is a given scalar.
For example, if $\norm{\cdot}$ is the infinity norm, the set of one-step safe initialization points after observing $k$ measurements from the disturbed dynamics is the projection to $x$-space of the feasible set of the following linear program:
\begin{align*}
\min_{x,\mu,\eta^+,\eta^-} \quad & c^\transpose  x\\
\textrm{s.t.} \quad & h_i^\transpose  x \leq b_i \quad i = 1, \dots ,r \\
& \sum_{j=1}^s \mu_j^{(i)} v_j + \sum_{\ell=1}^{k}\sum_{\ell'=1}^{n} \eta^{+(i)}_{\ell\ell'} (W_\ell + (y_\ell)_{\ell'} )   \\
&\qquad + \sum_{\ell=1}^{k}\sum_{\ell'=1}^{n} \eta^{-(i)}_{\ell\ell'} (W_\ell - (y_\ell)_{\ell'} ) + W_{k+1} \|h_i\|_1 \leq b_i \quad i = 1, \dots, r \\
& x h_i^\transpose  = \sum_{j=1}^s \mu_j^{(i)} V_j^\transpose  + \sum_{\ell=1}^{k}\sum_{\ell'=1}^{n} \eta^{+(i)}_{\ell\ell'} x_\ell e_{\ell'}^\transpose  - \sum_{\ell=1}^{k}\sum_{\ell'=1}^{n} \eta^{-(i)}_{\ell\ell'} x_\ell e_{\ell'}^\transpose  \quad i = 1, \dots, r \\
& \mu \geq 0, \quad \eta^+ \geq 0, \quad \eta^- \geq 0,
\end{align*}
where the input to the problem is the descriptions of $S$ and $U_0$ ($h_i,b_i$ and $V_j,v_j$) and the measurements $(x_\ell,y_\ell)$ and we have introduced dual variables $\mu_j^{(i)}$ for $i=1,\dots,r$, $j=1,\dots,s$ and $\eta^{+(i)}_{\ell\ell'},\eta^{-(i)}_{\ell\ell'}$ for $i=1,\dots,r$, $\ell=1,\dots,k$ and $\ell' = 1,\dots,n$.
\end{remark}
This is a special case of \Cref{thm:nonlinear socp} which will be shown in \Cref{sec:nonlinear}.
\begin{remark}
We note that problem \eqref{eq:one-step LP}
can be modified so that
one-step safety is achieved in the presence of 
bounded measurement noise.
That is, suppose that instead of directly observing $y_k = \trueA x_k$, we observe
\[
y_k = \trueA x_k + z_k,
\]
where $z_k$ represents the noise in the measurement and $\trueA \in U_0$.
We can still give an exact linear programming-based characterization of the one-step safety set in the case when we have $\|z_k\| \leq Z_k$, where $\norm{\cdot}$ is any norm whose unit ball is a polytope and $Z_k$ is a given scalar.
For example, if $\norm{\cdot}$ is the infinity norm, the set of one-step safe initialization points after observing $k$ noisy measurements is the projection to $x$-space of the feasible set of the following linear program:
\begin{align*}
\min_{x,\mu,\eta^+,\eta^-} \quad & c^\transpose  x\\
\textrm{s.t.} \quad & h_i^\transpose  x \leq b_i \quad i = 1, \dots ,r \\
& \sum_{j=1}^s \mu_j^{(i)} v_j + \sum_{\ell=1}^{k}\sum_{\ell'=1}^{n} \eta^{+(i)}_{\ell\ell'} (Z_\ell + (y_\ell)_{\ell'} )   \\
&\qquad + \sum_{\ell=1}^{k}\sum_{\ell'=1}^{n} \eta^{-(i)}_{\ell\ell'} (Z_\ell - (y_\ell)_{\ell'} ) \leq b_i \quad i = 1, \dots, r \\
& x h_i^\transpose  = \sum_{j=1}^s \mu_j^{(i)} V_j^\transpose  + \sum_{\ell=1}^{k}\sum_{\ell'=1}^{n} \eta^{+(i)}_{\ell\ell'} x_\ell e_{\ell'}^\transpose  - \sum_{\ell=1}^{k}\sum_{\ell'=1}^{n} \eta^{-(i)}_{\ell\ell'} x_\ell e_{\ell'}^\transpose  \quad i = 1, \dots, r \\
& \mu \geq 0, \quad \eta^+ \geq 0, \quad \eta^- \geq 0
\end{align*}
where the input to the problem is the descriptions of $S$ and $U_0$ ($h_i,b_i$ and $V_j,v_j$) and the measurements $(x_\ell,y_\ell)$ and we have introduced dual variables $\mu_j^{(i)}$ for $i=1,\dots,r$, $j=1,\dots,s$ and $\eta^{+(i)}_{\ell\ell'},\eta^{-(i)}_{\ell\ell'}$ for $i=1,\dots,r$, $\ell=1,\dots,k$ and $\ell' = 1,\dots,n$.
\end{remark}

We note that we can also exactly characterize one-step safety sets in the presence of both disturbances and noisy measurements.}

\subsection{An Algorithm for One-Step Safe Learning}
\label{sec:one-step:algorithm}

We start by giving a mathematical definition of (exact) safe learning specialized to the case of one-step safety and linear dynamics.
Recall the definition of the set $U_k$ in \eqref{eq:uk_linear}.
\begin{definition}[One-Step Safe Learning]
\label{def:one-step}
We say that one-step safe learning is possible if for some nonnegative integer $m$, we can sequentially choose vectors $x_k \in S$, for $k=1, \ldots, m$, and observe measurements $y_k = \trueA x_k$ such that:
\begin{enumerate}
    \item \textbf{(Safety)} for $k=1,\dots,m$, we have $A x_{k} \in S \quad \forall A \in U_{k-1}$,
    \item \textbf{(Learning)} the set of matrices $U_m$ is a singleton.
\end{enumerate}
\end{definition}


%

We now present our algorithm for checking the possibility of one-step safe learning (\algorithmref{alg:one-step}).
The proof of correctness of
\algorithmref{alg:one-step} is given in \theoremref{thm:one-step}.

\begin{remark}
As \theoremref{thm:one-step} will demonstrate, the particular choice of the parameter $\varepsilon\in(0,1]$ in the input to \algorithmref{alg:one-step} does not affect the detection of one-step safe learning by this algorithm. However, a smaller $\varepsilon$ leads to a lower cost of learning. Therefore, in practice, $\varepsilon$ should be chosen positive and as small as possible without causing the matrix $X$ in line~\ref{eq:x_mat} to be ill conditioned.
\end{remark}

%
%
%
\begin{algorithm2e}[htb]\label{alg:one-step}
\LinesNumbered
\SetKwInOut{Input}{Input}
\SetKwInOut{Output}{Output}
\SetKw{Break}{break}
\DontPrintSemicolon
\Input{polyhedra $S \subset \R^n$ and $U_0 \subset \R^{n \times n}$, cost vector $c \in \R^n$, and a constant $\varepsilon \in (0, 1]$.}
\Output{A matrix $\trueA \in \R^{n \times n}$ or a declaration that one-step safe learning is impossible.}
\For{$k = 0,\dots,n-1$}{
$D_k \gets \{ (x_j, y_j) \mid j=1,\ldots,k\}$\;
$U_k \gets \{ A \in U_0 \mid  Ax_j = y_j, \quad j = 1,\dots,k\}$\;
\If{$U_k$ is a singleton (cf. \lemmaref{lem:uniqueness})}{
\Return the single element in $U_k$ as $\trueA$\; \label{alg:early_return}
}
Let $x_k^\star$ be the projection to $x$-space of an optimal solution to problem \eqref{eq:one-step LP} with data $D_k$\;
\eIf{$x_k^\star$ is linearly independent from $\{x_1,\dots,x_k\}$}{
$x_{k+1} \gets x_k^\star$
}{
Let $S^1_k$ be the projection to $x$-space of the feasible region of problem \eqref{eq:one-step LP} with data $D_k$\;
Compute a basis $B_k \subset S^1_k$ of $\spanof{S^1_k}$ (cf. \lemmaref{lem:basis})\;
\For{$z_j \in B_k$}{
\If{$z_j$ is linearly independent from $\{x_1,\dots,x_k\}$}{
$x_{k+1} \gets (1-\varepsilon)  x_k^\star + \varepsilon z_j$\;
\Break
}
}
\If{no $z_j \in B_k$ is linearly independent from $\{x_1, \dots, x_k\}$}{
\Return one-step safe learning is impossible\; \label{alg:learning_impossible}
}
}
Observe $y_{k+1} \gets \trueA x_{k+1}$\;
}
Define matrix $X = [x_1,\dots,x_n]$ \label{eq:x_mat}\;
Define matrix $Y = [y_1,\dots,y_n]$\;
\Return $\trueA = YX^{-1}$ \label{alg:full_return}
\caption{One-Step Safe Learning Algorithm}
\end{algorithm2e}

\algorithmref{alg:one-step} invokes two subroutines which
we present next in \lemmaref{lem:uniqueness} and \lemmaref{lem:basis}.

\begin{lemma}\label{lem:uniqueness}
Let $A \in \R^{m \times n}$, $B \in \R^{m \times p}$, $c \in \R^m$, and define the polyhedron
\[P \defn \{ x \in \R^n \mid  \exists y \in \R^p \quad \textnormal{s.t.} \quad A x + B y \leq c \}.\]
The problem of checking if $P$ is a singleton can be reduced to solving $2n$ linear programs.
\end{lemma}
\myproof{
For each $i=1,\dots,n$, maximize and minimize the $i$-th coordinate of $x$ over $P$.
It is straightforward to check that $P$ is a singleton if and only if the optimal values of these two linear programs coincide for every $i=1,\dots,n$.
}
%
%
\newnewstuff{In the next lemma, the notation $\spanof{P}$ denotes the set of all linear combinations of points in a set $P \subseteq \R^n$ (see the appendix for a proof of this lemma).}
\begin{lemma}\label{lem:basis}
Let $A \in \R^{m \times n}$, $B \in \R^{m \times p}$, $c \in \R^m$, and define the polyhedron
\[P \defn \{ x \in \R^n \mid  \exists y \in \R^p \quad \textnormal{s.t.} \quad A x + B y \leq c \}.\]
One can find a basis of $\spanof{P}$ contained within $P$ by solving at most $2n^2$ linear programs.
\end{lemma}


Our next theorem is the main result of the section.
\begin{theorem}
\label{thm:one-step}
Given a safety region $S \subset \R^n$ and an uncertainty set $U_0 \subset \R^{n \times n}$, one-step safe learning is possible if and only if \algorithmref{alg:one-step} (with an arbitrary choice of $c \in \R^n$ and $\varepsilon \in (0, 1]$)
returns a matrix.
\end{theorem}
\myproof{
[``If'']
By construction, the sequence of \newnewstuff{initialization point}s chosen by \algorithmref{alg:one-step} satisfies the first condition of \definitionref{def:one-step}, since the vectors $x_k^\star$ and $z_j$ are both contained in $S^1_k$ and any vector in $S^1_k$ will remain in the safety region under the action of all matrices in $U_k$; i.e. all matrices in $U_A$ that are consistent with the measurements made so far.
If \algorithmref{alg:one-step} terminates early at line~\ref{alg:early_return} for some iteration $k$, then clearly the uncertainty set $U_k$ is a singleton.
On the other hand, 
if we reach line~\ref{alg:full_return}, then we must have $n$
linearly independent \newnewstuff{initialization point}s $\{x_1,\dots,x_{n}\}$.
From this, it is clear that the set $\{ A \in U_0 \mid A x_j = y_j, j = 1, \dots, n \} = \{A_\star\}$.
\par
[``Only if''] Suppose \algorithmref{alg:one-step} chooses points $\{x_1,\dots,x_{m}\}$ where $m < n$ and terminates at line~\ref{alg:learning_impossible}.
Then it is clear from the algorithm that $\{x_1,\dots,x_{m}\}$ must form a basis of $\spanof{S^1_m}$ and that $U_m$ is not a singleton.
Take $\tilde{m}$ to be any nonnegative integer and $\{\tilde{x}_1,\dots,\tilde{x}_{\tilde{m}}\}$ to be any sequence that satisfies the first condition of \definitionref{def:one-step}.
For $k=1, \ldots, \tilde{m}$, let
\begin{align*}
    \tilde{U}_k &= \{ A \in U_0 \mid  A\tilde{x}_j = \trueA \tilde{x}_{j},  j = 1,\dots,k\} \:, \\
    \tilde{S}^1_k &= \{x \in S \mid Ax \in S,  \forall A\in \tilde{U}_k\} \:.
\end{align*}
First we claim that $\tilde{x}_k \in S^1_m$ for $k=1,\dots,\tilde{m}$.
We show this by induction.
It is clear that $\tilde{x}_1 \in S^1_m$ since $\tilde{x}_1 \in S^1_0$ and $S^1_0 \subseteq S^1_m$.
Now we assume $\tilde{x}_1,\dots,\tilde{x}_{k} \in S^1_m$ and show that $\tilde{x}_{k+1} \in S^1_m$.
Since $\{x_1,\dots,x_{m}\}$ forms a basis of $\spanof{S^1_m}$, it follows that for any matrix $A$, $Ax_j=\trueA x_j$ for $j=1,\dots,m$ implies $Ax = \trueA x$ for all $x \in S^1_m$.
In particular, for any matrix $A$, $Ax_j=\trueA x_j$ for $j=1,\dots,m$ implies $A\tilde{x}_{j} = \trueA \tilde{x}_{j}$ for all $j=1,\dots,k$.
It follows that $U_m \subseteq \tilde{U}_k$ and therefore, $\tilde{S}^1_k \subseteq S^1_m$.
By the first condition of \definitionref{def:one-step}, we must have $\tilde{x}_{k+1} \in \tilde{S}^1_k$, and thus, $\tilde{x}_{k+1} \in S^1_m$.
This completes the inductive argument and shows that $\tilde{x}_k \in S^1_m$ for $k=1,\dots,\tilde{m}$.
From this, it follows that $U_m \subseteq \tilde{U}_{\tilde{m}}$.
Recall that $U_m$ is not a singleton, thus $\tilde{U}_{\tilde{m}}$ is not a singleton either.
Therefore, the sequence $\{\tilde{x}_1,\dots,\tilde{x}_{\tilde{m}}\}$ does not satisfy the second condition of \definitionref{def:one-step}.
}

\begin{corollary}
\label{cor:n steps}
Given a safety region $S \subset \R^n$ and an uncertainty set $U_0 \subset \R^{n \times n}$, if one-step safe learning is possible, then it is possible with at most $n$ measurements. 
\end{corollary}

\newnewstuff{Note that if $\trueA$ belongs to the interior of $U_0$, any algorithm needs at least $n$ measurements in order to learn $\trueA$.}

\subsection{The Value of Exploiting Information on the Fly}
\label{sec:cost of learning}

In addition to detecting the possibility of safe learning, \algorithmref{alg:one-step}
attempts to minimize the overall cost of learning (i.e.,\ $\sum_{k=1}^{m} c^\transpose  x_k$) by exploiting 
information gathered at every step.
In order to demonstrate the value of using information online,
we construct \algorithmref{alg:offline-one-step} which 
chooses $n$ \newnewstuff{initialization points} $x_1, \dots, x_n$ ahead of time based solely on $U_0$ and $S$.
This algorithm succeeds under the assumption that $S_0^1$ contains a basis of $\R^n$.
%
%
%
%
\begin{algorithm2e}[ht]\label{alg:offline-one-step}
\LinesNumbered
\SetKwInOut{Input}{Input}
\SetKwInOut{Output}{Output}
\SetKw{Break}{break}
\DontPrintSemicolon
\Input{polyhedra $S \subset \R^n$ and $U_0 \subset \R^{n \times n}$, cost vector $c \in \R^n$, and a constant $\varepsilon \in (0, 1]$.}
\Output{A matrix $\trueA \in \R^{n \times n}$ or failure.}
\If{$S_0^1$ does not contain a basis of $\R^n$ (cf. \lemmaref{lem:basis})}{
\Return failure
}
Compute a basis $\{ z_1, \dots, z_n \} \subset S_0^1$ of $\R^n$\;
Let $x_0^\star$ be the projection to $x$-space of an optimal solution to problem \eqref{eq:one-step LP} with data $D_0$\;
Set $x_k = (1-\varepsilon) x_0^\star + \varepsilon z_k$ for $k=1, \dots, n$\;
Observe $y_k \gets A_\star x_k$ for $k=1, \dots, n$\;
Define matrix $X = [x_1,\dots,x_n]$\;
Define matrix $Y = [y_1,\dots,y_n]$\;
\Return $\trueA = YX^{-1}$
\caption{Offline One-Step Safe Learning Algorithm}
\end{algorithm2e}

%
As $\varepsilon$ tends to zero, the cost of Algorithm~\ref{alg:offline-one-step}
approaches $n c^\transpose  x_0^\star$, where $x_0^\star$ is a minimum cost \newnewstuff{initialization point}
in $S_0^1$;
therefore, $n c^\transpose  x_0^\star$
serves as an \emph{upper bound} on the cost incurred by \algorithmref{alg:one-step}.
We note that $n c^\transpose  x_0^\star$ is also the minimum cost achievable by any one-step safe offline algorithm
that takes $n$ measurements, 
since all \newnewstuff{initialization points} $\{x_k\}$ of such an algorithm must come from $S^1_0$.

We refer the reader to \sectionref{sec:one-step example} for a numerical example comparing \algorithmref{alg:one-step} and \algorithmref{alg:offline-one-step}, and to \sectionref{sec:dim increase} for an example where exploiting online information is necessary for safe learning.
\subsection{A Lower Bound on the Cost of Safe Learning}
\label{sec:one step lower bound}
Consider a safety region $S\subset \Rn$, an initial uncertainty set $U_0 \subset \Rnn,$ and an affine function $c: S \mapsto \R_+$. By assuming knowledge of the matrix $\trueA$ governing the true dynamics, we can express the minimum cost of safe learning (cf.~the paragraph before Eq.~\eqref{eq:greedy_safe}) over all possible (online or offline) algorithms as the optimal value of the following optimization problem:

\begin{equation}
\label{eq:onestep lower bound}
\begin{aligned}
\inf_{m \in \mathbb{N},x_1,\dots, x_m \in \Rn} \quad &   \sum_{k=1}^m c(x_k)\\
\textrm{s.t.} \quad & x_k  \in S \quad  k = 1, \dots ,m\\
& Ax_1  \in S \quad \forall A \in U_0 \\
& Ax_2  \in S \quad \forall A \in \{ A \in U_0 \mid  A x_1 = \trueA x_1 \} \\
& Ax_3  \in S \quad \forall A \in \{ A \in U_0 \mid  A x_1 = \trueA x_1, \quad A x_2 = \trueA x_2 \} \\
&  \quad \quad \: \vdots  \\
& Ax_m  \in S \quad \forall A \in \{ A \in U_0 \mid  A x_k = \trueA x_k \quad k = 1,\dots,m-1 \} \\
&  \{ A \in U_0 \mid  A x_k = \trueA x_k \quad k = 1,\dots,m \} = \{ \trueA \}.
\end{aligned}
\end{equation}
For a fixed $m \in \mathbb{N},$ and assuming knowledge of $\trueA$, using a similar duality approach as in the proof of Proposition~\ref{prop:one-step LP}, the membership constraints in \eqref{eq:onestep lower bound} can be written as bilinear constraints in $x_1,\ldots,x_m$ and additional dual variables.
It is unclear however if \eqref{eq:onestep lower bound} can be solved tractably (even for fixed $m$).
Accordingly, we use the following easily computable lower bound on the minimum cost of safe learning for our numerical example in the next section.
Suppose $\trueA$ is in the interior of $U_0$.
Let $$S^1(\trueA) = \{ x \in S \mid \trueA x \in S \}$$ be the true one-step safety region of $\trueA$.
Suppose $x^\star$ is an optimal solution to the linear program that minimizes $c(x)$ over $S^1(\trueA)$.
Since one-step safe learning requires at least $n$ measurements, we cannot 
achieve a cost lower than $n c(x^\star)$.


\subsection{Numerical Example of One-Step Safe Learning}\label{sec:one-step example}
We present a numerical example with $n=4$.
Here, we take $U_0 = \{ A \in \R^{4 \times 4} \mid |A_{ij}| \leq 4 \quad \forall i,j \}$,
$S = \{ x \in \R^4 \mid \norm{x}_{\infty} \leq 1 \}$,
and $c=(-1,-1,0,0)^\transpose $.
We choose the matrix $\trueA$ uniformly at random among integer matrices in $U_0$
\[
\trueA = \begin{bmatrix*}[r]
 2 &  1 &  4 &  2\\
 2 & -3 & -1 & -2\\
-2 & -3 &  1 &  0\\
 2 &  0 & -2 &  2
\end{bmatrix*}.
\]
In this example, \algorithmref{alg:one-step}
takes four steps to safely recover $\trueA$.
The projection to the first two dimensions of the four vectors that \algorithmref{alg:one-step}
selects are plotted in \figureref{fig:one-step S} (note that two of the points are very close to each other).
Because of the cost vector $c$, points higher and further to the right in the plot have lower \newnewstuff{initialization} cost.
Also plotted in \figureref{fig:one-step S} are the projections to the first two dimensions
of the sets $S^1_k$ for $k\in\{0, 1, 2, 3\}$
and of the set $S^1(\trueA)$,
the true one-step safety region of $\trueA$.
In \figureref{fig:one-step U}, we plot $U_k$ (the remaining uncertainty after making $k$ measurements) for $k\in\{0, 1, 2, 3, 4\}$; we draw a two-dimensional projection of these sets of matrices by looking at the trace and the sum of the entries of each matrix in the set.
Note that $U_4$ is a single point since we have recovered the true dynamics after the fourth measurement.

\begin{figure}
\figureconts
{fig:one-step}
{\caption{One-step safe learning associated with the numerical example in \sectionref{sec:one-step example}.}}
{%
\subfigure[$S^1_k$ grows with $k$.][c]{%
\label{fig:one-step S}
\includegraphics[width=.5\textwidth -.5em]{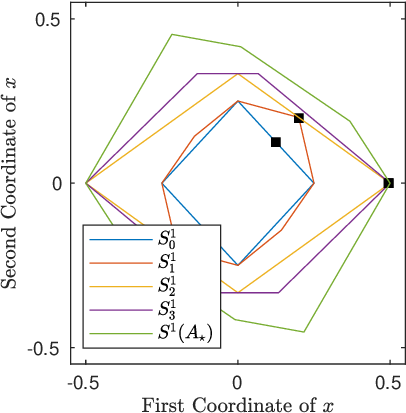}
} 
\subfigure[$U_k$ shrinks with $k$.][c]{%
\label{fig:one-step U}
\includegraphics[width=.5\textwidth -.5em]{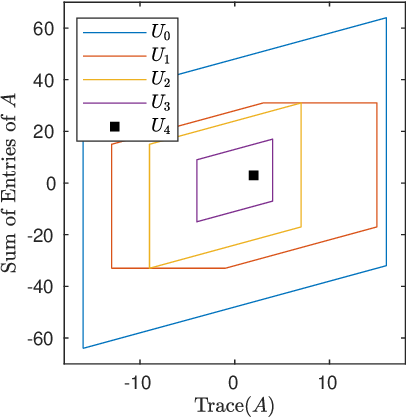}
}
}
\end{figure}

The cost of learning (i.e.,\ $\sum_{i=1}^{4} c_i^\transpose  x_i$) for the offline algorithm (\algorithmref{alg:offline-one-step}) approaches $-1$ as $\varepsilon \rightarrow 0$. The cost of learning for \algorithmref{alg:one-step} (with $\varepsilon=0.01$) is $-1.6385$. The lower bound on the cost of learning is $-2.2264$ (cf.\ \Cref{sec:one step lower bound}).\footnote{Note that adding a constant to the objective function to make it nonnegative over $S$, would shift all of our bounds by the same amount.}
We can see that the value of exploiting information on the fly is significant.

\versionii{To show that the above trend is not specific to the example we chose, we repeat the procedure for 100 randomly generated instances of this problem.
We use the same sets $S$ and $U_0$, we sample the matrix $\trueA$ uniformly at random from integer matrices in $U_0$, and we sample the cost vector $c$ uniformly at random from the unit sphere. In all 100 examples, we learn the true system after 4 iterations as guaranteed by Corollary~\ref{cor:n steps}. The cost of learning for \algorithmref{alg:one-step} was on average $-1.2151$ with a standard deviation of $0.4310$.
The cost of learning for \algorithmref{alg:offline-one-step} was on average $-0.7717$ with a standard deviation of $0.1038$.
The lower bound on the cost of learning was on average $-3.1792$ with a standard deviation of $1.2243$.
The box plot in \figureref{fig:boxplots} summarizes the distribution of initialization costs for each iterate.
For later iterates when more has been learned about the system, lower cost initialization points are chosen by the online algorithm.
}

\begin{figure}
\centering
\includegraphics[width=.5\textwidth]{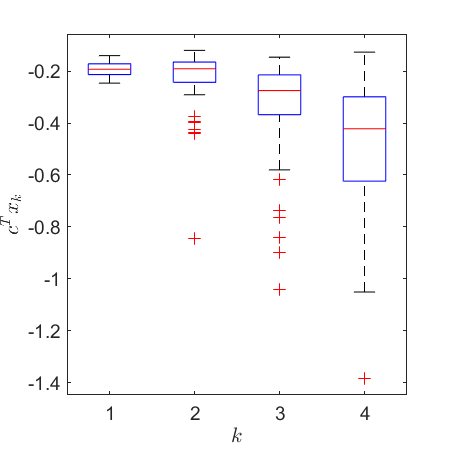}
\caption{\versionii{Initialization cost of the iterates chosen by \algorithmref{alg:one-step} (the online algorithm) for the distribution of four-dimensional problems described at the end of \sectionref{sec:one-step example}. The initialization cost for \algorithmref{alg:offline-one-step} (the offline algorithm) would be $4c^\transpose x_1$.}}
\label{fig:boxplots}
\end{figure}

\newstuff{\subsection{Failure of Offline Learning}
\label{sec:dim increase}

It is natural to ask if in every case that one-step safe learning is possible, whether it is also possible with
an offline algorithm (i.e., an algorithm that can only sample points from $S_0^1$).
In this subsection, we show that this is not the case, demonstrating the necessity of
exploiting information on the fly.

Consider the following example with $n=2$,
\[U_0 = \left \{ A \in \R^{2 \times 2} \bigg\vert \begin{bmatrix}
 1 &  0 \\
 0 & 1 
\end{bmatrix} \leq A \leq \begin{bmatrix}
 2 &  1 \\
 1 & 2 
\end{bmatrix} \right \}
, \quad
S = [1,3]^2,\]
an arbitrary cost vector $c$,
and $\trueA = \begin{bmatrix}
 1 &  0 \\
 1 & 1 
\end{bmatrix}$.
It is straightforward to see that no offline algorithm can recover $\trueA$.
Indeed, one can check that (i) $S_0^1 = \{(1, 1)^\transpose \}$, which does not contain a basis of $\R^2$,
and (ii)
\newnewstuff{
\[
U_1 = \mathrm{conv}\left( \left\{ \begin{bmatrix}
 1 &  0 \\
 0 & 2 
\end{bmatrix},\begin{bmatrix}
 1 &  0 \\
 1 & 1 
\end{bmatrix} \right\} \right)
\]
 which is not a singleton.}

By contrast, \algorithmref{alg:one-step} takes two steps to safely recover $\trueA$, demonstrating
that safe learning is possible. 
Plotted in \figureref{fig:dim increase S} are the sets $S^1_k$ for $k\in\{0, 1\}$ and the set $S^1(\trueA)$, the true one-step safety region of $\trueA$.
In \figureref{fig:dim increase U}, we plot $U_k$ (the remaining uncertainty after making $k$ measurements) for $k\in\{0, 1, 2\}$; we draw a two-dimensional projection of these sets of matrices by plotting the trace and the sum of the entries of each matrix in the set.
As noted previously, $U_1$ is not a singleton as we cannot exactly recover $\trueA$ from measuring its action on the single point in $S_0^1$.
However, $U_2$ is a singleton since we have recovered the true dynamics after the second measurement.
Thus, we see that the value of exploiting information on the fly is significant not just in terms of cost, but in terms of the possibility of learning as well.

\begin{figure}
\figureconts
{fig:dim increase}
{\caption{One-step safe learning associated with the numerical example in \sectionref{sec:dim increase}.}}
{%
\subfigure[$S^1_k$ grows with $k$.][c]{%
\label{fig:dim increase S}
\includegraphics[width=.5\textwidth -.5em]{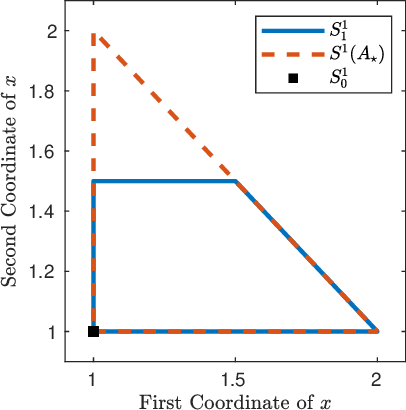}
} 
\subfigure[$U_k$ shrinks with $k$.][c]{%
\label{fig:dim increase U}
\includegraphics[width=.5\textwidth -.5em]{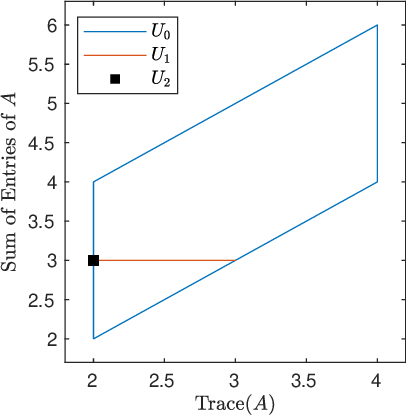}
}
}
\end{figure}

}
\section{Two-Step Safe Learning of Linear Systems}
\label{sec:linear:two-step}

In this section, we again focus on learning the linear dynamics in \eqref{eq:linear dynamics}.
However, unlike the previous section, we are interested in making queries to the system that are two-step safe.
\aaa{An} advantage of this formulation is that we may have fewer system resets and can potentially learn the dynamics with lower \newnewstuff{initialization} cost. \aaa{Moreover, it turns out that the robust optimization problem underlying the two-step safe learning problem remains tractable in the setting where the initial uncertainty set is an ellipsoid (in matrix space). We do not anticipate an exact tractable formulation of the $T$-step safe learning problem for $T\geq 3$.} 
\newnewstuff{Our analysis of the $T=2$ case (in addition to the limiting cases of $T=1$ and $T=\infty$) is mainly motivated by the aforementioned tractability reason (see \theoremref{thm:two-step}).}

\aaa{We take the input to the two-step safe learning problem to be a} polyhedral safety region $S \subset \R^n$ given in the form of \eqref{eq:S polyhedron}, an objective function representing \newnewstuff{initialization} cost which for simplicity we again take to be a linear function $c^\transpose  x$, and an uncertainty set $U_0 \subset \R^{n \times n}$ to which \newnewstuff{the matrix} $\trueA$ belongs.
\aaa{We assume the set $U_0$ is an ellipsoid}; this means that there is a strictly convex quadratic function $q: \R^{n \times n} \rightarrow \R$ such that
\[ U_0 = \left \{ A \in \R^{n \times n} \mid  q(A) \leq 0 \right \}.\]
An example of such an uncertainty set is $U_0 =  \{ A \in \R^{n \times n} \mid \norm{A-A_0}_F \leq \gamma\} $, where $A_0$ is a nominal matrix, $\gamma$ is a positive scalar, and $\norm{\cdot}_F$ \newnewstuff{denotes} to the Frobenius norm.
Having safely collected $k$ length-two trajectories $\{ (x_j, \trueA x_j, \trueA^2 x_j) \}_{j=1}^{k}$, 
our uncertainty around $\trueA$ reduces to:
\begin{align}
    U_k = \{ A \in U_0 \mid A x_j = \trueA x_j, A^2 x_j = \trueA^2 x_j, j=1, \dots, k \}. \label{eq:U_k_two_step}
\end{align}
The optimization problem we would like to solve to find the next best two-step safe \newnewstuff{initialization} point \versionii{(i.e., the version of \eqref{eq:greedy_safe} for this specific case)} is the following:
\begin{equation}\label{two-step}
\begin{aligned}
\min_{x} \quad & c^\transpose  x\\
\textrm{s.t.} \quad & x \in S \\
& A x \in S \quad \forall A \in U_k \\
& A^2 x \in S \quad \forall A \in U_k.
\end{aligned}
\end{equation}

\newstuff{The feasible region of \eqref{two-step} is the set of two-step safe points with information available at step $k$ and is denoted, using our convention, by $S^2_k$.}


\subsection{Reformulation via the S-Lemma}
In this subsection, we derive a tractable reformulation of problem \eqref{two-step}\newstuff{, which as a consequence results in an efficient semidefinite representation of the set $S^2_k$}.
\newnewstuff{Recall that $n$ denotes the dimension of the state and $r$ denotes the number of facets of the polytopic safety set $S$.}

\begin{theorem}\label{thm:two-step}
Problem \eqref{two-step} can be reformulated as a semidefinite program \newstuff{involving $3r$ scalar inequalities and $2r$ positive semidefinite constraints on matrices of size at most $(n^2 + 1) \times (n^2 + 1)$}.
\end{theorem}

Our proof makes use to the S-lemma 
\citep[see\newnewstuff{,} e.g.,][]{s_lemma} which we recall next.
\begin{lemma}[S-lemma]
For two quadratics functions $q_a$ and $q_b$, if there exists a point $\bar{x}$ such that $q_a(\bar{x}) < 0$, then the implication
\[ \forall x, \left[ q_a(x) \leq 0 \Rightarrow q_b(x) \leq 0 \right] \]
holds if and only if there exists a scalar $\lambda \geq 0$ such that
\[\lambda q_a(x) - q_b(x) \geq 0 \quad \forall x .\]
\end{lemma}
\myproofof{of \theoremref{thm:two-step}}{
Note that the set of equations
\[Ax_j = \trueA x_j, \quad  A^2 x_j = \trueA^2 x_j \quad j =1,\dots,k\]
in the definition of $U_k$ in \eqref{eq:U_k_two_step} is equivalent to the set of linear equations 
\begin{equation}\label{eq:two-step linear constraints}
    Ax_j = \trueA x_j, \quad  A (\trueA x_j) = \trueA^2 x_j \quad j =1,\dots,k.
\end{equation}
If there is only one matrix in $U_0$ that satisfies all of the equality constraints in \eqref{eq:two-step linear constraints} (a condition that can be checked via a simple modification of \lemmaref{lem:uniqueness}), then we have found $\trueA$ and \eqref{two-step} becomes a linear program.
Therefore, let us assume that more than one matrix in $U_0$ satisfies the constraints in \eqref{eq:two-step linear constraints}.
In order to apply the S-lemma, we need to remove these equality constraints, a task that we accomplish via variable elimination.
Let $\hat{n}$ be the dimension of the affine subspace of matrices that satisfy the constraints in \eqref{eq:two-step linear constraints} and let $\hat{A} \in \R^{n \times n}$ be an arbitrary member of this affine subspace.
Let $A_1,\dots,A_{\hat{n}}  \in \R^{n \times n}$ be a basis of the subspace
\[\{ A \in \R^{n \times n} \mid Ax_j = 0,\quad A(\trueA x_j) = 0 \quad j=1,\dots,k \}. \]
Consider an affine function $g : \R^{\hat{n}} \rightarrow \R^{n \times n}$ defined as follows:
\[
g(\hat{a}) \defn \hat{A} + \sum_{i=1}^{\hat{n}} \hat{a}_i A_i.
\]
The function $g$ has the properties that it is injective and that for each $A$ that satisfies the equality constraints, there must be a vector $\hat{a}$ such that $A = g(\hat{a})$.
In other words, the function $g$ is simply parametrizing the affine subspace of matrices that satisfy the equality constraints.
Now we can reformulate \eqref{two-step} as:
\begin{equation}\label{eq:two-step parametrization}
\begin{aligned}
\min_{x} \quad & c^\transpose  x\\
\textrm{s.t.} \quad & x \in S \\
& g(\hat{a}) x \in S \quad \forall \hat{a} \quad \textrm{s.t.} \quad q ( g(\hat{a})) \leq 0 \\
& g(\hat{a})^2 x \in S \quad \forall \hat{a} \quad \textrm{s.t.} \quad q ( g(\hat{a})) \leq 0.
\end{aligned}
\end{equation}
Let $\hat{q} \defn q \circ g$.
Since $q$ is a strictly convex quadratic function and $g$ is an injective affine map, $\hat{q}$ is also a strictly convex quadratic function.
Since we are under the assumption that there are multiple matrices in $U_0$ that satisfy the equality constraints, there must be a vector $\bar{a} \in \R^{\hat{n}}$ such that $\hat{q}(\bar{a}) < 0$.
To see this, take $\bar{a}_1 \neq \bar{a}_2$ such that $\hat{q}(\bar{a}_1),\hat{q}(\bar{a}_2) \leq 0$.
It follows from strict convexity of $\hat{q}$ that $\hat{q}(\frac{1}{2} (\bar{a}_1 + \bar{a}_2)) < 0$.
Using the definition of $S$, problem \eqref{eq:two-step parametrization} can be rewritten as:
\begin{equation}\label{eq:two-step bilevel}
\begin{aligned}
\min_{x} \quad & c^\transpose  x\\
\textrm{s.t.} \quad & h_i^\transpose  x \leq b_i \quad i = 1, \dots ,r \\
& \begin{bmatrix} \max_{\hat{a}} \quad & h_i^\transpose  g(\hat{a}) x  \\ 
\textrm{s.t.} \quad & \hat{q}(\hat{a}) \leq 0 \end{bmatrix} \leq b_i \quad i = 1, \dots ,r \\
& \begin{bmatrix} \max_{\hat{a}} \quad & h_i^\transpose  g(\hat{a})^2 x  \\ 
\textrm{s.t.} \quad & \hat{q}(\hat{a}) \leq 0 \end{bmatrix} \leq b_i \quad i = 1, \dots ,r.
\end{aligned}
\end{equation}
Let $q_{1,i}(\hat{a};x) = h_i^\transpose  g(\hat{a}) x - b_i$ and $q_{2,i}(\hat{a};x) = h_i^\transpose  g(\hat{a})^2 x - b_i$.
We consider these functions as quadratic functions of $\hat{a}$ parametrized by $x$.
Note that the coefficients of $q_{1,i}$ and $q_{2,i}$ depend affinely on $x$.
Using logical implications, problem \eqref{eq:two-step bilevel} can be rewritten as:
\begin{equation}\label{eq:two-step implications}
\begin{aligned}
\min_{x} \quad & c^\transpose  x\\
\textrm{s.t.} \quad & h_i^\transpose  x \leq b_i \quad i = 1, \dots ,r \\
& \forall \hat{a}, \left[ \hat{q}(\hat{a}) \leq 0 \Rightarrow q_{1,i}(\hat{a};x) \leq 0 \right] \quad i = 1, \dots ,r \\
& \forall \hat{a}, \left[ \hat{q}(\hat{a}) \leq 0 \Rightarrow q_{2,i}(\hat{a};x) \leq 0 \right] \quad i = 1, \dots ,r.
\end{aligned}
\end{equation}
Now we use the S-lemma to reformulate an implication between quadratic inequalities as a constraint on the global nonnegativity of a quadratic function.
Note that as we have already argued for the existence of a vector $\bar{a}$ such that $\hat{q}(\bar{a}) < 0$, the condition of the S-lemma is satisfied.
After introducing variables $\lambda_{1,i}$ and $\lambda_{2,i}$  for $i=1,\dots,r$, we apply the S-lemma $2r$ times to reformulate \eqref{eq:two-step implications} as the following program:
\begin{equation}\label{eq:two-step sdp}
\begin{aligned}
\min_{x,\lambda} \quad & c^\transpose  x\\
\textrm{s.t.} \quad & h_i^\transpose  x \leq b_i \quad i = 1, \dots ,r \\
& \lambda_{1,i} \hat{q}(\hat{a}) - q_{1,i}(\hat{a};x) \geq 0 \quad \forall \hat{a} \quad i = 1, \dots ,r \\
& \lambda_{2,i} \hat{q}(\hat{a}) - q_{2,i}(\hat{a};x) \geq 0 \quad \forall \hat{a} \quad i = 1, \dots ,r \\
& \lambda_{1,i} \geq 0, \quad \lambda_{2,i} \geq 0 \quad i=1,\dots,r.
\end{aligned}
\end{equation}
It is a standard procedure to convert the constraint that a quadratic function\newstuff{ of $N$ variables} is globally nonnegative into a semidefinite constraint\newstuff{ on a matrix of size $(N+1) \times (N+1)$}.
Note that the coefficients of $q_{1,i}$ and $q_{2,i}$ depend affinely on $x$; this results in linear matrix inequalities when \eqref{eq:two-step sdp} is converted into a semidefinite program.
}
\newstuff{
\subsection{Number of \newnewstuff{Two-Step} Trajectories Needed for Learning}
\label{sec:two-step num traj}
In \sectionref{sec:one-step}, we \newnewstuff{established} that when one-step safe learning is possible, it can be done with at most $n$ trajectories of length one (\newnewstuff{see} Corollary~\ref{cor:n steps}).
It is natural to ask how many trajectories might be required in the case of two-step safe learning. We show that generically, $\lceil \frac{n}{2} \rceil$ two-step trajectories suffice for learning.

Given $m$ two-step trajectories, let $X = [x_1,\dots,x_m]$ be a matrix whose columns are the vectors where our trajectories are initialized.
Since our trajectories are of length two, we will observe \newnewstuff{$Y^{(1)} \defn \trueA X$ and $Y^{(2)} \defn \trueA^2 X$.}
We can write these measurements as the following linear system in $A$:
\begin{equation}
\label{eq:two-step linear system}
\newnewstuff{A[X,Y^{(1)}] = [Y^{(1)},Y^{(2)}].}
\end{equation}
From this, it is clear that $A$ will be uniquely identifiable if the matrix \newnewstuff{$[X,Y^{(1)}]$} has rank $n$.
This is only possible if \newnewstuff{$[X,Y^{(1)}]$} has at least $n$ columns; in particular, this requires that $m \geq \lceil \frac{n}{2} \rceil$.
This suggests that we may be able to learn $A$ with only $\lceil \frac{n}{2} \rceil$ trajectories if we choose $X$ correctly.
Unfortunately, it is possible that no matter how we choose $X$, we may need more than $\lceil \frac{n}{2} \rceil$ trajectories in order to make \newnewstuff{$[X,Y^{(1)}]$} have rank $n$.
This can be seen for example if $A$ is the zero matrix \newnewstuff{or the identity matrix.}
Despite this, we can design an algorithm (Algorithm~\ref{alg:two-step}) for which $\lceil \frac{n}{2} \rceil$ trajectories suffice generically to make the matrix $[X,Y^{(1)}]$ have rank $n$, and hence for learning $\trueA$.
Our algorithm relies on the following lemma as a subroutine (see the appendix for a proof).
\begin{lemma}
\label{lem:sdp_convex_hull}
If $S_0^2$ is full-dimensional, one can solve $2n$ semidefinite programs to find $2n$ vectors in $S_0^2$ whose convex hull is full-dimensional \newnewstuff{(if $S_0^2$ is not full-dimensional, the same process will prove that it is not full-dimensional).
These semidefinite programs have the same variables and constraints as the program from \theoremref{thm:two-step}, and in addition, at most $n$ linear constraints.
}
\end{lemma}
Our algorithm for two-step safe learning is \algorithmref{alg:two-step}. We can prove the following theorem about it.

%
\begin{algorithm2e}[htb]\label{alg:two-step}
\LinesNumbered
\SetKwInOut{Input}{Input}
\SetKwInOut{Output}{Output}
\SetKwInOut{Require}{Require}
\SetKw{Break}{break}
\DontPrintSemicolon
\Require{$S_0^2$ full-dimensional}
\Input{polyhedron $S \subset \R^n$, ellipsoid $U_0 \subset \R^{n \times n}$, cost vector $c \in \R^n$, and a constant $\varepsilon \in (0, 1]$.}
\Output{A matrix $\trueA \in \R^{n \times n}$.}
Compute $2n$ vectors $z_1, \dots, z_{2n} \in S_0^2$
such that $\mathrm{conv}\{z_1, \dots, z_{2n}\}$ is full-dimensional (cf. Lemma~\ref{lem:sdp_convex_hull})\;
Define $m = \lceil \frac{n}{2} \rceil$\;
\For{$k = 0,\dots,m-1$}{
$U_k \gets \{ A \in U_0 \mid  Ax_j = y_j^{(1)}, A y_j^{(1)} = y_j^{(2)} \quad j = 1,\dots,k\}$\;
\If{$U_k$ is a singleton\footnotemark}{
\Return the single element in $U_k$ as $\trueA$\; \label{alg:two_step_early_return}
}
Let $x_k^\star$ be an optimal solution to problem \eqref{two-step} with the set $U_k$ (cf. \theoremref{thm:two-step})\;
Pick a random vector $\lambda \in \R^{2n}$ from the $2n$-dimensional simplex\footnotemark\;
$x_{k+1} \gets (1-\varepsilon) x_k^\star + \varepsilon \sum_{i=1}^{2n} \lambda_i z_i$\;
Observe $y_{k+1}^{(1)} \gets \trueA x_{k+1}$, $y_{k+1}^{(2)} \gets \trueA y_{k+1}^{(1)}$\;
}
\newnewstuff{Define matrix $X = [x_1,\dots,x_m]$\;
Define matrix $Y^{(1)} = [y_1^{(1)},\dots,y_{m}^{(1)}]$\;
Define matrix $Y^{(2)} = [y_1^{(2)},\dots,y_{m}^{(2)}]$\;
\Return $\trueA = [Y^{(1)},Y^{(2)}] [X,Y^{(1)}]^\transpose  ([X,Y^{(1)}] [X,Y^{(1)}]^\transpose )^{-1}$}
\caption{Two-Step Safe Learning Algorithm}
\end{algorithm2e}

\begin{theorem}
\label{thm:two-step generic}
\newnewstuff{Suppose $S^2_0$ is full-dimensional.
For any matrix $\trueA$ outside of a Lebesgue measure zero set in $\Rnn$, \algorithmref{alg:two-step} almost surely succeeds in safe learning using only $\lceil \frac{n}{2} \rceil$ trajectories.}
\end{theorem}

The proof of Theorem~\ref{thm:two-step generic} relies 
on the following proposition whose proof can be found in the appendix.
\begin{proposition}
\label{prop:perturbations_avoid_nullsets}
Let $\lambda^n$ denote the Lebesgue measure on $\R^{n}$.
\newnewstuff{Let $\{\delta_t\}_{t = 1}^{m}$ be a finite sequence of mutually independent
random variables in $\R^n$. Suppose that for each $t$, the law of $\delta_t$ is
absolutely continuous with respect to\ $\lambda^n$.
Let $\{f_t\}_{t = 1}^{m-1}$ be any sequence of functions
mapping $\R^{n \times t}$ to $\R^n$.
Define $z_1 = \delta_1$ and $z_{t} = f_{t-1}(z_1, ..., z_{t-1}) + \delta_t$ for $t =2,\dots,m$.}
For every $\lambda^{n m}$ null-set $N$, we have
$\Pr( (z_1, ..., z_m) \in N ) = 0$.
\end{proposition}

\begin{proof}[of Theorem~\ref{thm:two-step generic}]
\newnewstuff{First observe that by the convexity of $S^2_k$ (i.e., the feasible set of \eqref{two-step}), every initialization point chosen by \algorithmref{alg:two-step} is two-step safe.}

Assume for simplicity that $n$ is even and let $m=\lceil \frac{n}{2} \rceil$.
Consider the set 
\[
\mathcal{V} \defn  \{ [A,X] \in \R^{n \times (n+m)} \mid \text{det}(Z(A,X)) = 0 \}, 
\]
where $Z(A,X) \in \Rnn$ is the first $n$ columns of $[X,AX]$.
\newnewstuff{As a} zero-set of a polynomial, $\mathcal{V}$ is either all of $\R^{n \times (n+m)}$ or \newnewstuff{has} Lebesgue measure zero.
It is not all of $\R^{n \times (n+m)}$ since, defining $I_s$ to be the $s \times s$ identity matrix, we can take
\[
X = \begin{bmatrix*}[r]
 I_m\\
 \hline
 0
\end{bmatrix*},
\quad
A = \begin{bmatrix*}[r]
 0 & \vline & 0\\
 \hline
 I_{\lfloor \frac{n}{2} \rfloor} & \vline & 0
\end{bmatrix*}
\]
\newnewstuff{and observe that}
\[
\text{det}(Z(A,X)) = \text{det}(I_n) = 1 \neq 0.
\]
Therefore $\mathcal{V}$ must have Lebesgue measure zero.
Since the Lebesgue measure on $\R^{n \times (n+m)}$ is the completion of the product measure of the the Lebesgue measures of $\Rnn$ and $\R^{n \times m}$, we have that for almost every $A$, the set
\[
\mathcal{V}_A \defn \{ X \in \R^{n \times m} \mid Z(A,X) = 0 \}
\]
has Lebesgue measure zero.
Thus there must exist a set $\mathcal{A} \subset \Rnn$ of Lebesgue measure zero such that if $A \notin \mathcal{A}$, then $\mathcal{V}_{A}$ has Lebesgue measure zero.

\footnotetext[2]{The same approach as the proof of \lemmaref{lem:uniqueness} can be used to perform this check.}
\footnotetext[3]{Any distribution on the simplex that is absolutely continuous with respect to the Lebesgue measure would work, for example, the uniform distribution.}

Supposing $\trueA \notin \mathcal{A}$, we now apply Proposition~\ref{prop:perturbations_avoid_nullsets}
to the points $x_1,\dots,x_m$ produced by Algorithm~\ref{alg:two-step}.
Assume that the algorithm does not return at Line~\ref{alg:two_step_early_return},
otherwise there is nothing to prove.
Notice 
\newnewstuff{that for some choice of functions $f_1,\dots,f_{m-1}: \R^{n\times t} \rightarrow \R^n$,}
we can write
\newnewstuff{for $k=2,\dots, m$,} 
$x_{k} = f_{k-1}(x_0^\star, \dots, x_{k-1}^\star) + \delta_k$
with $\delta_k = \varepsilon \sum_{i=1}^{2n} \lambda_i z_i$.
It is clear that the law of $\delta_k$ is absolutely continuous \newnewstuff{with respect to}
the Lebesgue measure on $\R^n$ since \newnewstuff{$\text{conv} (\{z_1, \dots, z_{2n}\})$} is full-dimensional.
\newnewstuff{Hence, 
letting} $X \in \R^{n \times m}$ \newnewstuff{be} the matrix
with $x_1, ..., x_m$ as columns,
by Proposition~\ref{prop:perturbations_avoid_nullsets},
$\Pr(X\in \mathcal{V}_{\trueA}) = 0$.
Therefore, almost surely, we have
$\text{det}(Z(\trueA,X)) \neq 0$.
\newnewstuff{This proves that $[X,Y^{(1)}]$ almost surely has rank $n$ and hence 
\[
[Y^{(1)},Y^{(2)}] [X,Y^{(1)}]^\transpose  ([X,Y^{(1)}] [X,Y^{(1)}]^\transpose )^{-1}=\trueA.
\]
}

\end{proof}

}

\subsection{Numerical Example}
\label{sec:two-step example}

\begin{figure}[t]
\figureconts
{fig:two-step}
{\caption{Two-step safe learning associated with the numerical example in \sectionref{sec:two-step example}.}}
{%
\subfigure[$S^2_k$ grows with $k$.][c]{%
\label{fig:two-step S}
\includegraphics[width=.5\textwidth -.5em]{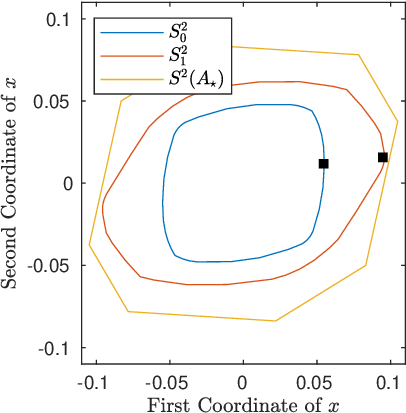}
} 
\subfigure[$U_k$ shrinks with $k$.][c]{%
\label{fig:two-step U}
\includegraphics[width=.5\textwidth -.5em]{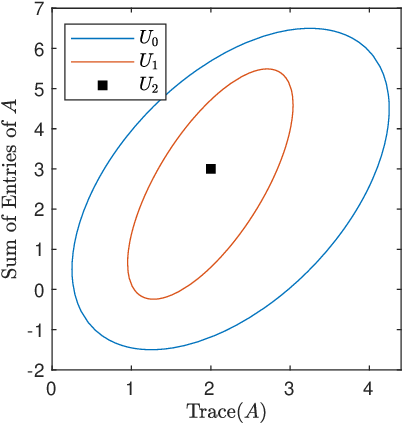}
}
}
\end{figure}

We present a numerical example \newnewstuff{of two-step safe learning}, again with $n=4$.
Here we choose a nominal matrix
\[
A_0 = \begin{bmatrix*}[r]
 2.25 &  0.75 &  4.25 &  1.75\\
 2.25 & -3.25 & -1.25 & -2.25\\
-2.00 & -2.75 &  1.25 &  0.00\\
 1.75 & -0.25 & -2.00 &  2.00
\end{bmatrix*}
\]
and let $U_0 = \{ A \in \R^{4 \times 4} \mid \norm{A-A_0}_F \leq 1 \}$.
We let $S = \{ x \in \R^4 \mid |x_i| \leq 1, i=1,\dots,4 \}$ and $c=(-1,0,0,0)^\transpose $.
We choose the true matrix $\trueA$ to be the same matrix used in \sectionref{sec:one-step example} \newnewstuff{(which belongs to $U_0$)}.
\par
In this example, \newnewstuff{with \algorithmref{alg:two-step}}%
, we learn the true matrix $\trueA$ by \newnewstuff{choosing two initialization points} that are each two-step safe.
In other words, \newnewstuff{\algorithmref{alg:two-step}} chooses $x_1\in\R^4$, observes $\trueA x_1$ and $\trueA^2 x_1$, then chooses $x_2\in\R^4$, and observes $\trueA x_2$ and $\trueA^2 x_2$.
We can verify that we have recovered $\trueA$ if the vectors $\{ x_1, \trueA x_1 , x_2, \trueA x_2\}$ are linearly independent, which is the case.
The projection to the first two dimensions of the two \newnewstuff{initialization points} $x_1$ and $x_2$ are plotted in \figureref{fig:two-step S}.
Because of the cost vector $c$, points further to the right in the plot have lower \newnewstuff{initialization} cost.
Also plotted are the projections to the first two dimensions of the sets
\[ \begin{aligned}
S^2_0 &= \{ x \in S \mid A x \in S, A^2 x \in S \quad \forall A \in U_0\}, \\
S^2_1 &= \{ x \in S \mid A x \in S, A^2 x \in S \quad \forall A \in U_1 \}, \\
S^2(\trueA) &= \{ x \in S \mid \trueA x \in S, \trueA^2 x \in S \}.
\end{aligned} \]
The sets $S^2_0$ and $S^1_0$ are the projections to $x$-space of the feasible regions of our two semidefinite programs (cf. \theoremref{thm:two-step}).
The set $S^2(\trueA)$ is the true two-step safety region of $\trueA$.
In \figureref{fig:two-step U}, we plot $U_k$ (the remaining uncertainty after observing $k$ trajectories of length two) for $k\in\{0, 1, 2\}$; we draw a two-dimensional projection of these sets of matrices by looking at the trace and the sum of the entries of each matrix in the set.
Note that $U_2$ is a single point since we have recovered the true dynamics after observing the second trajectory.
The cost of learning (i.e.,\ $c^\transpose  x_1 + c^\transpose  x_2$) 
is \newnewstuff{$-0.1493$}.

We can construct an analogue of the offline \algorithmref{alg:offline-one-step} by only making measurements from $S^2_0$.
This approach would first pick the optimal point in $S^2_0$ (i.e., $x_1$), and then another vector in $S^2_0$ close to $x_1$, but linearly independent from it.
The cost of learning for this offline approach would be $2c^\transpose  x_1 = -0.1099$.
Finally, we can again find a lower bound on the cost of learning
of any algorithm that \newnewstuff{chooses two two-step safe initialization points}
by assuming we know $\trueA$ ahead of time and optimizing $c^\transpose  x$ over $S^2(\trueA)$; in this example, the lower bound is $-0.2097$.
Here, again, we see that by using information on the fly, we can succeed at safe learning at a considerably lower cost than the offline approach.
\newstuff{\section{Infinite-Step Safe Learning of Linear Systems}
\label{sec:linear}



In contrast to the previous two sections, in this section we consider the problem of safely learning the linear dynamical system in \eqref{eq:linear dynamics} from trajectories of unbounded length.
This means that we are constrained to initializing the system at
points whose \aaa{entire} future trajectories are guaranteed
to remain in a specified safety region.

More formally, in the infinite-step safe learning problem, we have as input a polyhedral safety region $S \subset \R^n$ given in the form
\eqref{eq:S polyhedron},
an objective function representing \newnewstuff{initialization} cost which for simplicity we take to be a linear function $c^\transpose  x$, and a \newnewstuff{polyhedral} uncertainty set $U_0 \subset \R^{n \times n},$ given in the form \eqref{eq:U_0 polyhedron}, to which the matrix $\trueA$ governing the true dynamics belongs.
Having collected $k$ safe trajectories $\{ (x_\ell, \trueA x_\ell, \trueA^2 x_\ell, \ldots) \}_{\ell=1}^{k}$, 
our uncertainty around $\trueA$ reduces to
\begin{align} \nonumber
    U_k = \{ A \in U_0 \mid A x_\ell = \trueA x_\ell, A^2 x_\ell = \trueA^2 x_\ell,\ldots,A^n x_\ell = \trueA^n x_\ell, \ell=1, \dots, k \}.
\end{align}
Note that in the definition of $U_k$, information contained in the tail of the trajectories, beyond step $n$, is discarded. This is because
of the following proposition, which we prove in the appendix using the Cayley-Hamilton theorem.
\begin{proposition}
\label{prop:krylov}
Let $ \trueA\in\mathbb{R}^n$. For any vector $x \in \R^n$ and integer \newnewstuff{$k\geq n$}, we have
\begin{align*}
    &\{ A \in \Rnn \mid A x = \trueA x, A^2 x = \trueA^2 x, \dots, A^\newnewstuff{k} x = \trueA^\newnewstuff{k} x \}  \\ =&\{ A \in \Rnn \mid  A x = \trueA x, A^2 x = \trueA^2 x, \dots, A^n x = \trueA^n x \}.
\end{align*}
\end{proposition}


%
Given the sets $S, U_k$, and the vector $c$, the optimization problem we would like to solve to find the next best infinite-step safe \newnewstuff{initialization} point \versionii{(i.e., the version of \eqref{eq:greedy_safe} for this specific case)} is the following:
\begin{equation}
\label{eq:S_k_infinity}
\begin{aligned}
\min_{x} \quad & c^\transpose  x\\
\textrm{s.t.} \quad & x \in S \\
& A^t x \in S \quad \forall A \in U_k, \quad t=1, 2,\ldots.
\end{aligned}
\end{equation}


In keeping with the naming conventions of this work, we refer to the feasible region of \eqref{eq:S_k_infinity} as $S_k^{\infty}$; if $U_0$ is a single matrix $A$, we call \newnewstuff{this set} $S^{\infty}(A)$.
%
%

Unlike problems \eqref{eq:one-step} and \eqref{two-step}, which as we showed admit a reformulation as tractable conic programs, problem \eqref{eq:S_k_infinity} is in general intractable. In fact, even when the set $U_0$ is a singleton, deciding if a given vector $x$ is feasible to \eqref{eq:S_k_infinity} is NP-hard~\citep[Theorem 2.1]{ahmadi18rdo}.
Therefore, our aim in this section is to find tractable inner approximations to the feasible region of \eqref{eq:S_k_infinity}.

We now describe our assumptions on \eqref{eq:S_k_infinity} and their implications. We assume that our safety region $S$ is compact, as this is typically the case in most applications. It is also natural to require $S_0^\infty$ to be full-dimensional, as otherwise the implementation of a safe \newnewstuff{initialization} would be impossible in presence of arbitrarily small quantization error and/or physical perturbations.
Under these assumptions, we can make the following conclusions.
%
%
%
\begin{proposition}\label{prop:U0bounded.MarginallyStable}
Suppose that $S$ is compact and $S_0^\infty$ is full-dimensional. Then, $U_0$ is bounded\newnewstuff{\footnote{As the proof demonstrates, the claim that $U_0$ is bounded holds even under the weaker requirement that $S$ is compact and $S_0^1$ is full-dimensional.}} and contains only matrices with spectral radius\footnote{Recall that the spectral radius of a square matrix $A$ is defined as
$\rho(A) \defn \max_i |\lambda_i(A)|$, where $\lambda_i(A)$ is the $i$-th eigenvalue of $A$.} less than or equal to one.
\end{proposition}
\begin{proof}
Suppose for the sake of contradiction that $U_0$ is unbounded.
Because $U_0$ is a \newnewstuff{convex set}, there must exist a 
matrix $A_0 \in U_0$ and a nonzero matrix $D \in \R^{n \times n}$ such that $A_0 + \lambda D \in U_0$ for all $\lambda \geq 0$. 
Since $D \neq 0$, the nullspace of $D$
is not full-dimensional and therefore 
$S_0^\infty$ cannot be contained within it.
\newnewstuff{Therefore there exists a vector $x \in S_0^\infty$ such that $D x \neq 0$.}
Now observe that 
$(A_0 + \lambda D) x = A_0 x + \lambda D x \in S$
for every $\lambda \geq 0$, which contradicts
the compactness of $S$.

To prove that $U_0$ only contains matrices with spectral radius at most $1$, suppose for the sake of contradiction that there exists a matrix $\bar{A} \in U_0$ with an eigenvalue $\lambda \in \mathbb{C}$
satisfying $|\lambda| > 1$.
Because the spectral radius is dominated by the operator norm, we have that
for every nonnegative integer $k$, $\norm{\bar{A}^k} \geq \rho(\bar{A}^k) \geq |\lambda|^k$.
\newnewstuff{Let $R$ be a nonnegative scalar large enough such that $x \in S$ implies $\norm{x} \leq R$.
Let $g \in \R^n$ be a random vector with each entry an independent standard normal.
First, we claim that $\Pr(g \in S_0^\infty) = 0$.
Let $U \Sigma_k V^\transpose $ be the 
singular value decomposition of $\bar{A}^k$,
with the largest singular value placed on the first entry of $\Sigma_k$.
By the rotational invariance of Gaussian random vectors,
$\norm{\bar{A}^k g}$ has the same distribution as $\norm{\Sigma_k g}$.
Therefore, letting $g_1$ denote the first entry of the random vector $g$, we have
\begin{align*}
    \Pr(g \in S_0^\infty) & \leq \Pr(\norm{\bar{A}^k g} \leq R \quad \forall k\geq 0) \\
    &\leq \inf_{k \geq 0}\Pr(\norm{\bar{A}^k g} \leq R) \\
    &= \inf_{k \geq 0} \Pr(\norm{\Sigma_k g} \leq R) \\
    &\leq \inf_{k \geq 0} \Pr(\norm{\Sigma_k} |g_1| \leq R) \\
    &\leq \inf_{k \geq 0} \Pr(|\lambda|^k |g_1| \leq R) \\
    &= \inf_{k \geq 0} \Pr( |g_1| \leq R |\lambda|^{-k}) \\
    &= \inf_{k \geq 0} \frac{1}{\sqrt{2\pi}} \int_{-R|\lambda|^{-k}}^{R|\lambda|^{-k}} \exp(-s^2/2) \: ds \\
    &\leq \inf_{k \geq 0} \sqrt{\frac{2}{\pi}} R |\lambda|^{-k} = 0.
\end{align*}
Since $S_0^\infty$ is full-dimensional, its Lebesgue measure is positive.
Furthermore, we showed that the Gaussian measure of $S_0^\infty$ is zero.
Since the Lebesgue measure is absolutely continuous with respect to the Gaussian measure, this is a contradiction.
}
\end{proof}

In view of Proposition~\ref{prop:U0bounded.MarginallyStable}, if we want $S_0^{\infty}$ to be full-dimensional, we must assume that each matrix in $U_0$ has spectral radius less than or equal to one. We make the slightly stronger assumption that
$U_0$ only contains matrices with spectral radius less than one. (Recall that a matrix with spectral radius less than one is called \emph{stable} or \emph{Schur stable}.) Under this assumption, for the set $S_0^\infty$ to be nonempty, we need the origin to be in our safety region $S$ (as otherwise, all initial conditions would converge to the origin under \eqref{eq:linear dynamics} and eventually leave $S$). We work with the slightly stronger assumption that the origin belongs to the interior of $S$. Under this assumption, our representation of the polytope $S$ in \eqref{eq:S polyhedron} can be simplified (after potential rescaling) to:
\begin{equation}
\label{eq:S_polyhedron_all_ones}
S = \left \{ x \in \Rn \mid  h_i^\transpose  x \leq 1 \quad i = 1, \dots ,r \right \}.
\end{equation}

Before we state the main theorem of this section, we need to recall some
basic definitions.
Let $\mathbb{S}^{m \times m}$ denote the
space of $m \times m$ real-valued symmetric matrices. We say that
a matrix-valued function $M : \mathbb{R}^n \rightarrow \mathbb{S}^{m \times m}$ is a \emph{polynomial matrix}
if each entry $M_{ij}$ is a polynomial. 

\begin{definition}[SOS Polynomial and SOS Matrix]\label{def:sos}
A polynomial $p : \Rn \rightarrow \R$ is said to be a \emph{sum of squares} (SOS) if there exist some polynomials $q_1,\dots,q_r \newnewstuff{: \Rn \mapsto \R}$ such that $p = \sum_{i=1}^r q_i^2$.
A polynomial matrix $M : \Rn \rightarrow \mathbb{S}^{m \times m}$ is said to be a \emph{sum of squares matrix} (SOS matrix) if the scalar-valued polynomial $y^\transpose  M(x) y$ in the \newnewstuff{$n+m$} variables $(x,y)$ \newnewstuff{is SOS}.
\end{definition}

We can now present our main theorem of this section, which enables us to find infinite-step safe \newnewstuff{initialization} points. \newnewstuff{Our arguments thus far justify the assumptions that this theorem places on the uncertainty set $U_k$.}
%
%
%
\begin{theorem}\label{thm:linear}
Let the polyhedron $S \subseteq \R^n$
be as in \eqref{eq:S_polyhedron_all_ones}, and the polyhedron
$U_0 \subseteq \R^{n \times n}$ be
as in~\eqref{eq:U_0 polyhedron}.
\newnewstuff{For $t = 0,\dots,n$ and $\ell=1,\dots,k$, let $y_{t,\ell} \in \Rn$ be the $t$th vector in the $\ell$th observed trajectory; i.e., $y_{0,\ell}$ is the trajectory's initialization and $y_{t,\ell} = A^t y_{0,\ell}$.}
Let $\{e_p\}_{p=1}^n$ be the canonical basis vectors of $\Rn$.
For an even integer $d$, let
$\tilde{S}_{k,d}^{\infty}$ be the projection to $x$-space of the feasible region of the following optimization problem:

\begin{subequations}\label{eq:inf_linear_sdp}
\begin{align}
& \min_{x,Q,M_j,M_{t\ell p},\hat{M}_j,\hat{M}_{t\ell p},\sigma_{ij},\sigma_{it\ell p},\varepsilon} \quad  c^\transpose  x \tag{\ref{eq:inf_linear_sdp}}\\
& \mathrm{s.t.} \quad   Q(A) - A Q(A) A^\transpose  = \varepsilon I + M_0(A) + \sum_{j=1}^s M_j(A) (v_j - \Tr(V_j^\transpose  A)) \label{subeq:stabil} \\
&\quad + \sum_{t=1}^{n} \sum_{\ell=1}^{k} \sum_{p=1}^{n} M_{t \ell p}(A) e_p^\transpose (A y_{t-1,\ell} - y_{t,\ell}) \quad \forall A \in \Rnn \nonumber \\
& 1 - h_i^\transpose  Q(A) h_i = \sigma_{i0}(A) + \sum_{j=1}^s \sigma_{ij}(A) (v_j - \Tr(V_j^\transpose  A)) \label{subeq:polar dual} \\
&\quad 
 + \sum_{t=1}^{n} \sum_{\ell=1}^{k} \sum_{p=1}^{n} \sigma_{i t \ell p}(A) e_p^\transpose (A y_{t-1,\ell} - y_{t,\ell})
\quad i = 1, \dots, r \quad \forall A \in \Rnn \nonumber \\
& \begin{bmatrix}
 Q(A) &  x \\
 x^\transpose  & 1 
\end{bmatrix} = \hat{M}_0(A) + \sum_{j=1}^s \hat{M}_j(A) (v_j - \Tr(V_j^\transpose  A)) \label{subeq:ellipsoid} \\
&\quad + \sum_{t=1}^{n} \sum_{\ell=1}^{k} \sum_{p=1}^{n} \hat{M}_{t \ell p}(A) e_p^\transpose (A y_{t-1,\ell} - y_{t,\ell})
\quad \forall A \in \Rnn \nonumber \\
& \varepsilon > 0\label{subeq:strict},
\end{align}
\end{subequations}
\versionii{\begin{itemize}
\item where $Q(A), M_j(A)$ are $n \times n$ SOS matrices with degree at most $d$ for $j=0,\dots,s$, 
\item $M_{t \ell p}(A)$ are $n \times n$ symmetric polynomial matrices with degree at most $d$ for $t=1,\dots,n, \: \ell=1,\dots, k, \: p=1,\dots,n$, 
\item $\hat{M}_j(A)$ are $(n+1) \times (n+1)$ SOS matrices with degree at most $d$ for $j=0,\dots,s$,
\item $\hat{M}_{t \ell p}(A)$ are $(n+1) \times (n+1)$ symmetric polynomial matrices with degree at most $d$ for $t=1,\dots,n, \: \ell=1,\dots, k, \: p=1,\dots,n$, 
\item $\sigma_{ij}(A)$ are SOS polynomials with degree at most $d$ for $i=1,\dots,r,\: j=0,\dots,s$,
\item and $\sigma_{i t \ell p}(A)$ are polynomials with degree at most $d$ for $i=1,\dots,r,\: t=1,\dots,n, \:\ell=1,\dots,k, \: p=1,\dots,n.$
\end{itemize}}
\newnewstuff{
Then,
\begin{enumerate}
    \item[(i)] The program \eqref{eq:inf_linear_sdp} can be reformulated as a semidefinite program of size polynomial in the size of the input ($S$,$U_0$, $\{y_{t,\ell}\}$ and $c$).
    \item[(ii)] We have $\tilde{S}_{k,d}^{\infty} \subseteq S_k^\infty$ (i.e, any vector $x$ feasible to this semidefinite program is infinite-step safe).
    \item[(iii)] Furthermore, if $U_k$ is compact and contains only stable matrices, then, for large enough $d$, the set $\tilde{S}_{k,d}^{\infty}$ is full-dimensional.
\end{enumerate}
}
\end{theorem}
\newnewstuff{In words, \theoremref{thm:linear} allows us to optimize \versionii{the} initialization cost over semidefinite representable subsets of the set of infinite-step safe points.
While the theorem guarantees full-dimensionality of these subsets for large $d$, in our experience, small values of $d$ suffice for safe learning; see \sectionref{sec:infinite example}.}
We present the proof of this theorem in Section~\ref{sec:inf:sdp_reformulation}
after we review some results building up to it in Sections~\ref{sec:inf:rdo_review} and \ref{sec:inf:sos_review}.

\versionii{
\begin{remark}
We note that problem \eqref{eq:inf_linear_sdp}
can be modified so that
infinite-step safety is achieved in the presence of 
bounded measurement noise.
That is, suppose that instead of directly observing $y_{t,\ell} = \trueA y_{t-1,\ell}$, we observe
\[
\hat{y}_{t,\ell} = y_{t,\ell} + z_{t,\ell},
\]
where $z_{t,\ell}$ represents the noise in the measurement and $\trueA \in U_0$.
We can still give an SOS programming-based inner approximation of the infinite-step safety set in the case when we have $\|z_{t,\ell}\| \leq Z_{t,\ell}$, where $\norm{\cdot}$ is, e.g., any polynomial norm (see \cite{polynomial_norms} for a definition) or any norm whose unit ball is a polytope, and $Z_{t,\ell}$ is a given scalar.
To see this, observe that the vectors $\hat{y}_{t,\ell}$ will satisfy
\[
\hat{y}_{t,\ell} = \trueA (\hat{y}_{t-1,\ell} - z_{t-1,\ell}) + z_{t,\ell}.
\]
Now for example, if $\norm{\cdot}$ is the Euclidean norm, and if we have $\max_{A \in U_k} \|A\| \leq M_k$ for some constant $M_k$ (computed, e.g., by a semidefinite relaxation), we can derive the following inequality:
\begin{align*}
\|\hat{y}_{t,\ell} - \trueA \hat{y}_{t-1,\ell} \| &\leq \|\trueA\| \|z_{t-1,\ell}\| + \|z_{t,\ell}\| \\
&\leq M_k Z_{t-1,\ell} + Z_{t,\ell}.
\end{align*}
We can then adapt the methodology of \theoremref{thm:linear} by multiplying the SOS matrices and polynomials in semidefinite program \eqref{eq:inf_linear_sdp} by the polynomials\[\left \{A \mapsto (M_k Z_{t-1,\ell} + Z_{t,\ell})^2 - \left\| \hat{y}_{t,\ell} - A \hat{y}_{t-1,\ell} \right\|^2 \right \}_{t,\ell}\]
instead of $\{A \mapsto e_p^\transpose (A y_{t-1,\ell} - y_{t,\ell})\}_{t,\ell,p}$.
For this modified SOS program, claims (i) and (ii) of \theoremref{thm:linear} hold, and claim (iii) holds under the slightly stronger assumption that $U_0$ is compact.


We note that in the case of noisy measurements, Proposition~\ref{prop:krylov} does not apply anymore and it may be useful to use more than $n$ measurements from a trajectory.

\end{remark}

}







\subsection[Review of a Result from Ahmadi and G{\"{u}}nl{\"{u}}k (2024)]{\texorpdfstring{Review of a Result from \cite{ahmadi18rdo}}{Review of a Result from Ahmadi and G{\"{u}}nl{\"{u}}k (2024)}}

\label{sec:inf:rdo_review}
The basis of \eqref{eq:inf_linear_sdp} comes from the approach of \cite{ahmadi18rdo}.
\newnewstuff{Let the safety set $S$ be as in \eqref{eq:S_polyhedron_all_ones}.}
For a single stable matrix $A$, this approach can be used to compute tractable inner approximations of $S^{\infty}(A)$.

Recall that a matrix $P\in \Snn$ is positive definite (resp. positive semidefinite) if for every nonzero vector $x\in \Rn$ we have that $x^\transpose  P x > 0$ (resp. $x^\transpose  P x \geq 0$); we indicate such a matrix with the notation $P \succ 0$ (resp. $P \succeq 0$).
Furthermore, we use the notation $P \succ Q$ (resp. $P \succeq Q$) if we have that $P-Q$ is positive definite (resp. positive semidefinite).
\newnewstuff{Consider} the following semidefinite program:

\begin{equation}\label{eq:RDO}
\begin{aligned}
\min_{\newnewstuff{x \in \Rn, Q\in \Snn}} \quad & c^\transpose  x\\
\textrm{s.t.} \quad & Q \succ 0 \\
& Q \succeq A Q A^\transpose    \\
& h_i^\transpose  Q h_i \leq 1 \quad i = 1, \dots, r \\
& \begin{bmatrix}
 Q &  x \\
 x^\transpose  & 1 
\end{bmatrix} \succeq 0.
\end{aligned}
\end{equation}
\newnewstuff{
The following lemma is a special case of Theorem 2.11 from \cite{ahmadi18rdo}.
The proof carries some intuition behind the construction of \eqref{eq:inf_linear_sdp} and therefore we include it here.
}
\begin{lemma}\label{lem:RDO}
\newnewstuff{Let $S \subset \Rn$ be as in \eqref{eq:S_polyhedron_all_ones} and $A\in\Rnn$.}
Let $\tilde{S}^{\infty}(A)$ be the projection to $x$-space of the feasible region of \eqref{eq:RDO}.
We have $\tilde{S}^{\infty}(A) \subseteq S^{\infty}(A)$.
\end{lemma}
\begin{proof}
Let $E \defn \{x \mid x^\transpose  Q^{-1} x \leq 1\}$; first we show that the \newnewstuff{constraints $h_i^\transpose  Q h_i \leq 1$ for $i = 1, \dots, r$ imply the set inclusion $E \subseteq S$.
For a set $T \subseteq \Rn$, we define its \emph{polar} $T^\circ$ as $T^\circ \defn \{y \mid y^\transpose  x \leq 1, \forall x\in T\}$.
One can check that $E \subseteq S$ if and only if $S^\circ \subseteq E^\circ$, $S^\circ = \text{conv}(\{h_i\}_{i=1}^r)$, and $E^\circ = \{x \mid x^\transpose  Q x \leq 1\}$.}
Thus, for each $i$, the constraint $h_i^\transpose  Q h_i \leq 1$ implies $h_i \in E^\circ$.
By convexity, it follows that $S^\circ \subseteq E^\circ$ and therefore $E \subseteq S$ as desired.

Note that \newnewstuff{by the Schur complement lemma,} the constraint $\begin{bmatrix}
 Q &  x \\
 x^\transpose  & 1 
\end{bmatrix} \succeq 0$ implies that $x \in E$.
Thus, $x$ is in the safety region.
To show that the trajectory remains safe for all time it suffices to show that the set $E$ is invariant under the dynamics, i.e. that if $\bar{x}$ is in $E$, then so is $A \bar{x}$.
Fix an arbitrary point $\bar{x} \in E$.
By two applications of the Schur complement lemma, the constraint $Q \succeq A Q A^\transpose $ is equivalent to $Q^{-1} \succeq A^\transpose  Q^{-1} A$.
This linear matrix inequality implies that $\bar{x}^\transpose Q^{-1}\bar{x} \geq~\bar{x}^\transpose A^\transpose  Q^{-1} A\bar{x}$.
Thus, we have $(A\bar{x})^\transpose  Q^{-1} (A\bar{x}) \leq \bar{x}^\transpose Q^{-1}\bar{x} \leq 1$, and hence $A\bar{x} \in E$ as desired.
\end{proof}
\newnewstuff{The approach of \cite{ahmadi18rdo} and its extensions lead to infinite-safe sets for dynamics governed by a single matrix, or a group of matrices where the ``joint spectral radius'' is less than one.
Our \theoremref{thm:linear} extends their approach to the case where each individual matrix in $U_k$ is stable, which is a weaker condition than the joint spectral radius of the matrices in $U_k$ being less than one.}
This is the relevant setting for us which is not covered by \cite{ahmadi18rdo}.

We also note that the approach of \cite{ahmadi18rdo} gives a hierarchy of inner approximations to $S^{\infty}(A)$.
However, the first level of the hierarchy is sufficient for our goal of finding full-dimensional inner approximations.

\subsection{Review of Putinar's Positivstellensatz}
\label{sec:inf:sos_review}

In this subsection, we briefly review Putinar's Positivstellensatz and its matrix generalization due to Scherer and Hol which, when combined with \lemmaref{lem:RDO}, help us approximate the feasible region of \eqref{eq:S_k_infinity} with semidefinite programs. These theorems involve SOS polynomials and matrices (cf. Definition~\ref{def:sos}), and our interest in them stems from the following well-known fact: the constraint that an unknown polynomial or a polynomial matrix \newnewstuff{of a given degree} be SOS and satisfy a set of affine inequalities can be cast as an semidefinite program of tractable size; see, e.g., \cite{pablothesis}.

\begin{definition}[Archimedian Property]\label{def:Archimedian}
We say that a set of $n$-variate polynomials $\mathcal{G} = \{g_1,\dots,g_m\}$ satisfies the \emph{Archimedian property} if there exists a scalar $R$ and SOS polynomials $\sigma_0,\sigma_1,\dots,\sigma_m$ such that
\[
R^2 - \sum_{i=1}^n x_i^2 = \sigma_0(x) + \sum_{j=1}^m \sigma_j(x) g_j(x) \quad \forall x \in \Rn.
\]
\end{definition}

Note that the Archimedian property implies that the set
\begin{equation}\label{eq:semialgebraic}
K(\mathcal{G}) \defn \{x \in \Rn \mid g_i(x) \geq 0 \quad i=1,\dots,m \}
\end{equation}
is compact.
Furthermore, it is known that if $g_1,\dots,g_m$ are affine polynomials and if $K(\mathcal{G})$ is compact, then $\mathcal{G}$ satisfies the Archimedian property \citep[see, e.g.,][]{Laurent2009}.
Note that if we let $\mathcal{G} = \{A \mapsto v_j - \Tr(V_j^\transpose  A)\}_{j=1}^s$, then $U_0$ from \eqref{eq:U_0 polyhedron} equals $K(\mathcal{G})$ and $\mathcal{G}$ satisfies the Archimedian property.

\begin{theorem}[Putinar's Positivstellensatz \citep{putinar}]\label{thm:putinar}
Let $\mathcal{G} = \{g_1,\dots,g_m\}$ be a set of $n$-variate polynomials satisfying the Archimedian property and let $K(\mathcal{G})$ be as in \eqref{eq:semialgebraic}.
For any polynomial $p : \Rn \rightarrow \R$, we have $p(x) > 0$ for all $x\in K(\mathcal{G})$ if and only if there exists a positive scalar $\varepsilon$ and SOS polynomials $\sigma_0,\sigma_1,\dots,\sigma_m$ such that
\[
p(x) = \varepsilon + \sigma_0(x) + \sum_{j=1}^m \sigma_j(x) g_j(x) \quad \forall x \in \Rn.
\]
\end{theorem}

\begin{theorem}[Matrix Putinar's Positivstellensatz \citep{HolScherer}]\label{thm:matrix psatz}
Let $\mathcal{G} = \{g_1,\dots,g_m\}$ be a set of $n$-variate polynomials satisfying the Archimedian property and let $K(\mathcal{G})$ be as in \eqref{eq:semialgebraic}.
For any polynomial matrix $M : \Rn \rightarrow \mathbb{S}^{r \times r}$, we have $M(x) \succ 0$ for all $x\in K(\mathcal{G})$ if and only if there exists a positive scalar $\varepsilon$ and SOS matrices $S_0, S_1, \dots, S_m$ such that
\[
M(x) = \varepsilon I + S_0(x) + \sum_{j=1}^m S_j(x) g_j(x) \quad \forall x \in \Rn.
\]
\end{theorem}

\subsection{Proof of Theorem~\ref{thm:linear}}
\label{sec:inf:sdp_reformulation}


In addition to Theorems~\ref{thm:putinar}~and~\ref{thm:matrix psatz}, the proof of claim (iii) in \theoremref{thm:linear} also relies on the following technical
lemma.
\begin{lemma}
\label{lem:polynomial_lyapunov}
Let $U \subset \Rnn$ be a compact set of matrices.
Then, every matrix $A \in U$ is stable if and only if there exists a $\nn$ SOS matrix $P:\Rnn \mapsto \Snn$ such that
\begin{enumerate}
    \item $P(A) \succ 0 \quad \forall A \in U$,
    \item $P(A) - A^\transpose  P(A) A \succ 0 \quad \forall A \in U$.
\end{enumerate}
\end{lemma}
\begin{proof}
[``If'']
It is straightforward to check that the conditions imply that for any matrix $A \in U$, the positive definite Lyapunov function $V_A(x) = x^\transpose  P(A) x$ satisfies $V_A(Ax) < V(x)$ for all $x \neq 0$.
The stability of $A$ then follows from Lyapunov's stability theorem; see, e.g.,~\cite[Theorem 4.3]{zak}.
\par
[``Only if'']
For a positive integer $N$, let the SOS matrix $P_N:\Rnn \rightarrow \Snn$ be defined as follows:
\[
P_N(A) = \sum_{k=0}^N (A^k)^\transpose  (A^k).
\]
Clearly we have $P_N(A) \succ 0$ for each matrix $A \in U$ since each summand is positive semidefinite and the zeroth summand is the identity matrix.
We claim that for sufficiently large $N$, $P_N(A)$ will satisfy
$P_N(A) - A^\transpose  P_N(A) A \succ 0$ for
all $A \in U$.
\newnewstuff{Observe}
\[
P_N(A) - A^\transpose  P_N(A) A = I - (A^{N+1})^\transpose  (A^{N+1}).
\]
To show that $P_N(A) - A^\transpose  P_N(A) A \succ 0$ for each $A \in U$, \newnewstuff{we prove}
\[
\|(A^{N+1})^\transpose  (A^{N+1})\| < 1 \quad \forall A \in U.
\]
For a matrix $B$, 
let $\|B\|_{\infty} \defn \max_{ij} |B_{ij}|$.
Define the scalars $R, M$ as
\[
R \defn \max_{A \in U} \rho(A), \quad M \defn \max_{A \in U} \|A\|_{\infty}.
\]
Since each matrix $A \in U$ is stable
and $U$ is compact, $R < 1$.
Since $U$ is compact, $M < \infty$.
Now fix a matrix $A \in U$ and write $A=QTQ^{-1}$, where $Q \in \mathbb{C}^{n \times n}$ is unitary and $T \in \mathbb{C}^{n \times n}$ is upper triangular (this ``Schur decomposition'' always exists).
Observe that $\|A^N\|=\|T^N\|$.
\newnewstuff{We can bound} the norm of powers of a triangular matrix as follows (see Corollary 3.15 of~\cite{Dowler2013BoundingTN}):
\[
\|T^N\| \leq \sqrt{n} \sum_{j=0}^{n-1} \binom{n-1}{j} \binom{N}{j} \|T\|_\infty^j \rho(T)^{N-j}.
\]
In particular, we have
\begin{equation}
\label{eq:polynomial_lyapunov_bound}
\|A^N\| \leq \sqrt{n} \sum_{j=0}^{n-1} \binom{n-1}{j} \binom{N}{j} M^j R^{N-j}
\leq \sqrt{n} \sum_{j=0}^{n-1} \binom{n-1}{j} N^j M^j R^{N-j}.
\end{equation}
Inequality~\eqref{eq:polynomial_lyapunov_bound}
implies that $\lim_{N \rightarrow \infty} \| A^N \| = 0$.
Therefore, we can choose $N$ large enough such that
\[
\|(A^{N+1})^\transpose  (A^{N+1})\| \leq \|A^{N+1}\|^2 < 1.
\]
\end{proof}

We are now able to present the proof of the main result of this section.

\begin{proof}[of \theoremref{thm:linear}]

\newnewstuff{(i)} Recall that for any fixed degree $d$, the SOS constraints in \versionii{\eqref{eq:inf_linear_sdp}} can be reformulated as semidefinite programming constraints of size polynomial in $n$; see, e.g., \cite{pablothesis}.
The constraints in \eqref{subeq:stabil}, \eqref{subeq:polar dual}, \eqref{subeq:ellipsoid} can be imposed by coefficient matching via a number of linear equations bounded by a polynomial function of the size of the input $(S,U_0,\{y_{t,\ell}\})$.
The constraint that $\varepsilon > 0$ can be rewritten as the constraint $\begin{bmatrix} \varepsilon &1 \\ 1& \delta \end{bmatrix} \succeq 0 $ for a new variable $\delta$.
Therefore, \newnewstuff{for any fixed degree $d$,} \eqref{eq:inf_linear_sdp} is a semidefinite program of size polynomial in the size of the input $(S,U_0,\{y_{t,\ell}\},c)$.

\newnewstuff{(ii)} Let $(x,Q,M_j,M_{t\ell p},\hat{M}_j,\hat{M}_{t\ell p},\sigma_{ij},\sigma_{it\ell p},\varepsilon)$ be feasible to \eqref{eq:inf_linear_sdp}.
It is straightforward to check that 
for every $A \in U_k$, the tuple $(x, Q(A))$ satisfies the following constraints
\begin{equation}\label{eq:RDO_parametrized}
\begin{aligned}
& Q(A) \succ 0 \\
& Q(A) \succeq A Q(A) A^\transpose \\
& h_i^\transpose  Q(A) h_i \leq 1 \quad i = 1, \dots, r \\
& \begin{bmatrix}
 Q(A) &  x \\
 x^\transpose  & 1 
\end{bmatrix} \succeq 0.
\end{aligned}
\end{equation}
Therefore, by \lemmaref{lem:RDO}, we have $x \in \tilde{S}^\infty (A) \subseteq S^\infty (A)$ for every $A \in U_k$.
\newnewstuff{Hence,} 
\[
x \in \bigcap_{A \in U_k} S^\infty (A) = S_k^\infty.
\]
This implies that $\tilde{S}_{k,d}^\infty \subseteq S_k^\infty$.

\newnewstuff{(iii)} Suppose that $U_k$ is compact and contains only stable matrices. It follows that the set
\[
U_k^\transpose  \defn \{ A^\transpose  \mid A \in U_k \}
\]
is also compact and contains only stable matrices. By Lemma~\ref{lem:polynomial_lyapunov} applied to $U_k^\transpose $,
there exists an SOS matrix $P(A)$ which satisfies
\begin{align*}
P(A) &\succ 0 \quad \forall A \in U_0^\transpose  \\
P(A) &\succ A^\transpose  P(A) A  \quad \forall A \in U_0^\transpose .
\end{align*}
Now by defining $Q(A) \defn P(A^\transpose )$, we observe that
\begin{align*}
Q(A) &\succ 0 \quad \forall A \in U_0 \\
Q(A) &\succ A Q(A) A^\transpose   \quad \forall A \in U_0.
\end{align*}

Analogously to how we derived linear constraints in \eqref{eq:two-step linear constraints}, we can rewrite the description of $U_k$ as
\[
U_k = \{ A \in U_0 \mid A y_{t-1,\ell} = y_{t,\ell},t=1,\dots,n, \ell=1, \dots, k \}.
\]
Since $U_k$ is a compact polyhedron, the Archimedian property is satisfied
for the polynomials $\{ A \mapsto v_j - \Tr(V_j^\transpose  A)\}$ and $\{ A \mapsto \pm e_p^\transpose (A y_{t-1,\ell} - y_{t,\ell})\}$.
By \theoremref{thm:matrix psatz}, since $Q(A) - A Q(A) A^\transpose  \succ 0$ for every $A \in U_k$, there exists a positive scalar $\varepsilon$ and SOS matrices $M_j(A)$ and $M_{t\ell p}^\pm(A)$ that satisfy 
\begin{align*}
Q(A) - A Q(A) A^\transpose  =& \varepsilon I + M_0(A) + \sum_{j=1}^s M_j(A) (v_j - \Tr(V_j^\transpose  A)) \\
&+ \sum_{t=1}^{n} \sum_{\ell=1}^{k} \sum_{p=1}^{n} M_{t \ell p}^+(A) e_p^\transpose (A y_{t-1,\ell} - y_{t,\ell}) \\
&- \sum_{t=1}^{n} \sum_{\ell=1}^{k} \sum_{p=1}^{n} M_{t \ell p}^-(A) e_p^\transpose (A y_{t-1,\ell}-y_{t,\ell}) \quad \forall A \in \Rnn.
\end{align*}
By letting $M_{t\ell p}(A) = M_{t\ell p}^+(A) - M_{t\ell p}^-(A)$, we can satisfy \eqref{subeq:stabil}.

Now observe that for each $i \in \{1,\dots,r\}$, the function $A \rightarrow h_i^\transpose Q(A)h_i$ is continuous. Therefore, since $U_k$ is compact, there exists a positive scalar $\alpha$ satisfying
\[
\max_{i\in\{1,\dots,r\},A \in U_k} h_i^\transpose Q(A)h_i < \alpha.
\]
Observe that the tuple $(\varepsilon/\alpha, Q(A)/\alpha, M_j(A)/\alpha,M_{t\ell p}/\alpha)$
still satisfies \eqref{subeq:stabil}, \eqref{subeq:strict}, \versionii{and the SOS constraints}.
Furthermore, for each $i \in \{1,\dots,r\}$,
$ 1 - \alpha^{-1} h_i^\transpose  Q(A) h_i > 0$ for all $A \in U_k$.
Therefore, by \theoremref{thm:putinar}, and by a similar argument as in the case of constraint \eqref{subeq:stabil}, there exist SOS polynomials $\sigma_{ij}(A)$ and polynomials $\sigma_{it\ell p}(A)$
satisfying \eqref{subeq:polar dual}.

Since $Q(A)/\alpha$ is a continuous function of $A$ and since $Q(A)/\alpha$ is positive definite for each $A$ in the compact set $U_k$, there exists a scalar $\beta > 0$ such that $Q(A)/\alpha \succ \beta I$ for \newnewstuff{all} $A \in U_k$.
Then, we have \[\begin{bmatrix}
 Q(A)/\alpha - \beta I &  0 \\
 0 & \frac{1}{2} 
\end{bmatrix} \succ 0 \quad \forall A \in U_k.\]
It follows from \theoremref{thm:matrix psatz}, and by a similar argument as in the case of constraint \eqref{subeq:stabil}, that there exist some SOS matrices $\hat{M}_j(A)$ and symmetric polynomial matrices $\hat{M}_{t\ell p}(A)$ satisfying 
\begin{align*}
\begin{bmatrix}
 Q(A)/\alpha - \beta I &  0 \\
 0 & \frac{1}{2}
\end{bmatrix} =& \hat{M}_0(A) + \sum_{j=1}^s \hat{M}_j(A) (v_j - \Tr(V_j^\transpose  A)) \\
&+ \sum_{t=1}^{n} \sum_{\ell=1}^{k} \sum_{p=1}^{n} \hat{M}_{t \ell p}(A) e_p^\transpose (A y_{t-1,\ell} - y_{t,\ell}) \quad \forall A \in \Rnn.
\end{align*}
\newnewstuff{Observe that for} any $x\in\Rn$ satisfying $\|x\| \leq \sqrt{\frac{1}{2\beta}}$, we have $\begin{bmatrix}
 \beta I &  x \\
 x^\transpose  & \frac{1}{2}
\end{bmatrix} \succeq 0$.
\newnewstuff{Therefore}, $A \mapsto \hat{M}_0(A) + \begin{bmatrix}
 \beta I &  x \\
 x^\transpose  & \frac{1}{2}
\end{bmatrix}$ is still an SOS matrix of the same degree as $\hat{M}_0(A)$.
Hence, for any $x$ satisfying $\|x\| \leq \sqrt{\frac{1}{2\beta}}$,
we have that the tuple 
\begin{equation}\nonumber
    \left( x, Q/\alpha, M_j/\alpha, M_{t\ell p}/\alpha, \hat{M}_0 + \begin{bmatrix} \beta I & x \\ x^\transpose  & \frac{1}{2} \end{bmatrix}, \hat{M}_1, \dots, \hat{M}_s,\hat{M}_{t\ell p}, \sigma_{ij},\sigma_{it\ell p}, \varepsilon / \alpha \right)
\end{equation}
is feasible to \eqref{eq:inf_linear_sdp}
for some degree $d$ large enough.


\end{proof}

\subsection{Number of Trajectories Needed to Learn}
\label{sec:inf-step num traj}
In Corollary~\ref{cor:n steps}, we established that we need no more than $n$ one-step trajectories to safely learn the matrix $\trueA \in \Rnn$ governing the true linear dynamical system of interest.
Then in \theoremref{thm:two-step generic}, we proved that generically, it is possible to safely learn $\trueA$ using only $\lceil \frac{n}{2} \rceil$ trajectories \newnewstuff{of length two}.
We now show that generically, we can \newnewstuff{safely} learn $\trueA$ from a single trajectory of length $n$.
\begin{theorem}\label{thm: inf num traj}
\newnewstuff{
There exists a set $\mathcal{A} \subset \Rnn$ of Lebesgue measure zero such that if $\trueA \notin \mathcal{A}$, then by observing a single trajectory of length $n$ initialized} at random\footnote{Any distribution that is absolutely continuous with respect to the Lebesgue measure would work; for example, the uniform distribution.} from any full-dimensional infinite-step safe set (for example, the set $\tilde{S}_{0,d}^{\infty}$ defined in \theoremref{thm:linear} for large enough $d$), we will almost surely safely learn $\trueA$.
\end{theorem}
\begin{proof}
Consider the set 
\[
\mathcal{V} \defn \{ [A,x] \in \R^{n \times (n+1)} \mid \text{det}([x,Ax,A^2x,\dots,A^{n-1}x]) = 0 \}.
\]
This is the zero-set of a polynomial, therefore it is either the entire space or has Lebesgue measure zero.
It is not the entire space since we can take $A$ to be the matrix with ones on its first subdiagonal and zeros elsewhere and $x$ to be the vector with one as its first entry and zeros elsewhere.
With $A$ and $x$ defined this way, we have
\[
\text{det}([x,Ax,A^2x,\dots,A^{n-1}x]) = \text{det}(I) = 1 \neq 0.
\]
Therefore, $\mathcal{V}$ must have Lebesgue measure zero.
Since the Lebesgue measure on $\R^{n \times (n+1)}$ is the completion of the product measure of the Lebesgue measures of $\Rnn$ and $\Rn$, we have that for almost every $A$, the set
\[
\mathcal{V}_A \defn \{ x \in \Rn \mid \text{det}([x,Ax,A^2x,\dots,A^{n-1}x]) = 0 \}
\]
has Lebesgue measure zero.
Thus, there must exist a set $\mathcal{A} \subset \Rnn$ of Lebesgue measure zero such that if $A \notin \mathcal{A}$, then $\mathcal{V}_{A}$ has Lebesgue measure zero.
Now assume that $\trueA \notin \mathcal{A}$
and let $x$ be the initialization of our observed trajectory.
Because $x$ is sampled at random\footnotemark[6] from a full-dimensional infinite-step safe set, 
we have $\Pr( x \not\in \mathcal{V}_{\trueA}) = 1$.
When $x \not\in \mathcal{V}_{\trueA}$, we have  $\text{det}([x,\trueA x,\trueA^2x,\dots,\trueA^{n-1}x])$ is nonzero, and therefore $[x,\trueA x,\trueA^2x,\dots,\trueA^{n-1}x]$ is invertible.
Since we observe $[\trueA x,\trueA^2x,\dots,\trueA^{n}x]$, we can now recover $\trueA$ by solving a linear system
\[
\trueA[x,\trueA x,\trueA^2x,\dots,\trueA^{n-1}x] = [\trueA x,\trueA^2x,\dots,\trueA^{n}x]
\] \[
\Rightarrow \trueA = [\trueA x,\trueA^2x,\dots,\trueA^{n}x] \left([x,\trueA x,\trueA^2x,\dots,\trueA^{n-1}x]\right)^{-1}.
\]
\end{proof}

\subsection{Numerical Examples}\label{sec:infinite example}
In this section, we present two numerical examples of infinite-step safe learning.
\subsubsection{Inner and Outer Approximations of the Infinite-Step Safe Set}
\label{sec:infinite_example_one}
In our first example, we take $n=2$,
\[
U_0 =  \left \{ A \in \R^{2 \times 2} \middle\vert  A_{1,1} = A_{2,2} = \frac{1}{2}, A_{1,2},A_{2,1} \geq 0, A_{1,2} + A_{2,1} = \frac{9}{5} \right \} ,
\]
and
\[
S = \left \{ x \in \R^2 \middle\vert \begin{bmatrix}
  1 & 1 \\
 -1 & 0 \\
 0 & -1
\end{bmatrix} x \leq \begin{bmatrix}
 1 \\
 1  \\
 1
\end{bmatrix} \right \}.
\]
One can check that every matrix in $U_0$ is stable, though there are products of matrices in $U_0$ that have spectral radius greater than one and hence the techniques of \cite{ahmadi18rdo} do not apply.
We solve the semidefinite program \eqref{eq:inf_linear_sdp} with degree $d=4$
(\newnewstuff{in this example, the semidefinite program with $d=2$ is infeasible}).
In \figureref{fig:infinite1}, we plot $S$, $\tilde{S}_{0,4}^{\infty}$, and $\tilde{S}^{\infty}(A)$ for various matrices $A$ in $U_0$.
\newnewstuff{
We also plot $\bar{S}_0^\infty$ which is the intersection of $S^{10}(A)$ for various matrices $A$ in $U_0$; in particular, $\bar{S}_0^\infty$ is an outer approximation of $S^\infty_0$. This outer approximation is not too much larger than $\tilde{S}_{0,4}^{\infty}$, our inner approximation of $S^\infty_0$.

In this example, we also observe that
\[
\tilde{S}_{0,4}^{\infty} = \bigcap_{A\in U_0} \tilde{S}^{\infty}(A).
\]
From the proof of \theoremref{thm:linear}, we have that $\tilde{S}_{0,d}^{\infty} \subseteq \bigcap_{A\in U_0} \tilde{S}^{\infty}(A)$ for all $d$.
Therefore, this example shows not only that $d=4$ is high enough to get a full-dimensional inner approximation of $S^\infty_0$, but also that $d=4$ is sufficient to get the largest possible infinite-step safe set based on our approach.}
\subsubsection{Comparing One, Two, and Infinite-step Safety}
\label{sec:infinite_example_two}
In our second example, we take $n=2$,
\[
U_0 = \left \{ A \in \Rnn \middle\vert \: \left\|A - \begin{bmatrix}
 1 & .5 \\
 -.5 & .5  
\end{bmatrix} \right\|_F \leq 0.1 \ \right \},
\]
and the same safety region $S$ as in the previous example.
\newnewstuff{We take $\trueA = \begin{bmatrix}
 1.05 & .5 \\
 -.5 & .5  
\end{bmatrix} \in U_0$ and the initialization cost function to be given by the affine function $c(x) = (-1,0)^\transpose  x + 3$, which is nonnegative over $S$.

Since the initial uncertainty set $U_0$ is not polyhedral, we replace the linear programs in \algorithmref{alg:one-step} for one-step safe learning with semidefinite programs. This is done by taking \eqref{two-step}, discarding the two-step constraint, and then converting the resulting problem to a semidefinite program by the same method as in \theoremref{thm:two-step}. Our algorithm then learns $\trueA$ with two one-step safe trajectories with a total initialization cost of $3.1489$.

For two-step safe learning, we use \algorithmref{alg:two-step}.
We learn $\trueA$ with one two-step safe trajectory with an initialization cost of $1.9252$.

For infinite-step safe learning, %
}%
we adapt the method of \theoremref{thm:linear} to the non-polyhedral set $U_0$ by multiplying the SOS matrices and polynomials in semidefinite program \eqref{eq:inf_linear_sdp} by the polynomial
\[\left \{A \mapsto 0.1^2 - \left\|A - \begin{bmatrix}
 1 & .5 \\
 -.5 & .5  
\end{bmatrix} \right\|_F^2 \right \}\]
instead of $\{A \mapsto v_j - \Tr(V_j^\transpose  A)\}_{j=1}^s$.
\newnewstuff{We learn $\trueA$ with one infinite-step safe trajectory with an initialization cost of $2.0080$.

In this example, we see that two-step learning incurs the lowest total initialization cost.
This is because a single two-step safe trajectory is sufficient for learning $\trueA$.
Therefore, the initialization cost is incurred once as opposed to twice for one-step safe learning.
Additionally, requiring two-step safety is less restrictive than requiring infinite-step safety resulting in lower cost compared to infinite-step safe learning.
}

In \figureref{fig:infinite2}, we plot $S$, $S_0^1$, $S_1^1$, $S_0^2$, $\tilde{S}_{0,2}^{\infty}$ \newnewstuff{and the initialization points chosen by each algorithm.} We observe the inclusion relationships $\tilde{S}_{0,2}^{\infty} \subseteq S_0^2\subseteq S_0^1\subseteq S$ \newnewstuff{as expected.
Note that $S_0^1$ and $S_0^2$ are exact characterizations of the one-step and two-step safety sets, respectively, while $\tilde{S}_{0,2}^{\infty}$ is an inner approximation of $S_0^{\infty}$, the true infinite-step safety set.
Since $\tilde{S}_{0,2}^{\infty}$ is not much smaller than $S_0^2$, which is a superset of $S_0^{\infty}$, we see that our semidefinite representable set $\tilde{S}_{0,d}^{\infty}$ with $d=2$ closely approximates the true infinite-step safety set $S_0^{\infty}$.}

\versionii{To show that the above trend is not specific to the example we chose, we repeat the procedure for 100 randomly generated instances of this problem.
We use the same sets $S$ and $U_0$, we sample the matrix $\trueA$ uniformly at random from $U_0$, and we sample the cost vector $c$ uniformly at random from the unit sphere.
In all 100 examples, we learn the true system after two one-step trajectories as guaranteed by Corollary~\ref{cor:n steps}, or with a single trajectory of length two (or longer) as suggested by \Cref{thm:two-step generic} (or \Cref{thm: inf num traj}).
The cost of one-step learning was on average $3.5324$ with a standard deviation of $0.5907$.
The cost of two-step learning was on average $1.9484$ with a standard deviation of $0.1963$.
The cost of infinite-step learning was on average $2.0246$ with a standard deviation of $0.1746$.
The box plot in \figureref{fig:boxplots2} summarizes the distribution of initialization costs for each version.
Here, the initialization cost of infinite-step safe learning is slightly higher than that of two-step safe learning. Observe that since $S^\infty_0 \subseteq S^2_0$, and since a two-dimensional system can be learned with a single two-step trajectory, the cost of two-step safe learning here will never be worse than that of infinite-step safe learning. While $S^2_0 \subseteq S^1_0$, the cost of one-step safe learning is higher in these examples since initialization cost is paid twice.
}

\begin{figure}[h!t]
\figureconts
{fig:infinite}
{\caption{Infinite-step safe learning associated with the numerical examples in \sectionref{sec:infinite example}.}}
{%
\subfigure[Inner approximation of the infinite-step safety set of a linear system associated with the example in \sectionref{sec:infinite_example_one}.][t]{%
\label{fig:infinite1}
\includegraphics[width=.5\textwidth -.5em]{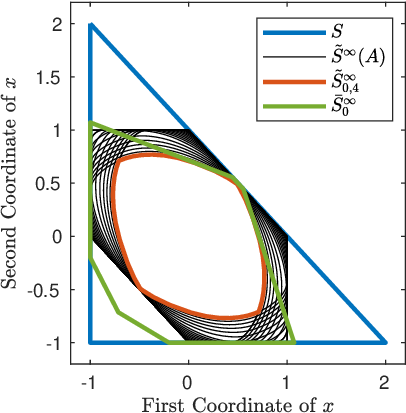}
} 
\subfigure[One-step and two-step safety sets and an inner approximation of the infinite-step safety set associated with the example in \sectionref{sec:infinite_example_two}.][t]{%
\label{fig:infinite2}
\includegraphics[width=.5\textwidth -.5em]{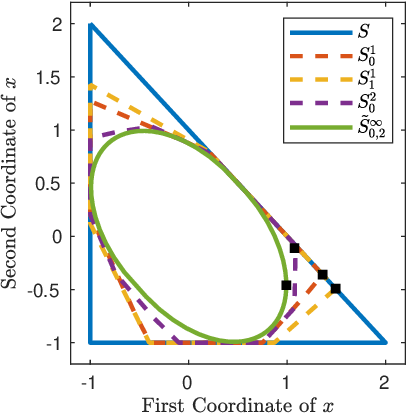}
}
}
\end{figure}

\begin{figure}[h!t]
\centering
\includegraphics[width=.5\textwidth]{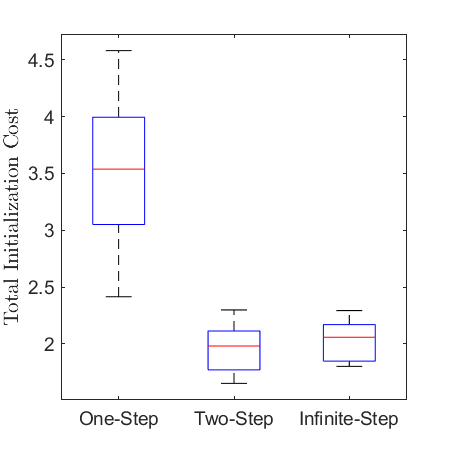}
\caption{\versionii{Total initialization cost of one, two, and infinite-step safe learning for the distribution of two-dimensional problems described at the end of \sectionref{sec:infinite_example_two}.}
}
\label{fig:boxplots2}
\end{figure}}
\section{One-Step Safe Learning of Nonlinear Systems}
\label{sec:nonlinear}

\newnewstuff{In this section, we turn our attention to} the problem of safely learning a dynamical system
of the form $x_{t+1} = \truef(x_t)$, where
\begin{align}
  \truef(x) = \trueA x + \trueg(x), \label{eq:nonlinear_model}
\end{align}
for some matrix $\trueA \in \R^{n \times n}$
and some possibly nonlinear map $\trueg :  \R^n \rightarrow \R^n$.
We take our safety region $S \subset \R^n$ to be the same as \eqref{eq:S polyhedron}.
Our initial knowledge about $\trueA, \trueg$ is membership in the sets
\begin{align*}
    U_{0,A} &\defn \left \{ A \in \R^{n \times n} \mid  \Tr(V_j^\transpose  A) \leq v_j \quad j = 1, \dots ,s \right \}, \\
    U_{0,g} &\defn \{ g : \R^n \rightarrow \R^n \mid \norm{g(x)}_{\infty} \leq \gamma \norm{x}_p^d \quad \forall x \in S \}.
\end{align*}
Here, $p \geq 1$ is either $+\infty$ or a rational number,
$\gamma$ is a given positive constant, and $d$ is a given nonnegative integer.
The use of the $\norm{\cdot}_\infty$ on $g$ in the definition of $U_{0, g}$
simplifies some of the following analysis,
though an extension to other semidefinite representable norms is possible.
\versionii{Here the parameter $d$ represents the growth rate of the nonlinearity; a larger value of $d$ corresponds to a nonlinearity that can grow faster away from the origin.}
Note that by taking $d=0$, e.g.,\ our model of uncertainty captures any
map $f$ which is bounded on $S$.
Again for simplicity, we assume a linear \newnewstuff{initialization} cost $c^\transpose  x$ for some vector $c \in \R^n$.

Our goal in this section is to demonstrate how to safely collect one-step safe trajectories for \eqref{eq:nonlinear_model} at minimum cost.
By doing so, we reduce our uncertainty on $\trueA$ and are able to fit a parametric model to $g$ that respects the constraints in $U_{0,g}$.
Having collected $k$ safe measurements $\{(x_j, y_j)\}_{j=1}^{k}$
with $y_j = \truef(x_j)$,
\newnewstuff{we can reduce our uncertainty around $\trueA$ as follows:
\begin{align} \nonumber
    U_{k,A} = \{ A \in U_0 \mid \norm{A x_j - y_j}_{\infty} \leq \gamma \norm{x_j}_p^d \quad j=1,\dots,k\}.
\end{align}}
The optimization problem to find the next cheapest one-step safe \newnewstuff{initialization point} \versionii{(i.e., the version of \eqref{eq:greedy_safe} for this specific case)} is then:
\begin{equation}\label{eq:nonlinear}
\begin{aligned}
\min_{x} \quad & c^\transpose  x &\\
\textrm{s.t.} \quad & x \in S &\\
& \newnewstuff{f(x) \in S \quad \forall~f \in \{ x \mapsto Ax+g(x) \mid A \in U_{k, A}, g \in U_{0, g}\}.}
\end{aligned}
\end{equation}
We show next that an exact reformulation of this problem can be solved in an efficient manner.
\subsection{Reformulation as a Second-Order Cone Program}
Our main result of this section is to derive a tractable reformulation of problem~\eqref{eq:nonlinear}.
\begin{theorem}\label{thm:nonlinear socp}
Problem \eqref{eq:nonlinear} can be reformulated as a second-order cone program.
\end{theorem}
\myproof{
We start by rewriting problem~\eqref{eq:nonlinear} using the definition of $S$:
\versionii{
\begin{equation}\label{eq:nonlinear bilevel}
\begin{aligned}
\min_{x} \quad & c^\transpose  x\\
\textrm{s.t.} \quad & h_i^\transpose  x \leq b_i \quad i = 1, \dots ,r \\
& \begin{bmatrix} \max_{A,g} \quad & h_i^\transpose  (A x + g(x))  \\ 
\textrm{s.t.} \quad & \Tr(V_j^\transpose  A) \leq v_j \quad j = 1, \dots ,s \\
& \norm{g(x)}_{\infty} \leq \gamma \norm{x}_p^d \quad \forall x \in S \\
& Ax_\ell + g(x_\ell) = y_\ell \quad \ell = 1,\dots,k 
\end{bmatrix} \leq b_i \quad i = 1, \dots ,r.
\end{aligned}
\end{equation}
}
Note that in the inner maximization problem in \eqref{eq:nonlinear bilevel},
the variable $x$ is fixed.
We claim that if $x \not\in \{x_1, \dots, x_k\}$, then
\versionii{
\begin{align}
&\begin{bmatrix} \max_{A,g} \quad & h_i^\transpose  (A x + g(x))  \\ 
\textrm{s.t.} \quad & \Tr(V_j^\transpose  A) \leq v_j \quad j = 1, \dots ,s \\
& \norm{g(x)}_{\infty} \leq \gamma \norm{x}_p^d \quad \forall x \in S \\
& Ax_\ell + g(x_\ell) = y_\ell \quad \ell = 1,\dots,k 
\end{bmatrix} \label{eq:three_brackets} \\ 
&\qquad = \begin{bmatrix} \max_{A,g} & h_i^\transpose  A x  \\ 
\textrm{s.t.} & \Tr(V_j^\transpose  A) \leq v_j \quad \forall j \\
& Ax_\ell + g(x_\ell) = y_\ell \quad \forall \ell \\
& \norm{g(x)}_{\infty} \leq \gamma \norm{x}_p^d \quad \forall x \in S
\end{bmatrix} + \begin{bmatrix}
\max_{A,g} & h_i^\transpose  g(x)  \\ 
\textrm{s.t.} & \Tr(V_j^\transpose  A) \leq v_j \quad \forall j \\
& Ax_\ell + g(x_\ell) = y_\ell \quad \forall \ell \\
& \norm{g(x)}_{\infty} \leq \gamma \norm{x}_p^d \quad \forall x \in S
\end{bmatrix}. \nonumber 
\end{align}
}
It is clear that the left-hand side is upper bounded by the right-hand side.
To show the reverse inequality, let
$(A_1, g_1)$ (resp.\ $(A_2, g_2)$) be feasible to the first (resp. second) problem on the right-hand side (if any of these of these problems is infeasible, then the inequality we are after is immediate).
Now let 
\versionii{
\begin{align*}
    \hat{g}_2(x) = \begin{cases}
        g_2(x) &\text{if } x \not\in \{x_1, \dots, x_k\} \\
        y_\ell - A_1 x_\ell &\text{if } x = x_\ell.
    \end{cases}
\end{align*}
}
It is straightforward to check that the pair $(A_1, \hat{g}_2)$ is feasible to the left-hand side of \eqref{eq:three_brackets},
therefore proving \eqref{eq:three_brackets}.

We now focus on reformulating each term on the right-hand side of \eqref{eq:three_brackets}, again under the assumption that \versionii{$x \not\in \{ x_1,\dots,x_k\}$.}
Using the constraint on $g$, the first term can be rewritten as follows:
\versionii{
\begin{equation}\label{eq:max A nonlinear}
\begin{aligned}
\max_{A} \quad & h_i^\transpose  A x  \\ 
\textrm{s.t.} \quad & \Tr(V_j^\transpose  A) \leq v_j \quad j = 1, \dots ,s \\
& \norm{Ax_\ell - y_\ell}_{\infty} \leq \gamma \norm{x_\ell}_p^d \quad \ell = 1,\dots,k .
\end{aligned}
\end{equation}
Note that \eqref{eq:max A nonlinear} is a linear program as it is equivalent to:
\begin{equation}\label{eq:nonlinear primal}
\begin{aligned}
\max_{A} \quad & h_i^\transpose  A x  \\ 
\textrm{s.t.} \quad & \Tr(V_j^\transpose  A) \leq v_j \quad j = 1, \dots ,s \\
& (Ax_\ell - y_\ell)_{\ell'} \leq \gamma  \norm{x_\ell}_p^d \quad \ell = 1,\dots,k \quad \ell' = 1,\dots,n\\
& -(Ax_\ell - y_\ell)_{\ell'} \leq \gamma  \norm{x_\ell}_p^d \quad \ell = 1,\dots,k \quad \ell' = 1,\dots,n.
\end{aligned}
\end{equation}
Here, the notation $(Ax_\ell - y_\ell)_{\ell'}$ represents the $\ell'$-th coordinate of the vector $(Ax_\ell - y_\ell)$.
Following the same approach as in \sectionref{sec:one-step}, we proceed by taking the dual of this linear program.
For $j=1,\dots,s$, $\ell=1,\dots,k$, and $\ell'=1,\dots,n$, let $\mu_j,\eta^+_{\ell\ell'},\eta^-_{\ell\ell'} \in \R$ be dual variables.
The dual of problem \eqref{eq:nonlinear primal} reads:
\begin{equation}
\begin{aligned}
\min_{\mu,\eta^+,\eta^-} \quad & \sum_{j=1}^s \mu_j v_j + \sum_{\ell=1}^k \sum_{\ell'=1}^n \eta^+_{\ell \ell'} (\gamma\norm{x_\ell}_p^d + (y_\ell)_{\ell'} )  + \sum_{\ell=1}^k \sum_{\ell'=1}^n \eta^-_{\ell \ell'} (\gamma\norm{x_\ell}_p^d - (y_\ell)_{\ell'} )\\ 
\textrm{s.t.} \quad & x h_i^\transpose  = \sum_{j=1}^s \mu_j V_j^\transpose  + \sum_{\ell=1}^k \sum_{\ell'=1}^n \eta^+_{\ell \ell'} x_\ell e_{\ell'}^\transpose  - \sum_{\ell=1}^k \sum_{\ell'=1}^n \eta^-_{\ell \ell'} x_\ell e_{\ell'}^\transpose  \\
& \mu \geq 0, \quad \eta^+ \geq 0, \quad \eta^- \geq 0,
\end{aligned}
\end{equation}
where $e_{\ell'}$ is the $l$-th coordinate vector.}
Now we turn our attention to the second term on the right-hand side of \eqref{eq:three_brackets}.
After eliminating the irrelevant constraints, the problem can be rewritten as:
\begin{equation}
\begin{aligned}
\max_{g} \quad &  h_i^\transpose  g(x)  \\ 
\textrm{s.t.} \quad & \norm{g(x)}_{\infty} \leq \gamma \norm{x}_p^d.
\end{aligned}
\end{equation}
Recall that the dual norm of $\norm{\cdot}_{\infty}$ is  $\norm{\cdot}_1$.
Therefore, the optimal value of this optimization problem is simply $\gamma \norm{h_i}_1 \cdot \norm{x}_p^d$.
%

Now consider the optimization problem:
\versionii{
\begin{equation}\label{eq:nonlinear socp}
\begin{aligned}
\min_{x,\mu,\eta^+,\eta^-} \quad & c^\transpose  x\\
\textrm{s.t.} \quad & h_i^\transpose  x \leq b_i \quad i = 1, \dots ,r \\
& \sum_{j=1}^s \mu_j^{(i)} v_j + \sum_{\ell=1}^k \sum_{\ell'=1}^n \eta^{+(i)}_{\ell\ell'} (\gamma\norm{x_\ell}_p^d + (y_\ell)_{\ell'} )   \\
&\qquad + \sum_{\ell=1}^k \sum_{\ell'=1}^n \eta^{-(i)}_{\ell\ell'} (\gamma\norm{x_\ell}_p^d - (y_\ell)_{\ell'} )  + \gamma \norm{h_i}_1 \cdot \norm{x}_p^d \leq b_i \quad i = 1, \dots, r \\
& x h_i^\transpose  = \sum_{j=1}^s \mu_j^{(i)} V_j^\transpose  + \sum_{\ell=1}^k \sum_{\ell'=1}^n \eta^{+(i)}_{\ell\ell'} x_\ell e_{\ell'}^\transpose  - \sum_{\ell=1}^k \sum_{\ell'=1}^n \eta^{-(i)}_{\ell\ell'} x_\ell e_{\ell'}^\transpose  \quad i = 1, \dots, r \\
& \mu \geq 0, \quad \eta^+ \geq 0, \quad \eta^- \geq 0.
\end{aligned}
\end{equation}
}If $d=0$, or if $d=1$ and $p\in\{1,+\infty\}$, then \eqref{eq:nonlinear socp} is a linear program.
Otherwise, the rationality of $p$ ensures that $\norm{x}^d_p$ is \emph{second-order cone representable} \citetext{see \citealp[Sect.~2.3]{bental_nemirovski}; \citealp[Sect.~2.5]{LOBO1998193}}.
This means that \eqref{eq:nonlinear socp} is indeed a second-order cone program.


Let $F \subset \R^n$ denote 
the projection of the feasible set of \eqref{eq:nonlinear socp} to $x$-space.
We claim that the feasible set of \eqref{eq:nonlinear} equals $F \cup \{ x_1, \dots, x_k\}$.
Indeed, since the vectors $x_k$ are one-step safe \newnewstuff{initialization points}, 
we have that $x_k \in S$ and $y_k \in S$.
This implies that $x_k$ is feasible to \eqref{eq:nonlinear}.
Furthermore, for $x \in F \setminus \{ x_1, \dots, x_k \}$,
we have shown that $x$ satisfies the constraints of \eqref{eq:nonlinear bilevel}
if and only if $x$ satisfies the constraints of \eqref{eq:nonlinear socp}.

Therefore, optimizing an objective function over the feasible set of \eqref{eq:nonlinear}
is equivalent to optimizing the same objective function over $F \cup \{x_1, \dots, x_k\}$.
%
}
\versionii{
\begin{remark}
We note that problem \eqref{eq:nonlinear socp}
can be modified so that
one-step safety is achieved in the presence of 
bounded disturbances.
That is, suppose that the dynamics were governed by
\[
x_{t+1} = \trueA x_t + \trueg(x_t) + w_t,
\]
where $w_t$ represents some potentially adversarial disturbance, $\trueA \in U_{0,A}$, and $\trueg \in U_{0,g}$.
We can still give an exact second-order cone programming-based characterization of the one-step safety set in the case when we have $\|w_t\| \leq W_t$, where $\norm{\cdot}$ is any norm whose unit ball is a polytope and $W_t$ is a given scalar.
For example, if $\norm{\cdot}$ is the infinity norm, the set of one-step safe initialization points after observing $k$ measurements from the disturbed dynamics is the projection to $x$-space of the feasible set of the following second-order cone program:
\begin{align*}
\min_{x,\mu,\eta^+,\eta^-} \quad & c^\transpose  x\\
\textrm{s.t.} \quad & h_i^\transpose  x \leq b_i \quad i = 1, \dots ,r \\
& \sum_{j=1}^s \mu_j^{(i)} v_j + \sum_{\ell=1}^k \sum_{\ell'=1}^n \eta^{+(i)}_{\ell\ell'} (\gamma\norm{x_\ell}_p^d + W_\ell + (y_\ell)_{\ell'} )   \\
&\qquad + \sum_{\ell=1}^k \sum_{\ell'=1}^n \eta^{-(i)}_{\ell\ell'} (\gamma\norm{x_\ell}_p^d + W_\ell - (y_\ell)_{\ell'} )  + \gamma \norm{h_i}_1 \cdot \norm{x}_p^d + W_{k+1} \|h_i\|_1 \leq b_i \quad i = 1, \dots, r \\
& x h_i^\transpose  = \sum_j \mu_j^{(i)} V_j^\transpose  + \sum_{\ell=1}^k \sum_{\ell'=1}^n \eta^{+(i)}_{\ell\ell'} x_\ell e_{\ell'}^\transpose  - \sum_{\ell=1}^k \sum_{\ell'=1}^n \eta^{-(i)}_{\ell\ell'} x_\ell e_{\ell'}^\transpose  \quad i = 1, \dots, r \\
& \mu \geq 0, \quad \eta^+ \geq 0, \quad \eta^- \geq 0
\end{align*}
where the input to the problem is the descriptions of $S$ and $U_0$ ($h_i,b_i$ and $V_j,v_j$) and the measurements $(x_\ell,y_\ell)$ and we have introduced dual variables $\mu_j^{(i)}$ for $i=1,\dots,r$, $j=1,\dots,s$ and $\eta^{+(i)}_{\ell\ell'},\eta^{-(i)}_{\ell\ell'}$ for $i=1,\dots,r$, $\ell=1,\dots,k$ and $\ell' = 1,\dots,n$.
\end{remark}
\begin{remark}
We note that problem \eqref{eq:nonlinear socp}
can be modified so that
one-step safety is achieved in the presence of 
bounded measurement noise.
That is, suppose that instead of directly observing $y_k = \trueA x_k + \trueg(x_k)$, we observe
\[
y_k = \trueA x_k + \trueg(x_k) + z_k,
\]
where $z_k$ represents the noise in the measurement, $\trueA \in U_{0,A}$, and $\trueg \in U_{0,g}$.
We can still give an exact second-order cone programming-based characterization of the one-step safety set in the case when we have $\|z_k\| \leq Z_k$, where $\norm{\cdot}$ is any norm whose unit ball is a polytope and $Z_k$ is a given scalar.
For example, if $\norm{\cdot}$ is the infinity norm, the set of one-step safe initialization points after observing $k$ noisy measurements is the projection to $x$-space of the feasible set of the following second-order cone program:
\begin{align*}
\min_{x,\mu,\eta^+,\eta^-} \quad & c^\transpose  x\\
\textrm{s.t.} \quad & h_i^\transpose  x \leq b_i \quad i = 1, \dots ,r \\
& \sum_{j=1}^s \mu_j^{(i)} v_j + \sum_{\ell=1}^k \sum_{\ell'=1}^n \eta^{+(i)}_{\ell\ell'} (\gamma\norm{x_\ell}_p^d + Z_\ell + (y_\ell)_{\ell'} )   \\
&\qquad + \sum_{\ell=1}^k \sum_{\ell'=1}^n \eta^{-(i)}_{\ell\ell'} (\gamma\norm{x_\ell}_p^d + Z_\ell - (y_\ell)_{\ell'} )  + \gamma \norm{h_i}_1 \cdot \norm{x}_p^d \leq b_i \quad i = 1, \dots, r \\
& x h_i^\transpose  = \sum_j \mu_j^{(i)} V_j^\transpose  + \sum_{\ell=1}^k \sum_{\ell'=1}^n \eta^{+(i)}_{\ell\ell'} x_\ell e_{\ell'}^\transpose  - \sum_{\ell=1}^k \sum_{\ell'=1}^n \eta^{-(i)}_{\ell\ell'} x_\ell e_{\ell'}^\transpose  \quad i = 1, \dots, r \\
& \mu \geq 0, \quad \eta^+ \geq 0, \quad \eta^- \geq 0
\end{align*}
where the input to the problem is the descriptions of $S$ and $U_0$ ($h_i,b_i$ and $V_j,v_j$) and the measurements $(x_\ell,y_\ell)$ and we have introduced dual variables $\mu_j^{(i)}$ for $i=1,\dots,r$, $j=1,\dots,s$ and $\eta^{+(i)}_{\ell\ell'},\eta^{-(i)}_{\ell\ell'}$ for $i=1,\dots,r$, $\ell=1,\dots,k$ and $\ell' = 1,\dots,n$.
\end{remark}

We note that we can also exactly characterize one-step safety sets in the presence of both disturbances and noisy measurements.
}

\subsection{Numerical Example}\label{sec:nonlinear example}
We present a numerical example with $n=4$.
We take 
\begin{align*}
      S &= \{ x \in \R^4 \mid |x_i| \leq 1, i=1,\dots,4 \}, \\
    U_{0,A} &= \{ A \in \R^{4 \times 4} \mid -4 \leq A_{ij} \leq 8, i = 1, \dots, 4, j = 1, \dots, 4 \}, \\
    U_{0,g} &= \{ g : \R^4 \rightarrow \R^4 \mid \norm{g(x)}_\infty \leq \gamma \quad \forall x \in S \}.
\end{align*}
In \figureref{fig:nonlinear gamma}, we plot $S^1_0$ (the one-step safety region without any measurements) projected to the first two dimensions of $x$ for $\gamma \in \{ 0,0.4,0.8 \}$.
As expected, larger values of $\gamma$ result in smaller one-step safety regions.

For our next experiment,
we choose the matrix $\trueA$ in \eqref{eq:nonlinear_model} to be the same matrix used in the example in \sectionref{sec:one-step example}.
We let $\gamma = 0.1$, and
\[ \trueg (x) = \frac{\gamma}{2} 
\left( x_2^2 - x_3 x_4, \quad
\sqrt{x_1^4 + x_3^4}, \quad
x_3 \sin^2(x_1), \quad
\sin^2(x_2) \right) ^\transpose  \in U_{0,g}. \]

Since the true system is not linear, we cannot hope to learn the dynamics in $n$ steps
as we did in the linear case.
We instead pick $30$ one-step safe points $x_1,\dots,x_{30}$ (by sequentially solving the second-order cone program from \theoremref{thm:nonlinear socp}) and observe $y_k = \truef(x_k)$ for each $k=1,\dots,30$.
In order to encourage exploration of the state space, we optimize in random directions in every iteration (instead of optimizing the same cost function throughout the process).
In \figureref{fig:nonlinear learning}, we plot $S^1_k$ (the one-step safety region after $k$ measurements) projected to the first two dimensions of $x$ for $k=0,\dots,30$.
\newnewstuff{Note that $S^1_k$ is the projection of the feasible set of \eqref{eq:nonlinear socp} to $x$-space.}
We also plot the projection of $S^1_\gamma (\trueA)$, which we define
as the set of one-step safe points if we knew $\trueA$, but not $\trueg$
\[ S^1_\gamma (\trueA) \defn \{x \in S \mid \trueA x + g(x) \in S \quad \forall g \in U_{0,g} \}.\]
Note that this set is an outer approximation to $S^1_k$.
Here we see that $S^1_k$ comes close to $S^1_\gamma (\trueA)$ already in thirty iterations.

\begin{figure}
\figureconts
{fig:nonlinear}
{\caption{One-step safe learning of a nonlinear system associated with the example in \sectionref{sec:nonlinear example}.}}
{%
\subfigure[Dependence of $S^1_0$ on $\gamma$.][c]{%
\label{fig:nonlinear gamma}
\includegraphics[width=.5\textwidth -.5em]{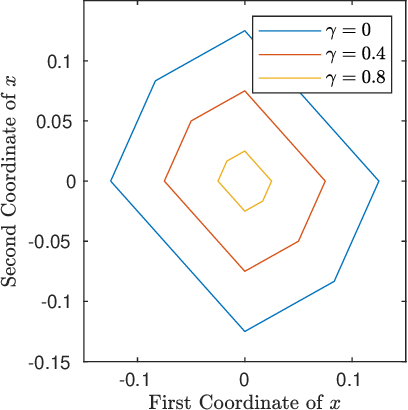}
} 
\subfigure[$S^1_k$ grows with $k$.][c]{%
\label{fig:nonlinear learning}
\includegraphics[width=.5\textwidth -.5em]{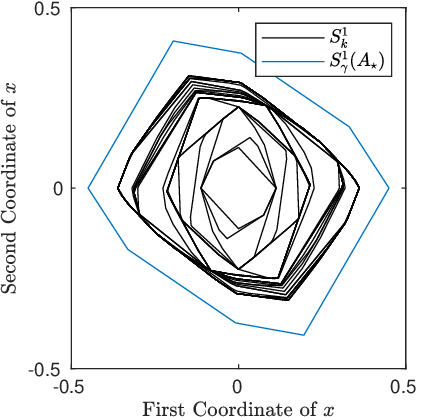}
}
}
\end{figure}

Finally, we undertake the task of learning the unknown nonlinear dynamics.
We only use information from our first $8$ data points in order to make the fitting task more challenging.
We fit a function of the form
\[\hat{f}(x) = \hat{A}x + \hat{g}(x) , \]
where $\hat{A} \in \R^{4 \times 4}$ and each entry of $\hat{g} : \R^4 \rightarrow \R^4$ 
is a homogeneous quadratic function of $x$.
Our regression is done by minimizing the least-squares loss function
\[ L(\hat{f}) =  \sum_{k=1}^8 \norm{\hat{f}(x_k) - y_k}^2. \]
We train two models.
The first model, $\hat{f}_{\mathrm{ls}}$, minimizes the least-squares loss with no constraints.
The second model, $\hat{f}_{\mathrm{SOS}}$, minimizes the least-squares loss
subject to the constraints that $\hat{A} \in U_{0,A}$, $\norm{\hat{A}x_k - y_k}_{\infty} \leq \gamma$ for $k=1,\dots,8$, and $\hat{g} \in U_{0,g}$.
%
%
The constraint that $\hat{g} \in U_{0,g}$ is imposed via sum of squares constraints \citetext{see, e.g., \citealp{pablothesis,doi:10.1137/20M1388644} for details}.
More specifically, we require that for $j=1, \dots, 4$,
\[
\gamma \pm \hat{g}_j(x) = \sigma_0^{j,\pm}(x) + \sum_{i=1}^r \sigma_i^{j,\pm}(x) (b_i - h_i^\transpose  x) \quad \forall x \in \R^4.
\]
Here, $\hat{g}_j(x)$ is the $j$-th entry of the vector $\hat{g}(x)$, and the functions $\sigma_i^{j,\pm}$, for $i=0, \dots, r$ and $j=1,\dots,4$, are sum of squares quadratic functions of $x$.
These constraints can be imposed by semidefinite programming.

%
We sample test points $z_1,\dots,z_{1000}$ uniformly at random in $S$ in order to estimate the 
generalization error.
The root-mean-square error (RMSE) is computed as
\[ \text{RMSE} (\hat{f}) = \sqrt{ \frac{1}{1000} \sum_{j=1}^{1000} \norm{ \hat{f}(z_i) - \truef(z_i)}^2 }. \]
The $\text{RMSE}(\hat{f}_{\mathrm{SOS}})$ of the constrained model is $0.0851$ and the $\text{RMSE}(\hat{f}_{\mathrm{ls}})$ of the unconstrained model is $0.2567$.
We see that imposing prior knowledge with sum of squares constraints
results in a significantly better fit.
\newstuff{\section{Infinite-Step Safe Learning of Nonlinear Systems}
\label{sec:inf_nonlinear}

\newnewstuff{In our final technical section,} we consider the problem of safely learning a dynamical system
of the same form as in \Cref{sec:nonlinear}, i.e.,
\begin{equation}\label{eq:nonlinear_model_inf}
x_{t+1} = \trueA x_t + \trueg (x_t)
\end{equation}
involving some matrix $\trueA \in \R^{n \times n}$ and some possibly nonlinear map $\trueg :  \R^n \rightarrow \R^n$.
We take the safety region $S \subset \R^n$ to be the same as \eqref{eq:S polyhedron}.
We take our initial knowledge about $\trueA, \trueg$ to be membership in the following sets:
\begin{align*}
    U_{0,A} &\defn \left \{ A \in \R^{n \times n} \mid  \Tr(V_j^\transpose  A) \leq v_j \quad j = 1, \dots ,s \right \}, \\
    U_{0,g} &\defn \{ g : \R^n \rightarrow \R^n \mid \norm{g(x)} \leq \gamma \norm{x} \quad \forall x \in S \},
\end{align*}
\newnewstuff{where $\gamma$ is a given nonnegative constant.}
In the definition of $U_{0,g},$ it is convenient to use the $\ell_2$ norm because of some technical reasons that will become clear in the proofs of the statements in this section. Again for simplicity, we assume that for some vector $c \in \R^n$, initializing the system at a point $x\in S$ comes at the cost $c^\transpose  x$.

Just as in \sectionref{sec:linear}, the notion of safety in this section is that of infinite-step safety; i.e., we can only initialize the system at points whose entire trajectory will remain in $S$ under all dynamical systems consistent with the information at hand.
By observing these trajectories, we can safely reduce our uncertainty on $\trueA$ and fit a parametric model to $\trueg$ that respects the constraints in $U_{0,g}$ (in the same way that we did in~\sectionref{sec:nonlinear}).
Unlike the setting of infinite-step safe learning of linear systems (\sectionref{sec:linear}), it might be useful to observe trajectories of length greater than $n$.
Assuming there is some limitation on \newnewstuff{memory}, it is sensible to truncate each trajectory after some time.
Suppose that we have collected $k$ trajectories and that the $\ell$th trajectory is of length $n_{\ell}$.
Let $y_{t,\ell} \in \Rn$ be the $t$th observed vector in the $\ell$th trajectory with $y_{0,\ell}$ being the trajectory's initialization.
With these observations, we can reduce our uncertainty around $\trueA$ as follows:
\begin{align} \nonumber
    U_{k,A} = \{ A \in U_0 \mid \norm{A y_{t-1,\ell} - y_{t,\ell}} \leq \gamma \norm{y_{t-1,\ell}} \quad t=1,\dots,n_{\ell}\quad \ell=1,\dots,k\}.
\end{align}
Since $U_{0,g}$ contains the zero map (i.e., one possibility for the unknown dynamics is $x_{t+1}~=~A x_t$ for some matrix $A\in U_{0,A}$), the problem considered in this section is at least as hard as that of \sectionref{sec:linear}. Therefore, we make the same assumptions as \sectionref{sec:linear} to have a full-dimensional infinite-step safety set. 
In particular, we assume that the origin is in the interior of $S$, which means that $S$ can be described as \eqref{eq:S_polyhedron_all_ones}, and that \newnewstuff{all matrices in $U_{0,A}$ are stable}.
Having collected $k$ infinite-step safe trajectories, the optimization problem we are interested in solving to find the next cheapest infinite-step safe \newnewstuff{initialization point} is:
\begin{equation}\label{eq:inf_nonlinear}
\begin{aligned}
\min_{x} \quad & c^\transpose  x &\\
\textrm{s.t.} \quad & x \in S &\\
& f^t(x) \in S \quad \forall~f \in \{ x \mapsto Ax+g(x) \mid A \in U_{k, A}, g \in U_{0, g}\} \quad t=1,2,\dots.
\end{aligned}
\end{equation}
In keeping with the naming conventions of this work, we refer to the feasible region of \eqref{eq:inf_nonlinear} as $S_k^{\infty}$, and if $U_{0,A}$ is a single matrix $A$, we call it $S^{\infty}_{\gamma}(A)$.
We can now present the main theorem of this section, which enables us to find infinite-step safe \newnewstuff{initialization} points \versionii{(i.e., the version of \eqref{eq:greedy_safe} for this specific case)}.
%
%
%
\begin{theorem} \label{thm:inf_nonlinear}
Let the polyhedron $S \subseteq \R^n$
be as in \eqref{eq:S_polyhedron_all_ones}, and the polyhedron
$U_{0,A} \subseteq \R^{n \times n}$ be
as in~\eqref{eq:U_0 polyhedron}.
For $\ell=1,\dots,k$ and $t = 0,\dots,n_{\ell}$, let $y_{t,\ell}$ be the $t$th vector in the $\ell$th observed trajectory.
For an even integer $d$, let
$\tilde{S}_{k,d}^{\infty}$ be the projection to $x$-space of the feasible region of the following optimization problem:

\begin{subequations}\label{eq:inf_nonlinear_sdp}
\begin{align}
& \min_{x,Q,M_j,M_{t\ell},\hat{M}_j,\hat{M}_{t\ell},\sigma_{ij},\sigma_{it\ell},\varepsilon,\lambda} \quad  c^\transpose  x \tag{\ref{eq:inf_nonlinear_sdp}}\\
& \mathrm{s.t.} \quad   \begin{bmatrix}
 Q(A) - A Q(A) A^\transpose  &  -A Q(A) \\
 -Q(A) A^\transpose  & -Q(A)
\end{bmatrix} -  \lambda \begin{bmatrix}
 \gamma^2 I &  0 \\
 0 & -I
\end{bmatrix} = \varepsilon I + M_0(A) + \sum_{j=1}^s M_j(A) (v_j - \Tr(V_j^\transpose  A)) \label{subeq:stabil n} \\
&\quad \quad \quad \quad \quad \quad + \sum_{\ell=1}^{k} \sum_{t=1}^{n_{\ell}} M_{t \ell}(A) (\gamma^2\norm{y_{t-1,\ell}}^2 - \norm{A y_{t-1,\ell} - y_{t,\ell}}^2) \quad \forall A \in \Rnn \nonumber \\
& 1 - h_i^\transpose  Q(A) h_i = \sigma_{i0}(A) + \sum_{j=1}^s \sigma_{ij}(A) (v_j - \Tr(V_j^\transpose  A)) \label{subeq:polar dual n} \\
&\quad \quad \quad \quad \quad \quad
 + \sum_{\ell=1}^{k} \sum_{t=1}^{n_{\ell}} \sigma_{i t \ell}(A) (\gamma^2\norm{y_{t-1,\ell}}^2 - \norm{A y_{t-1,\ell} - y_{t,\ell}}^2)
\quad i = 1, \dots, r \quad \forall A \in \Rnn \nonumber \\
& \begin{bmatrix}
 Q(A) &  x \\
 x^\transpose  & 1 
\end{bmatrix} = \hat{M}_0(A) + \sum_{j=1}^s \hat{M}_j(A) (v_j - \Tr(V_j^\transpose  A)) \label{subeq:ellipsoid n} \\
&\quad \quad \quad \quad \quad \quad + \sum_{\ell=1}^{k} \sum_{t=1}^{n_{\ell}} \hat{M}_{t \ell}(A) (\gamma^2\norm{y_{t-1,\ell}}^2 - \norm{A y_{t-1,\ell} - y_{t,\ell}}^2)
\quad \forall A \in \Rnn \nonumber \\
& \varepsilon > 0\label{subeq:strict n} \\
& \lambda \geq 0,
\end{align}
\versionii{\begin{itemize}
\item where $Q(A)$ is an $n \times n$ SOS matrix with degree at most $d$, 
\item $M_j(A)$ are $2n \times 2n$ SOS matrices with degree at most $d$ for $j=0,\dots,s$,
\item $M_{t \ell}(A)$ are $2n \times 2n$ SOS matrices with degree at most $d$ for $\ell=1,\dots, k, \:t=1,\dots,n_{\ell}$, 
\item $\hat{M}_j(A)$ are $(n+1) \times (n+1)$ SOS matrices with degree at most $d$ for $j=0,\dots,s$,
\item $\hat{M}_{t \ell}(A)$ are $(n+1) \times (n+1)$ SOS matrices with degree at most $d$ for $\ell=1,\dots, k, \: t=1,\dots,n_{\ell}$, 
\item $\sigma_{ij}(A)$ are SOS polynomials with degree at most $d$ for $i=1,\dots,r, \: j=0,\dots,s$,
\item and $\sigma_{i t \ell}(A)$ are SOS polynomials with degree at most $d$ for $i=1,\dots,r,\: \ell=1,\dots, k, \: t=1,\dots,n_{\ell}$.
\end{itemize}}
\end{subequations}
\newnewstuff{Then,
\begin{enumerate}
    \item[(i)] The program \eqref{eq:inf_nonlinear_sdp} can be reformulated as a semidefinite program of size polynomial in the size of the input ($S$, $U_0$, $\{y_{t,\ell}\}$, $c$, $\gamma$).
    \item[(ii)] We have $\tilde{S}_{k,d}^{\infty} \subseteq S_k^\infty$ (i.e, any vector $x$ feasible to this semidefinite program is infinite-step safe).
    \item[(iii)] Furthermore, if $U_{0,A}$ is compact and contains only stable matrices, and if $\gamma$ is smaller than some positive \newnewstuff{threshold} depending on $U_{0,A}$, then, for large enough $d$, the set $\tilde{S}_{k,d}^{\infty}$ is full-dimensional.
\end{enumerate}}
\end{theorem}
\newnewstuff{In words, \theoremref{thm:inf_nonlinear} allows us to optimize \versionii{the} initialization cost over semidefinite representable subsets of the set of points which are infinite-step safe under all nonlinear dynamics consistent with information at hand.
While the theorem guarantees full-dimensionality of these subsets for sufficiently small $\gamma$ and large $d$, in our experience, even when $\gamma$ is relatively large, small values of $d$ suffice for safe learning; see \sectionref{sec:nonlinear example}.}

\versionii{
\begin{remark}
We note that problem \eqref{eq:inf_nonlinear_sdp}
can be modified so that
infinite-step safety is achieved in the presence of 
bounded measurement noise.
That is, suppose that instead of directly observing $y_{t,\ell} = \trueA y_{t-1,\ell} + \trueg(y_{t-1,\ell})$, we observe
\[
\hat{y}_{t,\ell} = y_{t,\ell} + z_{t,\ell},
\]
where $z_{t,\ell}$ represents the noise in the measurement, $\trueA \in U_{0,A}$, and $\trueg \in U_{0,g}$.
We can still give an SOS programming-based inner approximation of the infinite-step safety set in the case when we have $\|z_{t,\ell}\| \leq Z_{t,\ell}$, where $\norm{\cdot}$ is, e.g., any polynomial norm (see \cite{polynomial_norms} for a definition) or any norm whose unit ball is a polytope, and $Z_{t,\ell}$ is a given scalar.
To see this, observe that the vectors $\hat{y}_{t,\ell}$ will satisfy
\[
\hat{y}_{t,\ell} = \trueA (\hat{y}_{t-1,\ell} - z_{t-1,\ell}) + \trueg(y_{t-1,\ell}) + z_{t,\ell}.
\]
Now for example, if $\norm{\cdot}$ is the Euclidean norm, and if we have $\max_{A \in U_{k,A}} \|A\| \leq M_k$ for some constant $M_k$ (computed, e.g., by a semidefinite relaxation), we can derive the following inequality:
\begin{align*}
\|\hat{y}_{t,\ell} - \trueA \hat{y}_{t-1,\ell} \| &\leq \|\trueA\| \|z_{t-1,\ell}\| + \|\trueg(y_{t-1,\ell})\| + \|z_{t,\ell}\| \\
&\leq M_k Z_{t-1,\ell} + \gamma \|y_{t-1,\ell}\| + Z_{t,\ell}\\
&\leq M_k Z_{t-1,\ell} + \gamma (\|\hat{y}_{t-1,\ell}\| + Z_{t-1,\ell}) + Z_{t,\ell}.
\end{align*}

We can then adapt the methodology of \theoremref{thm:inf_nonlinear} by multiplying the SOS matrices and polynomials in semidefinite program \eqref{eq:inf_nonlinear_sdp} by the polynomials\[\left \{A \mapsto (M_k Z_{t-1,\ell} + \gamma (\|\hat{y}_{t-1,\ell}\| + Z_{t-1,\ell}) + Z_{t,\ell})^2 - \left\| \hat{y}_{t,\ell} - A \hat{y}_{t-1,\ell} \right\|^2 \right \}_{t,\ell}\]
instead of $\{A \mapsto \gamma^2\norm{y_{t-1,\ell}}^2 - \norm{A y_{t-1,\ell} - y_{t,\ell}}^2\}_{t,\ell}$.
For this modified SOS program, all claims of \theoremref{thm:inf_nonlinear} hold.
\end{remark}

}

Before we present a proof of \theoremref{thm:inf_nonlinear}, we introduce a ``nonlinear version'' of \eqref{eq:RDO}, which applies to the case of a fixed matrix $A$:
\begin{equation}\label{eq:nonlinear_RDO}
\begin{aligned}
\min_{x \in \Rn,Q \in \Snn,\lambda \in \R} \quad & c^\transpose  x\\
\textrm{s.t.} \quad & Q \succ 0 \\
& \begin{bmatrix}
 Q - A Q A^\transpose  &  -A Q \\
 -Q A^\transpose  & -Q
\end{bmatrix} \succeq  \lambda \begin{bmatrix}
 \gamma^2 I &  0 \\
 0 & -I
\end{bmatrix} \\
& h_i^\transpose  Q h_i \leq 1 \quad i = 1, \dots, r \\
& \begin{bmatrix}
 Q &  x \\
 x^\transpose  & 1 
\end{bmatrix} \succeq 0 \\
& \lambda \geq 0.
\end{aligned}
\end{equation}
We now prove a nonlinear version of \lemmaref{lem:RDO}.
Recall the definition of the set $S^{\infty}_{\gamma}(A)$ from the paragraph after \eqref{eq:inf_nonlinear}.
\begin{lemma}\label{lem:nonlinear_RDO}
Let $\tilde{S}^{\infty}_{\gamma}(A)$ be the projection to $x$-space of the feasible region of \eqref{eq:nonlinear_RDO}.
Then, we have $\tilde{S}^{\infty}_{\gamma}(A) \subseteq S^{\infty}_{\gamma}(A)$.
\end{lemma}
\begin{proof}
Let $x,Q,$ and $\lambda$ be feasible to \eqref{eq:nonlinear_RDO}.
As in the proof of \lemmaref{lem:RDO}, the constraints $h_i^\transpose  Q h_i \leq 1$, $i = 1, \dots, r$, imply \[\{x \mid x^\transpose  Q^{-1} x \leq 1\} \subseteq S,\] and the constraint $\begin{bmatrix}
 Q &  x \\
 x^\transpose  & 1 
\end{bmatrix} \succeq 0$ implies $x^\transpose  Q^{-1} x \leq 1$.
Thus $x$ is in the safety region $S$.
\newnewstuff{
To show that the trajectory remains safe for all time, it suffices to show that the set $\{x \mid x^\transpose  Q^{-1} x \leq 1\}$ is invariant under all valid dynamics, i.e., for any vector $\bar{x}$ in this set and any vector $g(\bar{x})$ satisfying $\norm{g(\bar{x})} \leq \gamma \norm{\bar{x}}$, the vector $A\bar{x} + g(\bar{x})$ is also in the set.

Let $B \in \Rnn$ be an arbitrary matrix and let $\norm{B}$ denote its spectral norm.
We first claim that if $\norm{B} \leq \gamma$, then $(A+B)^\transpose  Q^{-1} (A+B) \preceq Q^{-1}$.
Fix an arbitrary vector $\hat{x}$ and let $\hat{y} = B^\transpose  \hat{x}$.
By the bound on the spectral norm of $B$, we have $\norm{\hat{y}} \leq \gamma \norm{\hat{x}}$.
By the second linear matrix inequality in \eqref{eq:nonlinear_RDO}, we have
\[
\begin{bmatrix}
 \hat{x} \\
 \hat{y}
\end{bmatrix}^\transpose  \begin{bmatrix}
 Q - A Q A^\transpose  &  -A Q \\
 -Q A^\transpose  & -Q
\end{bmatrix} \begin{bmatrix}
 \hat{x} \\
 \hat{y}
\end{bmatrix} \geq \lambda \begin{bmatrix}
 \hat{x} \\
 \hat{y}
\end{bmatrix}^\transpose  \begin{bmatrix}
 \gamma^2 I &  0 \\
 0 & -I
\end{bmatrix} \begin{bmatrix}
 \hat{x} \\
 \hat{y}
\end{bmatrix}.
\]
Rearranging, we have
\[
\hat{x}^\transpose  Q \hat{x} - (\hat{x}^\transpose  A Q A^\transpose  \hat{x} + 2\hat{x}^\transpose AQ\hat{y} + \hat{y}^\transpose Q\hat{y}) \geq 
\lambda (\gamma^2 \hat{x}^\transpose  \hat{x} - \hat{y}^\transpose  \hat{y}).
\]
Since $\lambda \geq 0$ and $\norm{\hat{y}} \leq \gamma \norm{\hat{x}}$, we have $\hat{x}^\transpose  Q \hat{x} \geq \hat{x}^\transpose  A Q A^\transpose  \hat{x} + 2\hat{x}^\transpose AQ\hat{y} + \hat{y}^\transpose Q\hat{y},$ which implies \[\hat{x}^\transpose (A+B) Q (A+B)^\transpose \hat{x} \leq \hat{x}^\transpose  Q \hat{x}.\] This is equivalent to the claimed linear matrix inequality by two applications of the Schur complement lemma.

Now we show invariance of the set $\{x \mid x^\transpose  Q^{-1} x \leq 1\}$ under all valid dynamics.
Let $\bar{x}$ be any vector such that $\bar{x}^\transpose  Q^{-1} \bar{x} \leq 1$, and let $B_{\bar{x}} \defn \frac{g(\bar{x})\bar{x}^\transpose }{\norm{\bar{x}}^2}$.
From the definition of $U_{0,g}$, we have $\norm{B_{\bar{x}}} \leq \gamma$.
By the claim we established above, we have $\bar{x}^\transpose (A+B_{\bar{x}})^\transpose Q^{-1}(A+B_{\bar{x}})\bar{x} \leq \bar{x}^\transpose  Q^{-1} \bar{x}$.
Since $(A+B_{\bar{x}})\bar{x} = A\bar{x} + g(\bar{x})$, it follows that
\[ \big( A\bar{x} + g(\bar{x}) \big) ^\transpose  Q^{-1} \big( A\bar{x} + g(\bar{x}) \big) \leq \bar{x}^\transpose  Q^{-1} \bar{x} \leq 1.\]
Thus, $A\bar{x} + g(\bar{x})$ is in the set $\{x \mid x^\transpose  Q^{-1} x \leq 1\}$ as desired.
}
\end{proof}
We can now present the proof of the main result of this section.
\begin{proof}[Proof of \theoremref{thm:inf_nonlinear}]

(i)
We make a similar argument as in the proof of (i) in \theoremref{thm:linear}.
Recall that for any fixed degree $d$, the SOS constraints can be reformulated as semidefinite programming constraints of size polynomial in $n$.
The ``$\forall A$'' constraints can be imposed by coefficient matching via a number of linear equations bounded by a polynomial function of the size of the input $(S,U_{0,A},\{y_{t,\ell}\},c,\gamma)$.
The constraint that $\varepsilon > 0$ can be rewritten as the constraint $\begin{bmatrix} \varepsilon &1 \\ 1& \delta \end{bmatrix} \succeq 0 $ for a new variable $\delta$.
Therefore, \newnewstuff{for any fixed degree $d$,} \eqref{eq:inf_linear_sdp} is a semidefinite program of size polynomial in the size of the input $(S,U_{0,A},\{y_{t,\ell}\},c,\gamma)$.

(ii)
Let $(x,Q,M_j,M_{t\ell},\hat{M}_j,\hat{M}_{t\ell},\sigma_{ij},\sigma_{it\ell},\varepsilon,\lambda)$ be feasible to \eqref{eq:inf_nonlinear_sdp}.
Then, it is straightforward to check that 
for every $A \in U_{k,A}$, the tuple $(x, Q(A),\lambda)$ satisfies the following constraints:
\begin{equation}
\begin{aligned}
& Q(A) \succ 0 \\
& \begin{bmatrix}
 Q(A) - A Q(A) A^\transpose  &  -A Q(A) \\
 -Q(A) A^\transpose  & -Q(A)
\end{bmatrix} \succeq  \lambda \begin{bmatrix}
 \gamma^2 I &  0 \\
 0 & -I
\end{bmatrix}\\
& h_i^\transpose  Q(A) h_i \leq 1 \quad i = 1, \dots, r \\
& \begin{bmatrix}
 Q(A) &  x \\
 x^\transpose  & 1 
\end{bmatrix} \succeq 0,
\end{aligned}
\end{equation}
and therefore, by \lemmaref{lem:nonlinear_RDO}, we have $x \in \tilde{S}^\infty_{\gamma} (A) \subseteq S^\infty_{\gamma} (A)$.
Hence,
\[
x \in \bigcap_{A \in U_{k,A}} S^\infty_{\gamma} (A) = S_k^\infty,
\]
implying that $\tilde{S}_{k,d}^\infty \subseteq S_k^\infty$.

(iii)
Suppose that $U_{0,A}$ is compact and contains only stable matrices.
By \lemmaref{lem:polynomial_lyapunov} and the arguments in the beginning of the proof of (iii) in \theoremref{thm:linear}, we can find an SOS matrix $\hat{Q}(A)$ satisfying
\begin{align*}
\hat{Q}(A) &\succ 0 \quad \forall A \in U_{0,A} \\
\hat{Q}(A) &\succ A \hat{Q}(A) A^\transpose   \quad \forall A \in U_{0,A}.
\end{align*}
In particular, there must be a positive constant $\delta$ such that $\hat{Q}(A) - A\hat{Q}(A)A^\transpose  \succeq \delta I$ for all $A\in U_{0,A}$.
Let $\hat{\lambda} \defn 1 + \max_{A \in U_{0,A}} \norm{\hat{Q}(A)A^\transpose  (\hat{Q}(A)-A\hat{Q}(A)A^\transpose  - \frac{\delta}{2}I)^{-1}A\hat{Q}(A) + \hat{Q}(A)}$, and take $\gamma$ to be a positive scalar less than $\sqrt{\frac{\delta}{2\lambda}}$.
By the Schur complement lemma and the fact that $\frac{\delta}{2} > \hat{\lambda} \gamma^2$, it follows that for every $A \in U_{0,A}$,
\[
\begin{bmatrix}
 \hat{Q}(A) - A \hat{Q}(A) A^\transpose  &  -A \hat{Q}(A) \\
 -\hat{Q}(A) A^\transpose  & -\hat{Q}(A)
\end{bmatrix} \succeq \begin{bmatrix}
 \frac{\delta}{2}I &  0 \\
 0 & -(\hat{\lambda}-1) I
\end{bmatrix} \succ \hat{\lambda} \begin{bmatrix}
 \gamma^2 I &  0 \\
 0 & - I
\end{bmatrix}.
\]
Since $U_{k,A}$ is compact, there exists a scalar $\alpha > 0$ satisfying $\max_{i\in\{1,\dots,r\},A \in U_{k,A}} h_i^\transpose \hat{Q}(A)h_i <~\alpha$.
Let us define $Q \defn \hat{Q}/\alpha$ and $\lambda \defn \hat{\lambda}/\alpha$.
Since $Q(A) \succ 0$ for all $A \in U_{k,A}$, we can find a scalar $\beta > 0$ such that $\begin{bmatrix}
 Q(A) - \beta I &  0 \\
 0 & \frac{1}{2} 
\end{bmatrix} \succ~0$ for all $A \in U_{k,A}$.
Summarizing, so far we have:
\begin{equation}\label{eq:nonlinear_sdp_proof}
\begin{aligned} 
\begin{bmatrix}
 Q(A) - A Q(A) A^\transpose  &  -A Q(A) \\
 -Q(A) A^\transpose  & -Q(A)
\end{bmatrix} -  \lambda \begin{bmatrix}
 \gamma^2 I &  0 \\
 0 & -I
\end{bmatrix} & \succ 0 \quad \forall A \in U_{k,A}\\
1 - h_i^\transpose  Q(A) h_i & > 0 \quad \forall A \in U_{k,A}\quad i = 1, \dots, r \\
\begin{bmatrix}
 Q(A) - \beta I &  0 \\
 0 & \frac{1}{2} 
\end{bmatrix} & \succ 0 \quad \forall A \in U_{k,A}.
\end{aligned}
\end{equation}
Since $U_{0,A}$ is a bounded polyhedron, the set of inequalities that define it (i.e., $\{A \rightarrow v_j - \Tr(V_j^\transpose  A)\}_{j=1}^s$) satisfy the Archimedian property.
Consider the following set of polynomials:
\[
\mathcal{G} \defn \{A \rightarrow v_j - \Tr(V_j^\transpose  A)\}_{j=1}^s \cup \bigcup_{\ell=1}^k \bigcup_{t=1}^{n_{\ell}} \{A \rightarrow \gamma^2\norm{y_{t-1,\ell}}^2 - \norm{A y_{t-1,\ell} - y_{t,\ell}}^2\}.
\]
As a superset of a set satisfying the Archimedian property, this set also satisfies the Archimedian property.
By applying \theoremref{thm:putinar} and \theoremref{thm:matrix psatz} to \eqref{eq:nonlinear_sdp_proof}, we can \newnewstuff{conclude the existence of} SOS matrices, $M_j,M_{t\ell},\hat{M}_j,\hat{M}_{t\ell}$, and SOS polynomials, $\sigma_{ij},\sigma_{it\ell}$ of some degree $d$, and a positive scalar $\varepsilon$ satisfying \newnewstuff{\eqref{subeq:stabil n}, \eqref{subeq:polar dual n}, \eqref{subeq:strict n}, and} the following:
\[
\begin{aligned}
\begin{bmatrix}
 Q(A) - \beta I &  0 \\
 0 & \frac{1}{2} 
\end{bmatrix} &= \hat{M}_0(A) + \sum_{j=1}^s \hat{M}_j(A) (v_j - \Tr(V_j^\transpose  A))  \\
&+ \sum_{\ell=1}^{k} \sum_{t=1}^{n_{\ell}} \hat{M}_{t \ell}(A) (\gamma^2\norm{y_{t-1,\ell}}^2 - \norm{A y_{t-1,\ell} - y_{t,\ell}}^2)
\quad \forall A \in \Rnn.
\end{aligned}
\]
Note that for any $x$ satisfying $\norm{x} \leq \sqrt{\frac{1}{2\beta}}$, we have that $\begin{bmatrix}
 \beta I &  x \\
 x^\transpose  & \frac{1}{2}
\end{bmatrix} \succeq 0$.
Therefore, \[A \mapsto \hat{M}_0(A) + \begin{bmatrix}
 \beta I &  x \\
 x^\transpose  & \frac{1}{2}
\end{bmatrix}\] is still an SOS matrix of the same degree as $\hat{M}_0(A)$.
Hence, for any $x$ satisfying $\norm{x} \leq \sqrt{\frac{1}{2\beta}}$,
we have that the tuple 
\[
    \left( x,Q,M_j,M_{t\ell},\hat{M}_0(A) + \begin{bmatrix}
 \beta I &  x \\
 x^\transpose  & \frac{1}{2}
\end{bmatrix},\hat{M}_j,\hat{M}_{t\ell},\sigma_{ij},\sigma_{it\ell},\varepsilon,\lambda \right)
\]
is feasible to \eqref{eq:inf_nonlinear_sdp}.
\end{proof}
\subsection{Numerical Example}\label{sec:nonlinear_example}
We present a numerical example with $n=2$.
Here we take $S$ and $U_{0,A}$ to be the same as $S$ and $U_0$ in \sectionref{sec:infinite_example_one}.
We solve the semidefinite program in \eqref{eq:inf_nonlinear_sdp} with degree $d=4$ (the program with $d=2$ is infeasible).
In Figure~\ref{fig:inf_nonlinear}, we plot the safety region $S$, and our semidefinite programming based inner approximations $\tilde{S}_{0,4}^{\infty}$ of the infinite-step safe set $S^\infty_0$ for $\gamma=0,0.02,0.04,0.06$.
\newnewstuff{We also plot a set $\bar{S}_0^\infty$, which is the same outer approximation of $S_0^\infty$ as in \sectionref{sec:infinite_example_one}.}
Note that $\bar{S}_0^\infty$ is an outer approximation of $S^\infty_0$ for any value of $\gamma$.
\begin{figure}
\label{fig:inf_nonlinear}
\centering
\includegraphics[width=.5\textwidth]{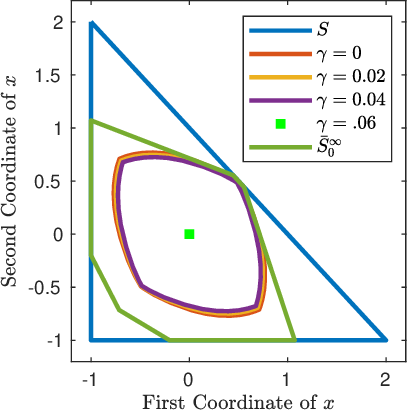}
\caption{The numerical example in Section~\ref{sec:nonlinear_example}: the safety set $S$, the sets $\tilde{S}_{0,4}^{\infty}$ for four different values of $\gamma$, and the set $\bar{S}^\infty_0$, which is an outer approximation to $S^\infty_0$ for any value of $\gamma$.}
\end{figure}

For $\gamma = 0.06$, our semidefinite program is infeasible and therefore we can only certify that the origin is infinite-step safe.
This is intended behavior since for $\gamma=0.06$, the true infinite-step safe set is just the origin.
To see why, observe that if $\trueA = \begin{bmatrix}
 0.5 &  0.45 \\
 0.45 & 0.5 
\end{bmatrix} \in U_{0,A},$ and $\trueg(x) = 0.055*x \in U_{0,g}$,
then we have $\truef(x) = \begin{bmatrix}
 0.555 &  0.45 \\
 0.45 & 0.555 
\end{bmatrix} x$ which is unstable since $\rho \left( \begin{bmatrix}
 0.555 &  0.45 \\
 0.45 & 0.555 
\end{bmatrix} \right) > 1$.
This means that the true infinite-step safe set is not full-dimensional (see Preposition~\ref{prop:U0bounded.MarginallyStable}).
By slightly perturbing $\trueg$ within $U_{0,g}$, we can obtain another valid unstable linear map $\hat{f}$ whose lower-dimensional stable subspace is different than that of $\truef$. Therefore, when $\gamma=0.06,$ the true infinite-step safe set is indeed just the origin.

For $\gamma =0.02$ or $0.04$ for example, and for any nonlinear system of the type \eqref{eq:nonlinear_model_inf}, with $\trueA \in~U_{0,A}$ and $\trueg \in U_{0,g}$, we can choose initialization points within our full-dimensional sets $\tilde{S}_{0,4}^{\infty}$ and safely observe their trajectories.
Having safely collected trajectory data, following the same exact procedure as in~\sectionref{sec:nonlinear example}, we can narrow the uncertainty on the linear part of the dynamics and use semidefinite programming to fit a polynomial map to the nonlinear part of the dynamics in such a way that the information in $U_{0,g}$ is respected and the error on the observations is minimized.}
\versionii{\section{Safe Learning with Specialized Side Information}
\label{sec:side info}
In previous sections, the initial information we assumed on the matrix $\trueA\in\Rnn$ governing the linear part of an unknown dynamical system was in terms of membership to an initial uncertainty set $U_0\subset  \Rnn$ which took the form of a polyhedron or an ellipsoid. Such uncertainty sets already capture natural side information such as being close to a nominal matrix in $1,2,\infty$ norm, having a banded structure, or being sparse with a known support. In this section, we give three examples of more specialized side information for which we can still \emph{exactly} characterize the $T$-step safe set of a linear system for $T=1$ (or higher in special cases) as the feasible set of a tractable conic program. Extensions of this research direction to other types of side information, different time horizons, and nonlinear systems is left for future research.


Throughout this section, we work with a polyhedral safety region $S \subset \R^n$ given in the form of \eqref{eq:S polyhedron}, and for simplicity, a linear objective function $c^\transpose  x$ representing initialization cost.
Our goal is to provide a tractable reformulation of the following optimization problem
\begin{equation}
\label{eq:structure problem}
\begin{aligned}
\min_{x \in \Rn} \quad & c^\transpose  x\\
\textrm{s.t.} \quad & x \in S \\
& Ax \in S \quad \forall A\in U_0,
\end{aligned}
\end{equation}
for three different classes of sets $U_0$.

\subsection{Sparse Matrices with Unknown Support}
Suppose that we know that the matrix $\trueA$ governing the linear dynamics in \eqref{eq:linear dynamics} has bounded entries of which only a limited number are nonzero.
We can then represent our initial uncertainty set as
\begin{equation}\label{eq:sparse U0}
U_0 = \{A \in \Rnn \mid \|A\|_0 \leq k, \|A\|_\infty \leq M \},
\end{equation}
where $k \in \mathbb{N}$ and $M \geq 0$ are given constants and $\|A\|_0$ (resp. $\|A\|_\infty$) denotes the number of nonzeros (resp. the largest entry in absolute value) of the matrix $A$.\footnote{Note that merely assuming that $\|A\|_0 \leq k$ cannot lead to safe learning. Indeed, if the safety region $S$ is compact, no nonzero point can be one-step safe with regards to this information even when $k$=1.}
In the following theorem, we establish that problem \eqref{eq:structure problem} has an exact linear programming reformulation.
We introduce auxiliary variables $\eta^{+(i)},\eta^{-(i)},\beta^{(i)} \in \Rnn$, and $\alpha^{(i)} \in \R$ for $i=1,\dots,r$.
\begin{theorem}\label{thm:sparse theorem}
The feasible set of problem \eqref{eq:structure problem} with $U_0$ as in \eqref{eq:sparse U0} is the projection to $x$-space of the feasible set of the following linear program:
\begin{equation}\label{eq:sparse LP}
\begin{aligned}
\min_{x,\eta^{+(i)},\eta^{-(i)},\beta^{(i)},\alpha^{(i)}} \quad & c^\transpose  x\\
\textrm{s.t.} \quad & h_i^\transpose  x \leq b_i \quad i = 1, \dots ,r \\
& M \Tr(J\beta^{(i)}) + Mk\alpha^{(i)} \leq b_i \quad i = 1, \dots ,r\\
& xh_i^\transpose = \eta^{+(i)} - \eta^{-(i)} \quad i = 1, \dots ,r\\
& \eta^{+(i)} + \eta^{-(i)} = \beta^{(i)} + \alpha^{(i)} J \quad i = 1, \dots ,r\\
& \eta^{+(i)},\eta^{-(i)},\beta^{(i)} \geq 0, \: \alpha^{(i)} \geq 0\quad i=1,\dots,r,
\end{aligned}
\end{equation}
where $J \in \Rnn$ is the matrix of all ones.
In particular, the optimal values of \eqref{eq:structure problem} and \eqref{eq:sparse LP} are the same and the optimal solutions of \eqref{eq:structure problem} are the optimal solutions of \eqref{eq:sparse LP} projected to $x$-space.
\end{theorem}
Before we prove this theorem, we characterize the convex hull of the set $U_0$ with the following standard lemma.
We recall our notation $\text{conv}(\Omega)$ to denote the convex hull of a set $\Omega \subseteq \Rn$.
\begin{lemma}\label{lem:sparse hull}
For all $n,k \in \mathbb{N}$ and all $M \geq 0$, we have:
\[
\textrm{\textup{conv}}(\{x \in \Rn \mid \|x\|_0 \leq k, \|x\|_\infty \leq M \}) = \{x \in \Rn \mid \|x\|_1 \leq Mk, \|x\|_\infty \leq M \}.
\]
\end{lemma}
\begin{proof}[Proof of \theoremref{thm:sparse theorem}]
We first write \eqref{eq:structure problem} as the bilevel program:
\begin{equation}\label{eq:sparse bilevel}
\begin{aligned}
\min_{x} \quad & c^\transpose  x\\
\textrm{s.t.} \quad & h_i^\transpose  x \leq b_i \quad i = 1, \dots ,r \\
& \begin{bmatrix} \max_A \quad & h_i^\transpose  A x  \\ 
\textrm{s.t.} \quad & A \in U_0 \end{bmatrix} \leq b_i \quad i = 1, \dots ,r.
\end{aligned}
\end{equation}
Observe that in the inner problems, the objective function $A \mapsto h_i^\transpose A x$ is a linear function of the variable $A$.
From this and \lemmaref{lem:sparse hull}, we have
\[
\begin{bmatrix} \max_A \quad & h_i^\transpose  A x  \\ 
\textrm{s.t.} \quad & A \in U_0 \end{bmatrix} 
= \begin{bmatrix} \max_A \quad & h_i^\transpose  A x  \\ 
\textrm{s.t.} \quad & A \in \text{conv}(U_0) \end{bmatrix}
= \begin{bmatrix} \max_A \quad & h_i^\transpose  A x  \\ 
\textrm{s.t.} \quad & \|A\|_1 \leq Mk \\
& \|A\|_\infty \leq M \end{bmatrix}.
\]
Introducing a new variable $\bar{A}\in \Rnn$, we rewrite this latter problem as a linear program:
\[
\begin{bmatrix} \max_A \quad & h_i^\transpose  A x  \\ 
\textrm{s.t.} \quad & \|A\|_1 \leq Mk \\
& \|A\|_\infty \leq M \end{bmatrix}
=\begin{bmatrix} \max_{A,\bar{A}} \quad & h_i^\transpose  A x  \\ 
\textrm{s.t.} \quad & -\bar{A} \leq A \leq \bar{A}\\
& \Tr(J\bar{A}) \leq Mk \\
& \bar{A} \leq MJ \end{bmatrix}.
\]
We proceed by taking the dual of the inner problems, treating the $x$ variable as fixed.
By introducing dual variables $\eta^{+(i)},\eta^{-(i)},\beta^{(i)} \in \Rnn$, and $\alpha^{(i)} \in \R$ for $i=1,\dots,r$, and by invoking strong duality of linear programming, we have
\begin{equation}\label{eq:sparse duality}
\begin{bmatrix} \max_{A,\bar{A}} \quad & h_i^\transpose  A x  \\ 
\textrm{s.t.} \quad & -\bar{A} \leq A \leq \bar{A}\\
& \Tr(J\bar{A}) \leq Mk \\
& \bar{A} \leq MJ \end{bmatrix}
= \begin{bmatrix} \min\limits_{\eta^{+(i)},\eta^{-(i)},\beta^{(i)},\alpha^{(i)}} \quad & M \Tr(J\beta^{(i)}) + Mk\alpha^{(i)}  \\ 
\textrm{s.t.} \quad & xh_i^\transpose = \eta^{+(i)} - \eta^{-(i)}\\
& \eta^{+(i)} + \eta^{-(i)} = \beta^{(i)} + \alpha^{(i)} J \\
& \eta^{+(i)},\eta^{-(i)},\beta^{(i)} \geq 0, \: \alpha^{(i)} \geq 0 \\
\end{bmatrix}
\end{equation}
for $i=1,\dots,r$.
Thus by replacing the inner problems of \eqref{eq:sparse bilevel} with the right-hand side of \eqref{eq:sparse duality}, the min-max problem \eqref{eq:sparse bilevel} becomes a min-min problem.
This min-min problem can be combined into a single minimization problem and be written as problem \eqref{eq:sparse LP}.
Indeed, if $x$ is feasible to \eqref{eq:sparse bilevel}, for that fixed $x$ and for each $i$, there exist values of $\eta^{+(i)},\eta^{-(i)},\beta^{(i)},\alpha^{(i)}$ that attain the optimal value for \eqref{eq:sparse duality} and therefore the tuple $(x,\eta^{+(i)},\eta^{-(i)},\beta^{(i)},\alpha^{(i)})$ will be feasible to \eqref{eq:sparse LP}.
Conversely, if some $(x,\eta^{+(i)},\eta^{-(i)},\beta^{(i)},\alpha^{(i)})$ is feasible to \eqref{eq:sparse LP}, it follows that $x$ is feasible to \eqref{eq:sparse bilevel}. This is because for any fixed $x$ and for each $i$, the optimal value of the left-hand side of \eqref{eq:sparse duality} is bounded from above by the objective value of the right-hand side evaluated at any feasible $(x,\eta^{+(i)},\eta^{-(i)},\beta^{(i)},\alpha^{(i)})$.
\end{proof}
\subsection{Low-Rank Matrices}
Suppose that we know that the matrix $\trueA$ governing the linear dynamics in \eqref{eq:linear dynamics} has bounded spectral norm and is low rank.
We can then write the initial uncertainty set as
\begin{equation}\label{eq:lowrank U0}
U_0 = \{A \in \Rnn \mid \text{rk}(A) \leq k, \|A\| \leq M \},
\end{equation}
where $k \in \mathbb{N}$ and $M \geq 0$ are given and $\text{rk}(A)$ (resp. $\|A\|$) denotes the rank (resp. spectral norm) of the matrix $A$.\footnote{Note that merely assuming that $A$ is low rank cannot lead to safe learning. Indeed, for any two vectors, $x,y \in \Rn$, with $x\neq 0$, the rank-one matrix $\frac{yx^\transpose}{\|x\|^2}$ takes $x$ to $y$; thus, a rank-one matrix can take any nonzero point to an unsafe point in just one step.}
In the following theorem, we establish that problem \eqref{eq:structure problem} has an exact semidefinite programming reformulation.
We introduce auxiliary variables $\eta^{(i)}_1, \eta^{(i)}_3 \in \Snn$, $\eta^{(i)}_2 \in \Rnn$ and $\alpha^{(i)} \in \R$ for $i=1,\dots,r$.
\begin{theorem}\label{thm: lowrank theorem}
The feasible set of problem \eqref{eq:structure problem} with $U_0$ as in \eqref{eq:lowrank U0} is the projection to $x$-space of the feasible set of the following semidefinite program:
\begin{equation}\label{eq:lowrank sdp}
\begin{aligned}
\min_{x,\eta^{(i)}_1,\eta^{(i)}_2,\eta^{(i)}_3,\alpha^{(i)}} \quad & c^\transpose  x\\
\textrm{s.t.} \quad & h_i^\transpose  x \leq b_i \quad i = 1, \dots ,r \\
& M\Tr(\eta^{(i)}_1) + \Tr(\eta_3^{(i)}) + \alpha^{(i)} Mk \leq b_i \quad i = 1, \dots ,r\\
& \begin{bmatrix}
    \alpha^{(i)} I & 2 \eta_2^{(i)} + xh_i^\transpose  \\
    2\eta_2^{(i)\transpose} + h_ix^\transpose & \alpha^{(i)} I
\end{bmatrix} \succeq 0 \quad i = 1, \dots ,r\\
& \begin{bmatrix}
    \eta_1^{(i)} & \eta_2^{(i)} \\
    \eta_2^{(i)\transpose} & \eta_3^{(i)}
\end{bmatrix} \succeq 0 \quad i = 1, \dots ,r,
\end{aligned}
\end{equation}
where $I$ denotes the $\nn$ identity matrix.
In particular, the optimal values of \eqref{eq:structure problem} and \eqref{eq:lowrank sdp} are the same and the optimal solutions of \eqref{eq:structure problem} are the optimal solutions of \eqref{eq:lowrank sdp} projected to $x$-space.
\end{theorem}
Before we prove this theorem, we recall a result that characterizes the convex hull of $U_0$.
We use the notation $\|A\|_*$ to denote the nuclear norm of the matrix $A$, i.e., the sum of its singular values.
\begin{lemma}[\cite{Hiriart-Urruty2012}]
\label{lem:lowrank hull}
For all $n,k \in \mathbb{N}$ and all $M \geq 0$, we have:
\[
\text{\textup{conv}}(\{A \in \Rnn \mid \text{\textup{rk}}(A) \leq k, \|A\| \leq M \}) = \{A \in \Rnn \mid \|A\|_* \leq Mk, \|A\| \leq M \}.
\]
\end{lemma}
\begin{proof}[Proof of \theoremref{thm: lowrank theorem}]
We first write \eqref{eq:structure problem} as the bilevel program:
\begin{equation}\label{eq:lowrank bilevel}
\begin{aligned}
\min_{x} \quad & c^\transpose  x\\
\textrm{s.t.} \quad & h_i^\transpose  x \leq b_i \quad i = 1, \dots ,r \\
& \begin{bmatrix} \max_A \quad & h_i^\transpose  A x  \\ 
\textrm{s.t.} \quad & A \in U_0 \end{bmatrix} \leq b_i \quad i = 1, \dots ,r.
\end{aligned}
\end{equation}
Observe that in the inner problems, the objective function $A \mapsto h_i^\transpose A x$ is a linear function of the variable $A$.
From this and \lemmaref{lem:lowrank hull}, we have
\[
\begin{bmatrix} \max_A \quad & h_i^\transpose  A x  \\ 
\textrm{s.t.} \quad & A \in U_0 \end{bmatrix} 
= \begin{bmatrix} \max_A \quad & h_i^\transpose  A x  \\ 
\textrm{s.t.} \quad & A \in \text{conv}(U_0) \end{bmatrix}
= \begin{bmatrix} \max_A \quad & h_i^\transpose  A x  \\ 
\textrm{s.t.} \quad & \|A\|_* \leq Mk \\
& \|A\| \leq M \end{bmatrix}.
\]
Introducing new variables $W_1,W_2 \in \Snn$, we rewrite this latter problem as a semidefinite program:
\[
\begin{bmatrix} \max_A \quad & h_i^\transpose  A x  \\ 
\textrm{s.t.} \quad & \|A\|_* \leq Mk \\
& \|A\| \leq M \end{bmatrix}
=\begin{bmatrix} \max\limits_{A,W_1,W_2} \quad & h_i^\transpose  A x  \\ 
\textrm{s.t.} \quad & \begin{bmatrix}
    W_1 & A \\
    A^\transpose & W_2
\end{bmatrix} \succeq 0 \\
& \frac{1}{2} (\Tr(W_1) + \Tr(W_2)) \leq Mk \\
&\begin{bmatrix}
    MI & A \\
    A^\transpose & I
\end{bmatrix} \succeq 0 \end{bmatrix}.
\]
We proceed by taking the dual of the inner problems, treating the $x$ variable as fixed.
By introducing dual variables $\eta^{(i)}_1, \eta^{(i)}_3 \in \Snn$, $\eta^{(i)}_2 \in \Rnn$, and $\alpha^{(i)} \in \R$ for $i=1,\dots,r$, we claim that
\begin{equation}\label{eq:lowrank duality}
\begin{bmatrix} \max\limits_{A,W_1,W_2}   & h_i^\transpose  A x  \\ 
\textrm{s.t.}   & \begin{bmatrix}
    W_1 & A \\
    A^\transpose & W_2
\end{bmatrix} \succeq 0 \\
& \frac{1}{2} (\Tr(W_1) + \Tr(W_2)) \leq Mk \\
&\begin{bmatrix}
    MI & A \\
    A^\transpose & I
\end{bmatrix} \succeq 0 \end{bmatrix}
= \begin{bmatrix} \min\limits_{\eta_1^{(i)},\eta_2^{(i)},\eta_3^{(i)},\alpha^{(i)}}   & M\Tr(\eta^{(i)}_1) + \Tr(\eta_3^{(i)}) + \alpha^{(i)} Mk  \\ 
\textrm{s.t.}   & \begin{bmatrix}
    \alpha^{(i)} I & 2 \eta_2^{(i)} + xh_i^\transpose  \\
    2\eta_2^{(i)\transpose} + h_ix^\transpose & \alpha^{(i)} I
\end{bmatrix} \succeq 0 \\
&\begin{bmatrix}
    \eta_1^{(i)} & \eta_2^{(i)} \\
    \eta_2^{(i)\transpose} & \eta_3^{(i)}
\end{bmatrix} \succeq 0 \end{bmatrix}
\end{equation}
for $i=1,\dots,r$.
By taking $A$ to be the zero matrix and $W_1$ and $W_2$ to be small enough positive multiples of the identity matrix, we observe that the problem on the left hand side of \eqref{eq:lowrank duality} is strictly feasible.
Similarly, by taking $\eta_2^{(i)}$ to be the zero matrix, $\eta_1^{(i)}$ and $\eta_3^{(i)}$ to be identity matrices, and $\alpha^{(i)}$ sufficiently large, we observe that the problem on the right hand side of \eqref{eq:lowrank duality} is strictly feasible.
Thus, the equality in \eqref{eq:lowrank duality} follows from strong duality of semidefinite programming (see, e.g., \cite[Theorem~6.3.4]{Lovasz2003}).
Thus, by replacing the inner problems of \eqref{eq:lowrank bilevel} with the right-hand side of \eqref{eq:lowrank duality}, the min-max problem \eqref{eq:lowrank bilevel} becomes a min-min problem.
This min-min problem can be combined into a single minimization problem and be written as problem \eqref{eq:lowrank sdp}.
Indeed, if $x$ is feasible to \eqref{eq:lowrank bilevel}, for that fixed $x$ and for each $i$, there exist values of $\eta^{(i)}_1,\eta^{(i)}_2,\eta^{(i)}_3,\alpha^{(i)}$ that attain the optimal value for \eqref{eq:lowrank duality} and therefore the tuple $(x,\eta^{(i)}_1,\eta^{(i)}_2,\eta^{(i)}_3,\alpha^{(i)})$ will be feasible to \eqref{eq:lowrank sdp}.
Conversely, if some $(x,\eta^{(i)}_1,\eta^{(i)}_2,\eta^{(i)}_3,\alpha^{(i)})$ is feasible to \eqref{eq:lowrank sdp}, it follows that $x$ is feasible to \eqref{eq:lowrank bilevel}. This is because for any fixed $x$ and for each $i$, the optimal value of the left-hand side of \eqref{eq:lowrank duality} is bounded from above by the objective value of the right-hand side evaluated at any feasible $(x,\eta^{(i)}_1,\eta^{(i)}_2,\eta^{(i)}_3,\alpha^{(i)})$.
\end{proof}

\begin{remark}
In the special case when $M\leq1$, the projection to $x$-space of the feasible set of the semidefinite program in \eqref{eq:lowrank sdp} is not only an exact characterization of the one-step safe set, but also of the $T$-step safe set for any $T$ (including $T=\infty$).
This follows from the fact that \[
A \in U_0 \Rightarrow A^t \in U_0 \quad \forall t,
\]
as the spectral norm is submultiplicative and $\text{\textup{rk}}(A^t) \leq \text{\textup{rk}}(A)$.
\end{remark}

\subsection{Permutation Matrices}
Suppose that we know that the matrix $\trueA$ governing the linear dynamics in \eqref{eq:linear dynamics} acts on a vector by permuting its entries.
In other words,
\begin{equation}\label{eq:permute U0}
U_0 = \{ A \in \Rnn \mid A\text{ is a permutation matrix} \},
\end{equation}
where we recall that a permutation matrix is a binary square matrix with each row and each column containing exactly one nonzero entry.
While there are $n!$ matrices in $U_0$, the following theorem establishes that problem \eqref{eq:structure problem} has an exact reformulation as a linear problem of polynomial size.
The proof of this theorem invokes the fact that $\text{conv}(U_0)$ is the set of $\nn$ doubly stochastic matrices \citep{Birkhoff46}.
We introduce auxiliary variables $u^{(i)},w^{(i)} \in \Rn$ for $i=1,\dots,r$.
\begin{theorem}[follows from Theorem 3.8 of \cite{ahmadi18rdo}]
The feasible set of problem \eqref{eq:structure problem} with $U_0$ as in \eqref{eq:permute U0} is the projection to $x$-space of the feasible set of the following linear program:
\begin{equation}\label{eq:permute LP}
\begin{aligned}
\min_{x,u^{(i)},w^{(i)}} \quad & c^\transpose  x\\
\textrm{s.t.} \quad &\mathbf{1}^\transpose u^{(i)} + \mathbf{1}^\transpose w^{(i)} \leq b_i \quad i = 1, \dots ,r\\
& u^{(i)}\mathbf{1}^\transpose + \mathbf{1}w^{(i)\transpose} \geq xh_i^\transpose \quad i = 1, \dots ,r,
\end{aligned}
\end{equation}
where $\mathbf{1}$ denotes the $n$-dimensional vector of all ones.
In particular, the optimal values of \eqref{eq:structure problem} and \eqref{eq:permute LP} are the same and the optimal solutions of \eqref{eq:structure problem} are the optimal solutions of \eqref{eq:permute LP} projected to $x$-space.
\end{theorem}

\begin{remark}
The projection to $x$-space of the feasible set \eqref{eq:permute LP} is not only an exact characterization of the one-step safe set, but also for the $T$-step safe set for any $T$ (including $T=\infty$).
This follows from the fact that \[
A \in U_0 \Rightarrow A^t \in U_0 \quad \forall t
\]
since the permutation matrices form a group closed under matrix multiplication.
\end{remark}

We remark that more generally, whenever a tractable conic programming based description of $\text{conv}(U_0)$ is available, one can invoke conic programming strong duality theory to get a tractable characterization of the one-step safe set.
}
\versionii{\section{Controlled Safe Learning}
\label{sec: controlled}
In this section, we extend our mathematical framework for safe learning to a setting where in addition to choosing the initialization point to the dynamics, we can also choose a control input. We present generalizations of our previous results to the case of linear control affine dynamics and time horizon $T=1$. Extensions to other settings is left for future work.


Consider the linear control affine dynamical system
\[
x_{t+1} = \trueA x_t + B_\star u_t,
\]
defined by the matrices $\trueA \in \Rnn$ and $B_\star \in \R^{n \times \bar{n}}$, where $x_t \in \Rn$ (resp. $u_t \in \R^{\bar{n}}$) is the state (resp. control input) at time $t$.
Suppose we have a safety region in $x$-space again called $S \subset \Rn$ and defined as the polyhedron in \eqref{eq:S polyhedron}. In addition, suppose we have a set of admissible controls $C \subseteq \R^{\bar{n}}$ defined as
\[
C = \{ u \in \R^{\bar{n}} \mid \bar{h}_i^\transpose u \leq \bar{b}_i \quad i=1,\dots,\bar{r} \}.
\]
The matrices $\trueA$ and $B_\star$ are unknown, but respectively belong to initial uncertainty sets $U_{0,A} \subset~\Rnn$ and $U_{0,B} \subset \R^{n \times \bar{n}}$. Let us again assume a polyhedral form for these sets:
\begin{align*}
U_{0,A} &=  \{ A \in \R^{n \times n} \mid  \Tr(V_j^\transpose  A) \leq v_j \quad j = 1, \dots ,s  \}  \\
U_{0,B} &= \{ B \in \R^{n \times \bar{n}} \mid \Tr(\bar{V}_j^\transpose  B) \leq \bar{v}_j \quad j = 1, \dots ,\bar{s} \},
\end{align*}
where $V_1,\dots,V_s \in \R^{n \times n}$, $\bar{V}_1,\dots,\bar{V}_{\bar{s}} \in \R^{n \times \bar{n}}$, and $v_1,\dots,v_s,\bar{v}_1,\dots,\bar{v}_{\bar{s}} \in \R$ are given.
Suppose that we have collected $k$ measurements from the true system in the form of tuples $(x_\ell,u_\ell,y_\ell)$ such that $y_\ell = \trueA x_\ell + B_\star u_\ell$ for $\ell = 1,\dots,k$.
Our updated uncertainty set $U_k$ is then the set of pairs of matrices which agree with our initial information and our $k$ measurements; i.e.,
\[
U_k = \left \{ (A,B) \in \Rnn \times \R^{n \times \bar{n}} \middle \vert A \in U_{0,A}, B\in U_{0,B}, y_\ell = A x_\ell + B u_\ell \quad \ell = 1,\dots,k \right \}.
\]
We may also define $U_{k,A}$ (resp. $U_{k,B}$) to be the projection to $A$-space (resp. $B$-space) of $U_k$. With this updated information, we can then define the \emph{one-step controlled safe set} as
\[
CS^1_k = \{x \in \Rn \mid \exists u\in C \text{ s.t. } \forall (A,B) \in U_k, Ax + Bu \in S\}.
\]
Now we can formalize what it means to safely learn in the controlled case, analogously to Definition~\ref{def:one-step}.
\begin{definition}[One-Step Controlled Safe Learning]\label{def: controlled safe learning}
We say that one-step controlled safe learning is possible if for some nonnegative integer $m$, we can sequentially choose vectors $x_k \in S$ and $u_k \in C$, for $k=1, \ldots, m$, and observe measurements $y_k = \trueA x_k + B_\star u_k$ such that:
\begin{enumerate}
    \item \textbf{(Safety)} for $k=1,\dots,m$, we have $A x_k + B u_k \in S \quad \forall (A,B) \in U_{k-1}$,
    \item \textbf{(Learning)} the sets of matrices $U_{m,A}$ and $U_{m,B}$ are singletons.
\end{enumerate}
\end{definition}
Note that by the safety requirement, each chosen initialization point $x_k$ must lie in $CS^1_{k-1}$. One can define $T$-step controlled safe learning analogously. Just as in the autonomous case, we note that controlled safe learning is possible for some $T$ if and only if it is possible for $T$=1. Therefore, the reader can note that our Corollary~\ref{cor:one-step controlled} below provides an efficient algorithm for checking the possibility of controlled safe learning.

We model initialization (resp. control) cost for simplicity with a linear function $c^\transpose x$ (resp. $\bar{c}^\transpose u$) for some given vector $c \in \Rn$ (resp. $\bar{c} \in \R^{\bar{n}}$).
The following problem finds the next cheapest pair of initialization point and control which ensure one-step safety:
\begin{equation}\label{eq:one-step controlled}
\begin{aligned}
\min_{x\in\R^n, u \in \R^{\bar{n}}} \quad & c^\transpose  x + \bar{c}^\transpose u\\
\textrm{s.t.} \quad & x \in S \\
& u \in C \\
& A x + Bu \in S \quad \forall (A,B) \in U_k. 
\end{aligned}
\end{equation}

Similarly as in Proposition~\ref{prop:one-step LP}, we can show that the feasible region of \eqref{eq:one-step controlled} is the projection to $(x,u)$-space of the feasible region of the following linear program, where we have added auxiliary variables $\mu_j^{(i)} \in \R$ for $ i=1,\dots,r$, $j=1,\dots,s$, and $\bar{\mu}_j^{(i)} \in \R$ for $ i=1,\dots,\bar{r}$, $j=1,\dots,\bar{s}$, and $\eta_\ell^{(i)} \in \R^n$ for $i=1,\dots,r$, $\ell=1,\dots,k$.

\begin{proposition}
The feasible set of problem \eqref{eq:one-step controlled} is the projection to $(x,u)$-space of the feasible set of the following linear program:
\begin{equation}\label{eq:one-step LP controlled}
\begin{aligned}
\min_{x,u,\mu,\bar{\mu},\eta} \quad & c^\transpose  x + \bar{c}^\transpose u\\
\textrm{s.t.} \quad & h_i^\transpose  x \leq b_i \quad i = 1, \dots ,r \\
& \bar{h}_i^\transpose  u \leq \bar{b}_i \quad i = 1, \dots ,\bar{r} \\
& \sum_{\ell=1}^{k} y_\ell^\transpose  \eta_\ell^{(i)} + \sum_{j=1}^{s} \mu_j^{(i)} v_j + \sum_{j=1}^{\bar{s}} \bar{\mu}_j^{(i)} \bar{v}_j \leq b_i \quad i = 1, \dots ,r \\
& x h_i^\transpose  = \sum_{\ell=1}^{k} x_\ell \eta_\ell^{(i)\transpose } + \sum_{j=1}^{s} \mu_j^{(i)} V_j^\transpose  \quad i = 1, \dots ,r \\
& u h_i^\transpose = \sum_{\ell=1}^{k} u_\ell \eta_\ell^{(i)\transpose } + \sum_{j=1}^{\bar{s}} \bar{\mu}_j^{(i)} \bar{V}_j^\transpose  \quad i = 1, \dots ,r \\
& \mu^{(i)} \geq 0 \quad i = 1, \dots ,r \\
& \bar{\mu}^{(i)} \geq 0 \quad i = 1, \dots ,r.
\end{aligned}
\end{equation}
In particular, the optimal values of \eqref{eq:one-step controlled} and \eqref{eq:one-step LP controlled} are the same and the optimal solutions of \eqref{eq:one-step controlled} are the optimal solutions of \eqref{eq:one-step LP controlled} projected to $(x,u)$-space.
\end{proposition}

We omit the proof of this proposition as it is very similar to the proof of Proposition~\ref{prop:one-step LP}, with its main ingredient being strong duality of linear programming. Our final two corollaries give the controlled analogues of Theorem~\ref{thm:one-step} and Corollary~\ref{cor:n steps}. The proofs are similar and hence omitted.

\begin{corollary}\label{cor:one-step controlled}
Given a safety region $S \subset \R^n$, a set of admissible controls $C \subseteq \R^{\bar{n}},$ and uncertainty sets $U_{0,A} \subset \R^{n \times n}$ and $U_{0,B} \subset \R^{n \times \bar{n}}$, one-step controlled safe learning (see Definition~\ref{def: controlled safe learning}) is possible if and only if \algorithmref{alg:one-step} with the input tuple
\[
\left(S \times C, \left \{ \begin{bmatrix}
    A & B \\
    0 & I
\end{bmatrix} \in \R^{(n+\bar{n}) \times (n + \bar{n})} \middle \vert A \in U_{0,A}, B \in U_{0,B} \right \}, c,\varepsilon \right)
\]
(with an arbitrary choice of $c \in \R^{n+\bar{n}}$ and $\varepsilon \in (0, 1]$) returns a matrix.

\end{corollary}

\begin{corollary}
Given a safety region $S \subset \R^n$, a set of admissable controls $C \subseteq \R^{\bar{n}},$ and uncertainty sets $U_{0,A} \subset \R^{n \times n}$ and $U_{0,B} \subset \R^{n \times \bar{n}}$, if one-step controlled safe learning is possible, then it is possible with at most $n + \bar{n}$ measurements. 
\end{corollary}}
\newnewstuff{\section{Future Research Directions}
\label{sec:conclusion}

\versionii{%
Besides extending the results of this paper to time horizons $T$ other than $1,2,\infty$, the results of \sectionref{sec:side info} to other types of side information, and the results of Section~\ref{sec: controlled} on controlled safe learning beyond the one-step linear control affine case, we list some potential directions for future research below:}
\begin{itemize}


\versionii{\item We addressed the settings of noisy measurements for time horizon $T=1,\infty,$ and disturbances in the dynamics for $T=1$.
Can we extend the treatment of noisy measurements to $T=2$?
Can we extend the treatment of disturbed dynamics to $T=2,\infty$? It would also be interesting to consider distributional assumptions on noise or disturbances and devise a statistical analysis of the resulting safe learning problem.}


\item Can one bound the suboptimality of our greedy online algorithm for minimizing the cost of safe learning against the idealized minimum cost of safe learning (cf.\ the paragraph above Eq.~\eqref{eq:greedy_safe} and e.g., problem \eqref{eq:onestep lower bound})?
How would this bound depend on the input parameters 
$S$, $U_0$, $T$, and the true dynamics $\truef$?
Furthermore, since it is not clear if any possible algorithm can achieve the idealized minimum cost of safe learning for every $f_\star \in U_0$,
one could instead consider comparing 
to the best valid algorithm for safe learning 
that achieves the lowest worst-case (minimax) cost over all $f_\star \in U_0$. How would the greedy algorithm compare to this minimax optimal algorithm?






\item In Sections \ref{sec:nonlinear} and \ref{sec:inf_nonlinear}, we studied systems which consist of a linear term plus a nonlinear term with bounded growth.
While this description is fairly general, 
further specialization to practical nonlinear systems, such as piecewise affine systems or systems parameterized with a known set of basis functions, could potentially allow one to safely recover
the nonlinear part of the dynamics.\footnote{
Under suitable growth assumptions, one could apply the results in \Cref{sec:nonlinear}
(resp.\ \Cref{sec:inf_nonlinear})
to derive inner approximations of the one (resp.\ infinite) step safe set of e.g.\ a piecewise affine system. However, by specializing the description of the system, 
it may be possible to derive less conservative inner approximations (or even exact characterizations).
}


\item In Sections \ref{sec:linear} and \ref{sec:inf_nonlinear}, we approximated infinite-step safe sets by deriving semidefinite programs whose size depend on the maximum allowed degree of certain SOS polynomials and matrices.
While we found small degrees to suffice empirically, it would be interesting to
bound, for some class of problem instances, the degree one must choose in order for the proposed approximation to the safe sets to be full-dimensional (assuming the true set is full-dimensional).


\versionii{\item 
While in this paper we focused on discrete-time systems, our mathematical framework for safe learning also applies to continuous-time systems. Extending our results to the continuous-time setting would broaden the scope of our work and capture problems in control theory or physics that are modelled with ordinary or partial differential equations. Unlike the discrete-time setting, where every time horizon is a different case to study, in continuous-time, there would essentially only be two cases: either the time horizon $T$ is finite or infinite. 
Another difference is that for many parametric classes of continuous-time dynamical systems such as linear or polynomial systems, an infinitesimally small noiseless observation of a single generic trajectory suffices for learning the system. Therefore, one must assume a discretized access to the trajectory for the sequential learning problem to be nontrivial. Despite these distinctions, the concepts behind our paradigm carry over to continuous time; e.g., safety sets would still grow and uncertainty sets would still shrink as more information is gathered. We suspect that in the continuous-time case, the analysis would likely be much more focused on the behavior of the system on the boundary of the safety region, since unsafe trajectories must exit the safety region through its boundary. We suspect that the literature on maximal invariant sets and peak estimation for continuous-time dynamical systems would be relevant to extending our algorithms and theory to continuous time; see, e.g.,~\cite{blanchini_invariance_1999,nagumo_1942,korda14,BELL20103625,miller2023distance} and references therein.}


\end{itemize}
}

\acks{
\versionii{We would like to thank two anonymous referees whose comments have greatly improved our manuscript.}
AAA and AC were partially supported by the MURI award of the AFOSR, the DARPA Young Faculty Award, the
CAREER Award of the NSF, the Google Faculty Award, the Innovation Award of the School of Engineering and Applied Sciences at Princeton University, and the Sloan Fellowship.}

\bibliography{paper,citations}

\appendix

\newstuff{\section{(Omitted Proofs)}
\label{sec:appendix}

\subsection{Proof of Lemma~\ref{lem:basis}}

\myproof{
We form the desired basis $\{e_i\}$ iteratively and with an inductive argument.
Let $e_1$ be any nonzero vector in $P$ (existence of such a vector can be checked by the argument in the proof of \lemmaref{lem:uniqueness}); if there is no such vector, we return the empty set.
Let $\{e_1,\dots,e_k\}$ be a linearly independent set in $P$.
We will either find an additional linearly independent vector $e_{k+1} \in P$, or show that the dimension of the span of $P$ is $k$.
Let $x$, $x^+$, and $x^-$ be variables in $\R^n$, $y^+$ and $y^-$ be variables in $\R^p$, and $\lambda^+$ and $\lambda^-$ be variables in $\R$.
Consider the following linear programming feasibility problem:
\begin{equation}\label{eq:one-step proof}
\begin{aligned}
& e_i^\transpose x = 0 \quad i=1,\dots,k\\
& x = x^+ - x^-\\
& Ax^+ + By^+ \leq \lambda^+ c \\
& Ax^- + By^- \leq \lambda^- c \\
& \lambda^+ \geq 0 \\
& \lambda^- \geq 0.
\end{aligned}
\end{equation}
Let $F\subseteq \R^n$ be the projection to $x$-space of the feasible region of this problem.
We claim that $F = \{0\}$ if and only if the dimension of $\spanof{P}$ is $k$.
Moreover, if there is solution to \eqref{eq:one-step proof}
with $x \neq 0$,
then there is also a solution $(x, x^{\pm}, y^{\pm}, \lambda^{\pm})$ where $\lambda^+,\lambda^- \neq 0$.
In this case, either $\frac{x^+}{\lambda^+}$ or $\frac{x^+}{\lambda^+}$ can be taken as $e_{k+1}$.

Suppose first that the dimension $\spanof{P}$ is at least $k+1$; then there is a vector $\tilde{x}\in \spanof{P}$ that is linearly independent from $\{e_1,\dots,e_k\}$.
By subtracting the projection of $\tilde{x}$ to $\spanof{\{e_1,\dots,e_k\}}$, we will find a nonzero vector $x \in \spanof{P}$ that is orthogonal to the vectors $e_1,\dots,e_k$.
We claim this vector $x$ is feasible to \eqref{eq:one-step proof} for some choice of $(x^\pm,y^\pm,\lambda^\pm)$.
Indeed, since $x \in \spanof{P}$, then
\[ x = \sum_{j=1}^r \lambda_j x_j, \]
for some vectors $x_1,\dots,x_r \in P$ and some nonzero scalars $\lambda_1,\dots,\lambda_r$.
For each $j$, as $x_j \in P$, there exists a vector $y_j$ such that $Ax_j + By_j \leq c$.
Let $J$ denote the set of indices $j$ such that $\lambda_j >0$.
It is easy to check that the assignment
\begin{equation}
\begin{aligned}
(x^+,y^+,\lambda^+) &= (\sum_{j \in J} \lambda_j x_j,\sum_{j \in J} \lambda_j y_j,\sum_{j \in J} \lambda_j) \:,\\
(x^-,y^-,\lambda^-) &= (- \sum_{j \notin J} \lambda_j x_j,- \sum_{j \notin J} \lambda_j y_j,- \sum_{j \notin J} \lambda_j)
\end{aligned} 
\end{equation}
satisfies system \eqref{eq:one-step proof}.
Hence, we have shown that if $F=\{0\}$ then the dimension of $\spanof{P}$ is $k$.
\par
To see the converse implication, suppose $x \neq 0$, and that the tuple $(x,x^\pm,y^\pm,\lambda^\pm)$ is feasible to system \eqref{eq:one-step proof}.
Without loss of generality we assume $\lambda^\pm\geq 1$; if not, we replace the tuple with
\begin{align}
    (x, x^\pm + \hat{x},y^\pm + \hat{y},\lambda^\pm + 1), \label{eq:make_lambda_nonzero}
\end{align}
where $\hat{x}$ and $\hat{y}$ are any vectors satisfying $A \hat{x} + B \hat{y} \leq c$.
Then, since $A \frac{x^+}{\lambda^+} + B \frac{y^+}{\lambda^+} \leq c$, the vector $\frac{x^+}{\lambda^+} \in P$.
By the same argument, $\frac{x^-}{\lambda^-} \in P$.
It follows from the orthogonality constraint of \eqref{eq:one-step proof} that at least one of the vectors $\frac{x^+}{\lambda^+}$ and $\frac{x^-}{\lambda^-}$ is linearly independent from $\{e_1,\dots,e_k\}$ and can be taken as $e_{k+1}$, also proving that the dimension of $\spanof{P}$ is at least $k+1$.
\par
Note that the condition $F=\{0\}$ can be checked by solving $2n$ linear programs (cf.\ the proof of \lemmaref{lem:uniqueness}); if $F\neq \{0\}$, then at least one of these $2n$ linear programs will return a tuple $(x,x^\pm,y^\pm,\lambda^\pm)$ where $x \neq 0$.
We then transform this tuple via \eqref{eq:make_lambda_nonzero}
to ensure that both $\lambda^+, \lambda^- \neq 0$
(we can take $\hat{x} = e_1$ and $\hat{y}$ to be any vector
such that $A e_1 + B \hat{y} \leq c$).
Since we cannot have more than $n$ linearly independent vectors in $\spanof{P}$, this procedure is repeated at most $n$ times.
}

\subsection{Proof of Lemma~\ref{lem:sdp_convex_hull}}
\begin{proof}
Let $e_i \in \R^n$ be the $i$-th canonical basis vector.
We construct $2n$ points
$\{ x_1^{\pm}, \dots, x_n^{\pm} \}$
using the following iterative procedure.

To construct $x_1^{\pm}$, 
we first solve
\eqref{two-step}
with $c =  \pm e_1$
and set $x_1^{+}, x_1^{-1}$ to be optimal
solutions for $+e_1, -e_1$, respectively.
Now to construct $x_{k+1}^{\pm}$
given $x_1^\pm,\dots,x_k^\pm$, 
we solve \eqref{two-step} with $c =  \pm e_{k+1}$ and with the additional constraints that $e_i^\transpose  x = \frac{e_i^\transpose  x_i^+ + e_i^\transpose  x_i^-}{2}$ for each $i=1,\dots,k$; call the resulting optimal points $x_{k+1}^\pm$.
By \Cref{thm:two-step}, 
every $x_k^{\pm}$, \newnewstuff{for $k=1, \dots, n$} is the solution to a 
semidefinite program.

We now prove that $\mathrm{conv}(x_1^{\pm}, \dots, x_n^{\pm})$ is a full-dimensional subset of $S_0^2$.
For a vector $x \in \Rn$ and a positive scalar $r$, let $B(x,r)$ represent the %
\newnewstuff{closed $\ell_2$ }%
ball centered at $x$ of radius $r$.
Let $x_0$ and $r_0$ be such that $B(x_0,r_0) \subseteq S_0^2$;
such a point exists by the assumption
that $S_0^2$ is full-dimensional.

We first show by induction that for each $k=0,\dots,n$, there exist some $x_k$ and $r_k > 0$ such that $B(x_k,r_k) \subseteq S_0^2$ and $e_i^\transpose  x_k = \frac{e_i^\transpose  x_i^+ + e_i^\transpose  x_i^-}{2}$ for each $i=1,\dots,k$.
The base case $k=0$ holds by assumption.
Assume for $k < n$ we have such an $x_k$ and $r_k > 0$, and let us show the corresponding statement for $k+1$.
By the properties assumed of $x_k$ and $r_k$, and by the definition of $x_{k+1}^\pm$, we have $e_{k+1}^\transpose x_{k+1}^+ \leq e_{k+1}^\transpose x_k - r_k$ and $e_{k+1}^\transpose x_{k+1}^- \geq e_{k+1}^\transpose x_k + r_k$.
Therefore, we have $e_{k+1}^\transpose x_{k+1}^+ < e_{k+1}^\transpose x_{k+1}^-$.
Assume without loss of generality that $\frac{e_{k+1}^\transpose  x_{k+1}^+ + e_{k+1}^\transpose  x_{k+1}^-}{2} \leq e_{k+1}^\transpose x_k$ \newnewstuff{(if the inequality is reversed, swap $x_{k+1}^+$ with $x_{k+1}^{-}$)}.
Let \newnewstuff{$\lambda \in [0,1)$} be such that
\[
\frac{e_{k+1}^\transpose  x_{k+1}^+ + e_{k+1}^\transpose  x_{k+1}^-}{2} = \lambda e_{k+1}^\transpose x_{k+1}^+ + (1-\lambda) e_{k+1}^\transpose x_k.
\]
Now we define $x_{k+1} \defn \lambda x_{k+1}^+ + (1-\lambda) x_k$ and \newnewstuff{$r_{k+1} = (1-\lambda) r_k$}.
It is clear by this definition that $x_{k+1}$ satisfies the constraints $e_i^\transpose  x_{k+1} = \frac{e_i^\transpose  x_i^+ + e_i^\transpose  x_i^-}{2}$ for each $i=1,\dots,k$ since it is a convex combination of the vectors $x_{k+1}^+$ and $x_k$ which also satisfy those constraints.
It is also clear by the choice of $\lambda$ that $x_{k+1}$ satisfies $\frac{e_{k+1}^\transpose  x_{k+1}^+ + e_{k+1}^\transpose  x_{k+1}^-}{2} = e_{k+1}^\transpose  x_{k+1}$.
Since $S_0^2$ is a convex set and since we have $x_{k+1}^+ \in S_0^2$ and $B(x_k,r_k) \subseteq S_0^2$, it follows that $S_0^2$ contains the convex hull of $x_{k+1}^{+}$ and $B(x_k, r_k)$.
\newnewstuff{Observe that $B(x_{k+1},r_{k+1})$ lies inside this convex hull and therefore also $S_0^2$, by the following Minkowski arithmetic:
\begin{align*}
\text{conv}(\{x_{k+1}^{+}\} , B(x_k, r_k)) &\supseteq  \lambda x_{k+1}^{+} + (1-\lambda) B(x_k, r_k) \\
&= \lambda x_{k+1}^{+} + B((1-\lambda)x_k,(1-\lambda) r_k) \\
&= B(\lambda x_{k+1}^{+} + (1-\lambda)x_k,(1-\lambda) r_k) \\
&= B(x_{k+1},r_{k+1}).
\end{align*}}
This establishes the statement for $k+1$ \newnewstuff{and concludes the inductive argument}.

Thus, for each $k=0,\dots,n$, there exist some $x_k$ and $r_k > 0$ such that $B(x_k,r_k) \subseteq S_0^2$ and $e_i^\transpose  x_k = \frac{e_i^\transpose  x_i^+ + e_i^\transpose  x_i^-}{2}$ for each $i=1,\dots,k$.
It now follows that $e_{k}^\transpose x_{k}^+ < e_{k}^\transpose x_{k}^-$ for each $k=1,\dots,n$.
Let $T$ be the $n\times n$ matrix whose $i$-th column is $x_i^+ - x_i^-$.
Then $T_{ii} = e_{i}^\transpose x_{i}^+ - e_{i}^\transpose x_{i}^- \neq 0$, and for $k>i$ we have $T_{ik}=0$ since $e_{i}^\transpose x_{k}^+ = e_{i}^\transpose x_{k}^- = \frac{e_{i}^\transpose  x_{i}^+ + e_{i}^\transpose  x_{i}^-}{2}$.
Therefore $T$ is invertible, because it is a lower triangular with nonzero entries on its diagonal.
Define $$x_c \defn \sum_{i=1}^n \frac{1}{2n}(x_{i}^+ + x_{i}^-)$$
and the set 
\[
    M \defn \{x_c + u \mid \|T^{-1} u\|_{\infty} \leq \frac{1}{2n} \}.
\] 
Observe that $M$ is full-dimensional.
To conclude the proof, we show that $M \subseteq \text{conv}(\{x_i^\pm\}_{i=1}^n)$.
Indeed, for any $x_c + u \in M$,
\begin{align*}
x_c + u &= \sum_{i=1}^n \frac{1}{2n}(x_{i}^+ + x_{i}^-) + T \sum_{i=1}^n e_i e_i^\transpose  T^{-1} u \\
&= \sum_{i=1}^n (\frac{1}{2n} + e_i^\transpose  T^{-1} u)x_{i}^+ + (\frac{1}{2n} - e_i^\transpose  T^{-1} u) x_{i}^- \\
&\in \text{conv}(\{x_i^{\pm}\}_{i=1}^{n}).
\end{align*}

\newnewstuff{Note that if $S_0^2$ is not full-dimensional, the points $\{x_i^{\pm}\}_{i=1}^{n}$ supplied by this algorithm would satisfy $e_{k}^\transpose x_{k}^+ = e_{k}^\transpose x_{k}^-$ for at least one $k=1,\dots,n$.}
\end{proof}

\subsection{Proof of Proposition~\ref{prop:perturbations_avoid_nullsets}}

\begin{proof}
We proceed by induction.
For the base case, let $N_1$ be an arbitrary $\lambda^n$ null-set.
We clearly have $\Pr(z_1 \in N_1) = \Pr(\delta_1 \in N_1) = 0$ by
the absolute continuity of the law of $\delta_1$ w.r.t.\ $\lambda^n$.

For the inductive hypothesis, let $k$ be an integer satisfying
$1 \leq k \leq m-1$.
Suppose that for every $\lambda^{nk}$ null-set $N_k$, we have
$\Pr((z_1, \dots, z_k) \in N_k) = 0$.
Now let $N_{k+1}$ be an arbitrary $\lambda^{n(k+1)}$ null-set.
For any $\bar{z}_{1:k} \in \R^{n \times k}$, define the slice
$N_{k+1}(\bar{z}_{1:k}) = \{ \bar{z}_{k+1} \in \R^n \mid (\bar{z}_{1:k}, \bar{z}_{k+1}) \in N_{k+1} \}$.
Next, define the set:
\begin{align*}
    N^0_{k+1} = \{ \bar{z}_{1:k} \in \R^{n \times k} \mid N_{k+1}(\bar{z}_{1:k}) \text{ is $\lambda^n$-measurable and } \lambda^n(N_{k+1}(\bar{z}_{1:k})) = 0 \}.
\end{align*}
By the Fubini-Tonelli theorem for complete measures (see e.g. Theorem 2.39 of~\cite{folland}),
$N^0_{k+1}$ is $\lambda^{nk}$-measurable and $\lambda^{nk}((N^0_{k+1})^c) = 0$.
Abbreviating $z_{1:k} = (z_1, \dots, z_k)$,
\begin{align*}
    \Pr((z_1, \dots, z_{k+1}) \in N_{k+1}) &= \Pr( \{ (z_{1:k}, z_{k+1}) \in N_{k+1} \} \cap \{ z_{1:k} \in N^0_{k+1} \} ) \\
    &\qquad + \Pr( \{ (z_{1:k}, z_{k+1}) \in N_{k+1} \} \cap \{ z_{1:k} \in (N^0_{k+1})^c \}  ) \\
    &\leq \Pr( \{ (z_{1:k}, z_{k+1}) \in N_{k+1} \} \cap \{ z_{1:k} \in N^0_{k+1} \} ) + \Pr(z_{1:k} \in (N^0_{k+1})^c) \\
    &\stackrel{(a)}{=} \Pr( \{ (z_{1:k}, z_{k+1}) \in N_{k+1} \} \cap \{ z_{1:k} \in N^0_{k+1} \} ) \\
    &= \Pr( \{ z_{k+1} \in N_{k+1}(z_{1:k}) \} \cap \{ z_{1:k} \in N^0_{k+1} \} ) \\
    &= \Pr( \{ \delta_{k+1} \in N_{k+1}(z_{1:k}) - f_k(z_{1:k}) \} \cap \{ z_{1:k} \in N^0_{k+1} \} ) \\
    &\stackrel{(b)}{=} 0.
\end{align*}
Above, (a) follows by the inductive hypothesis and
the fact that $(N^0_{k+1})^c$ is a $\lambda^{nk}$ null-set.
Furthermore, (b) follows since
when $z_{1:k} \in N^0_{k+1}$, then 
$N_{k+1}(z_{1:k})$ is a $\lambda^n$ null-set, and hence
by the translation invariance of $\lambda^n$,
$N_{k+1}(z_{1:k}) - f_k(z_{1:k})$
is also a $\lambda^n$ null-set. Therefore,
by the absolute continuity of the law of $\delta_{k+1}$
w.r.t.\ $\lambda^n$
and the independence of $\delta_{k+1}$
from $\delta_1, \dots, \delta_k$,
\begin{align*}
\Pr( \delta_{k+1} \in N_{k+1}(z_{1:k}) - f_k(z_{1:k}) \mid z_{1:k} \in N^0_{k+1}) &= 
\Pr_{\delta_{k+1}}( \delta_{k+1} \in N_{k+1}(z_{1:k}) - f_k(z_{1:k}) ) = 0.
\end{align*}
\end{proof}

\subsection{Proof of Proposition~\ref{prop:krylov}}
\begin{proof}
It is sufficient to show that for 
each integer $k$ satisfying $k\geq n$,
\[
\left[ A x = \trueA x, A^2 x = \trueA^2 x, \dots, A^n x = \trueA^n x \right] \Rightarrow A^k x = \trueA^k x.
\]
Clearly this statement holds for $k=n$.
We now assume the statement holds for
some $k \geq n$ and show that it also holds for $k+1$.
By the Cayley-Hamilton theorem, we have
\[
\trueA^k \in \text{span}(I,\dots,\trueA^{n-1})
\]
and from this it follows that
\[
\trueA^kx \in \text{span}(x,\dots,\trueA^{n-1}x).
\]
Therefore, there exist scalars $\lambda_i$, $i=0,\dots,n-1$, such that $\trueA^k x= \sum_{i=0}^{n-1} \lambda_i \trueA^i x$.
Now we have:
\begin{align*}
    A^{k+1}x &= A A^k x \stackrel{(a)}{=} A \trueA^k x\\
    &= A \left( \sum_{i=0}^{n-1} \lambda_i \trueA^i x \right) = \sum_{i=0}^{n-1} \lambda_i A \trueA^i x  \stackrel{(b)}{=} \sum_{i=0}^{n-1} \lambda_i A^{i+1} x \stackrel{(c)}{=} \sum_{i=0}^{n-1} \lambda_i \trueA^{i+1} x\\
    &= \trueA \left( \sum_{i=0}^{n-1} \lambda_i \trueA^i x \right) = \trueA \trueA^k x = \trueA^{k+1} x,
\end{align*}
where (a) follows from the inductive hypothesis and (b) and (c) follow from the assumption that
$A^i x = \trueA^i x$ for $i=1, \dots, n$.
\end{proof}}

\end{document}